\documentclass[a4paper, reqno,11pt]{amsart}
\usepackage[utf8]{inputenc}
\usepackage[T1]{fontenc}
\usepackage[english]{babel}
\usepackage{todonotes}
\usepackage{babelbib}
\usepackage{amsmath}
\usepackage{amsfonts}
\usepackage{amssymb}
\usepackage{amsthm}
\usepackage{color}
\usepackage{amsthm}
\usepackage{graphicx}
\usepackage{hyperref}
\usepackage{textcomp}
\usepackage{array}
\usepackage{pgf}
\usepackage{tikz}
\usepackage{tikz-3dplot}
\usepackage{tikz-cd}
\usetikzlibrary{cd}
\usetikzlibrary{decorations.markings}
\usepackage{float}
\usetikzlibrary{arrows}
\usetikzlibrary{patterns}
\usetikzlibrary{matrix,arrows}
\usepackage{mathrsfs}
\usepackage{bbm,yhmath,stmaryrd}
\usepackage{geometry}
\newtheorem{defn}[equation]{Definition}

\newtheorem{theo}[equation]{Theorem}
\newtheorem{prop}[equation]{Proposition}
\newtheorem{lem}[equation]{Lemma}

\newtheorem{cor}[equation]{Corollary}
\newtheorem{ex}[equation]{Example}
\newtheorem{qu}[equation]{Question}

\newtheorem{rem}[equation]{Remark}

\numberwithin{equation}{section}

\newcommand{\ddbar}{\partial\bar{\partial}}
\newcommand \Z{\mathbb Z}

\newcommand \N{\mathbb N}
\newcommand \C{\mathbb C}
\newcommand \R{\mathbb R}
\newcommand \Q{\mathbb Q}
\newcommand \gO {\mathcal O}

\newcommand \M {\mathscr {M}}
\newcommand \Ld {\mathscr {L}}
\newcommand \D {\mathbb{D}}
\newcommand \de {\Delta}
\newcommand \U {\mathscr U}

\newcommand \CP {\mathbb P}
\newcommand \A {\mathbb A}

\newcommand \X {\mathscr{X}}

\newcommand \bS {\mathbb S}

\newcommand \fl {\longrightarrow}
\newcommand \eps {\varepsilon}

\DeclareMathOperator{\lcm}{lcm}
\DeclareMathOperator{\KE}{KE}
\DeclareMathOperator{\Sk}{Sk}

\DeclareMathOperator{\Proj}{Proj}
\DeclareMathOperator{\ord}{ord}

\DeclareMathOperator{\Ker}{Ker}

\DeclareMathOperator \GL {GL}
\DeclareMathOperator \Div {Div}

\DeclareMathOperator{\PSH}{PSH}
\DeclareMathOperator{\CPSH}{CPSH}
\DeclareMathOperator{\SH}{SH}
\DeclareMathOperator{\Ric}{Ric}

\DeclareMathOperator{\Sp}{Spec}

\DeclareMathOperator{\Sing}{Sing}
\DeclareMathOperator{\MA}{MA}

\DeclareMathOperator{\triv}{triv}
\DeclareMathOperator{\an}{an}

\DeclareMathOperator{\hyb}{hyb}

\DeclareMathOperator{\Conv}{Conv}

\DeclareMathOperator{\ev}{ev}
\DeclareMathOperator{\pr}{pr}
\DeclareMathOperator{\FS}{FS}
\DeclareMathOperator{\DFS}{DFS}
\DeclareMathOperator{\hol}{hol}

\DeclareMathOperator{\rf}{ref}

\DeclareMathOperator{\NA}{NA}

\DeclareMathOperator{\Frac}{Frac}

\DeclareMathOperator{\sch}{sch}

\DeclareMathOperator{\xdiv}{div}
\newtheorem*{thmA}{Theorem A}
\newtheorem*{thmB}{Theorem B}

\title{Global pluripotential theory on hybrid spaces}
\author{Léonard Pille-Schneider}
\begin{document}
\date{}
\nocite{*}
\maketitle
\begin{abstract}
Let $A$ be a Banach ring, and $X/A$ be a projective scheme of finite type, endowed with a semi-ample line bundle $L$. We define a class $\PSH(X,L)$ of plurisubharmonic metrics on $L$ on the Berkovich analytification $X^{\an}$ and prove various basic properties thereof.
\\We focus in particular on the case where $A$ is a hybrid ring of complex power series and $X/A$ is a smooth variety, so that $X^{\an}$ is the hybrid space associated to a degeneration $X$ of complex varieties over the punctured disk. We then prove that when $L$ is ample, any plurisubharmonic metric on $L$ with logarithmic growth at zero admits a canonical plurisubharmonic extension to the hybrid space $X^{\hyb}$. We also discuss the continuity of the family of Monge-Ampère measures associated to a continuous plurisubharmonic hybrid metric.
\\In the case where $X$ is a degeneration of canonically polarized manifolds, we prove that the canonical psh extension is continuous on $X^{\hyb}$ and describe it explicitly in terms of the canonical model (in the sense of MMP) of the degeneration.
\end{abstract}
\section*{Introduction}
The study of plurisubharmonic functions and positive currents on complex manifolds - referred to as \emph{pluripotential theory} - has proven itself over the past decades to be a central tool in complex Kähler and algebraic geometry. The heuristic idea that one should be able to develop a pluripotential theory on Berkovich analytic spaces over a non-archimedean field $K$, similar to the classical one over the complex numbers, is by now well-established, see for instance \cite{Thu2}, \cite{CLD} \cite{BFJ2}, \cite{BE}, \cite{BJv4} in chronological order. 
\\To be more precise, let $K$ be a complete non-archimedean field, and $X/K$ a proper algebraic variety. In this setting, Berkovich's theory of $K$-analytic spaces \cite{Ber} associates to the variety $X/K$ its Berkovich analytification $X^{\an}$, in a similar spirit to how one associates to a variety over $\C$ its complex analytification. One may then define in this setting a class $\PSH(X,L)$ of plurisubharmonic - or semi-positive - metrics on $L^{\an}$, whose properties mimic those of plurisubharmonic metrics in the complex case.
\\While the work of Chambert-Loir - Ducros \cite{CLD} provides a local definition of semi-positivity, we will adopt a global point of view throughout this paper, following the approach initiated by Boucksom-Favre-Jonsson in \cite{BFJ2} and further developped in \cite{BJtriv}, \cite{BJv4}. The basic idea is as follows: for $m \ge 1$, given a non-zero global section $s \in H^0(X, mL)$, the singular metric (using additive notation for metrics) $\phi = m^{-1} \log \lvert s \rvert$ should be plurisubharmonic, as follows in the complex world from the Lelong-Poincaré formula. It moreover follows from Demailly's seminal work \cite{Dem} on regularization of plurisubharmonic functions that on a smooth polarized complex variety $(X, L)$, the class $\PSH(X,L)$ is the smallest class of singular metrics containing the metrics of the form $\phi = m^{-1} \log \lvert s \rvert$ as above, and that is furthermore stable by addition of constants, finite maxima and decreasing limits (see thm. \ref{theo psh2}). It is thus natural to take this characterization as the definition of $\PSH(X, L)$ in the non-archimedean setting, which turns out to be consistent with the more local, Chambert-Loir - Ducros approach, by the results from \cite{BE}.
\\One of the upsides of Berkovich's construction of analytic spaces is that it works over more general bases than non-archimedean fields: given a Banach ring $(A, \lvert \cdot \rvert)$, one can define its Berkovich spectrum $\M(A)$, and for any scheme $X/A$ of finite type, a Berkovich analytic space $X^{\an} \xrightarrow{\pi} \M(A)$ equipped with a continuous structure map to $\M(A)$. Each point $x \in \M(A)$ has a residue field $\mathscr{H}(x)$ in a natural sense, which is a complete valued field, and the fiber $\pi^{-1}(x)$ of the structure map is naturally homeomorphic to Berkovich analytification of the base change $X_{\mathscr{H}(x)}$. One may thus view the $A$-analytic space $X^{\an}$ as the family of analytic spaces $( X^{\an}_{\mathscr{H}(x)} )_{x \in \M(A)}$ over different base fields, in a similar manner to which one views a scheme over a ring $A$ as a family of varieties over different base fields, parametrized by $\Sp A$. More general Berkovich analytic spaces over Banach rings were studied more extensively recently in \cite{LP}, where it is for instance proved that $A$-analytic spaces form a category in a natural way, under certain assumptions on $A$.
\\At this point, pluripotential theory on Berkovich spaces over a Banach ring remains a vastly unexplored territory. In this paper, assuming that the base ring $A$ is integral, we define a class $\PSH(X, L)$ of plurisubharmonic metrics on the analytification of a scheme $X/A$ of finite type, endowed with a semi-ample line bundle $L$. Roughly speaking, a singular plurisubharmonic metric $\phi \in \PSH(X, L)$ can be seen as a family of psh metrics $\phi_x \in \PSH(X_{\mathscr{H}(x)}, L_{\mathscr{H}(x)})$ on the fibers of the structure map, varying in a plurisubharmonic way with respect to $x \in \M(A)$ as well. Note however that it can happen that $\phi_x \equiv -\infty$ for some $x \in \M(A)$, as one would expect in the complex world. The class $\PSH(X,L)$ is defined following the global approach of \cite{BFJ2}, \cite{BE} : it is the smallest class of singular metrics on $L$ that contains metrics of the form $m^{-1} \log \lvert s \rvert$ whenever $s \in H^0(X, mL)$ is a non-zero global section, and is stable under addition of constants, finite maxima and decreasing limits.
\\Our main concern is the case where $A$ is a \emph{hybrid ring} (see section \ref{sec hyb} for the definitions) and $X \fl \D^*$ is a projective degeneration of complex manifolds - we will always assume the degeneration to be meromorphic at $t=0$, which means that we may view $X$ as a projective scheme over the field $K = \C((t))$ of complex Laurent series - so that $X^{\an} = X^{\hyb}$ is the associated hybrid space, as studied for instance in \cite{BJ}, \cite{Fav}. The hybrid space $X^{\hyb} \xrightarrow{\pi} \bar{\D}_r$ comes with a continuous structure map to the closed disk of radius $r \in (0,1)$, such that $\pi^{-1}(0) = X^{\an}_K$ is the Berkovich analytification of $X$ with respect to the non-archimedean $t$-adic absolute value on $K$, while $\pi^{-1}(\bar{\D}_r^*)$ can be naturally identified with the restriction of the degeneration $X$ to the closed punctured disk $\bar{\D}_r^*$ - up to rescaling the absolute value on the fiber $X_t$ by a factor $\frac{ \log r }{\log \lvert t \rvert}$. As a result, this provides a natural way to see the (suitably rescaled) complex manifolds $(X_t)_{t \in \D^*}$ degenerate to a non-archimedean analytic space $X^{\an}_K$ as $t\rightarrow0$. In this setting, if $L$ is a semi-ample line bundle on $X$, then a psh metric on $(X^{\hyb}, L^{\hyb})$ corresponds to the data of a family of psh metrics $\phi_t \in \PSH(X_t, L_t)$ varying in a subharmonic way with respect to $t$, together with a non-archimedean metric $\phi_0 \in \PSH(X^{\an}, L^{\an})$.
\\In the case where the line bundle $L$ is ample on $X$, given a plurisubharmonic metric $\phi$ on $L$ which has logarithmic growth at $t=0$ (see definition \ref{def log g}), one can associate to it a psh metric $\phi^{\NA} \in \PSH(X^{\an}, L^{\an})$ on the non-archimedean analytic space $X^{\an}$, encoding the generic Lelong numbers of $\phi$ along the centrals fibers of models of $(X, L)$ over the disk. In the case where $X = Y \times \D^*$ is a product and $\phi$ is an $\bS^1$-invariant metric on $p_1^*L$ for an ample line bundle $L$ on $Y$, one can view $\phi : \R_{\ge 0} \fl \PSH(Y, L)$ as a psh ray on $Y$, and the non-archimedean limit was defined by Berman-Boucksom-Jonsson \cite{BBJ} in the context of their proof of the Yau-Tian-Donaldson conjecture; while the general case was treated in \cite{Reb}. 
\\Our main result states that given a psh metric $\phi \in \PSH(X, L)$ with logarithmic growth at zero, then the associated non-archimedean metric $\phi^{\NA}$ induces a canonical semi-positive extension of $\phi$ to the hybrid space:
\begin{thmA} Let $(X, L) \fl \D^*$ be a polarized degeneration of complex manifolds, and $\phi \in \PSH(X, L)$ a psh metric with logarithmic growth at $t=0$. Then the singular metric $\phi^{\hyb}$ on $(X^{\hyb}, L^{\hyb})$ such that:
$$\phi^{\hyb}_0 = \phi^{\NA},$$
$$\phi^{\hyb}_{| X_t} = \phi_t$$
is plurisubharmonic.
\end{thmA}
Note that not all psh metrics on the hybrid space arise in this way: the point $0 \in \bar{\D}_r$ is pluripolar - non-negligible in the sense of hybrid pluripotential theory - so that psh hybrid metrics are not uniquely recovered by their restriction to the punctured disk. This subtlety however disappears when restricting our attention to continuous psh metrics, as $0 \in \bar{\D}_r$ has empty interior, so that every \emph{continuous} psh metric on $L^{\hyb}$ is of the form described in Theorem A.
\\We now move on to discuss Monge-Ampère measures in this context. We start with the case where $X/K$ is an $n$-dimensional projective variety over a complete valued field $K$, endowed with $n$ ample line bundles $(L_1,...,L_n)$. When $K = \C$, it follows from the work of Bedford-Taylor \cite{BT} that the complex Monge-Ampère operator:
$$ (\phi_1,..., \phi_n) \mapsto dd^c \phi_1 \wedge ... \wedge dd^c \phi_n,$$
defined \emph{a priori} on tuples of smooth, positive-definite metrics $\phi_i$ on $L_i$, extends to \emph{continuous} psh metrics. When the field $K$ is non-archimedean, it is possible to define in a similar way a measure-valued Monge-Ampère operator, denoted by $\MA(\phi_1,...,\phi_n)$, where $(\phi_1,...,\phi_n)$ is a tuple of continuous psh metrics on the $L_i$, this was done in \cite{CL}, \cite{BFJ1} when $K$ is discretely-valued of equicharacteristic zero, and extended to the general case in \cite{BE}. As a result, if $\phi$ is a continuous psh metric on a relatively polarized scheme $(X,L)/A$ over a Banach ring $A$, it induces a family of Monge-Ampère measures $(\MA(\phi_x))_{x \in \M(A)}$ on $X^{\an} \xrightarrow{\pi} \M(A)$, where $\phi_x$ is the restriction of $\phi$ to $X^{\an}_{\mathscr{H}(x)} \simeq \pi^{-1}(x)$, and $\MA(\phi_x) := \MA(\phi_x,..., \phi_x)$. Assuming that $X$ is flat over $A$, it is a natural question to wonder whether or not the family of Monge-Ampère measures $(\MA(\phi_x))_{x \in \M(A)}$ is weakly continuous with respect to $x \in \M(A)$.
\\In the case where $A$ is a hybrid ring, we compare our formalism with the set-up of Favre \cite{Fav} for continuous psh metrics on the hybrid space, which yields a continuity result for the family of Monge-Ampère measures on the hybrid space associated to a continuous psh hybrid metric (theorem \ref{theo cont MA}):
\begin{theo}{\cite[thm. 4.2]{Fav}} \label{theo Favre intro} Let $X \fl \D^*$ be a projective degeneration of complex manifolds, polarized by an ample line bundle $L$. If $\phi$ is a continuous plurisubharmonic metric on $(X^{\hyb}, L^{\hyb})$, then the family of Monge-Ampère measures $(\mu_t)_{t \in \bar{\D}_r}$ on $X^{\hyb}$, defined by:
$$\mu_t = \MA(\phi_t),$$
where $\phi_t = \phi_{| X_t}$ for $t \in \bar{\D}^*_r$ and $\phi_0 = \phi_{| X^{\an}_K}$, is weakly continuous.
\end{theo}
Finally, we apply our setup to refine our previous work \cite{PS} bearing on degenerations of canonically polarized manifolds $X \fl \D^*$. Writing $L = K_{X/\D^*}$, it follows from the classical Aubin-Yau theorem that each fiber $X_t$ admits a unique Kähler-Einstein metric $\phi_t$ on $L_t$, whose curvature form $\omega_t = dd^c \phi_t$ has negative constant Ricci curvature:
$$\Ric(\omega_t) = -\omega_t.$$
Moreover, it follows from the results of Schumacher \cite{Schu} that the family of Kähler-Einstein metrics $\phi =(\phi_t)_{t \in \D^*}$ also has positive curvature in the direction of the base and has logarithmic growth at $t=0$, so that $\phi \in \PSH(X, L)$.
\\The machinery of the Minimal Model Program (see section \ref{sec CV alg neg}) furthermore implies that after a finite base change, the family admits a unique canonical model $\X_c / \D$, which has ample relative canonical bundle $K_{\X_c/\D}$. Moreover, by the results of Song-Sturm-Wang \cite{Song}, \cite{SSW}, the Kähler-Einstein metrics converge in a natural sense to a unique Kähler-Einstein current $\omega_{KE, 0}$ on the special fiber $\X_{c, 0}$, and even though this current does not have bounded potentials, its singularities are milder than any log poles. We are thus able to show that the non-archimedean limit of the Kähler-Einstein metrics is the model metric $\phi_{K_{\X_c/\D}}$ (see example \ref{ex model}) associated to the canonical model $(\X_c, K_{\X_c/\D})$:
\begin{thmB}
Let $X \xrightarrow{\pi} \D^*$ be a degeneration of canonically polarized manifolds, $L = K_{X/\D^*}$, and let $\phi_{\KE} \in \PSH(X, L)$ be the family of Kähler-Einstein metrics. We assume that the family $X$ has semi-stable reduction over $\D$ (which can always be achieved after finite base change). Then the metric on $L^{\hyb}$ defined by:
$$\phi_{| X} = \phi_{\KE},$$
$$\phi_0 = \phi_{K_{\X_c/R}}$$
is continuous and plurisubharmonic.
\end{thmB}
In particular, using Favre's theorem mentioned above, this recovers the convergence of associated Monge-Ampère measures, which was previously obtained in \cite[thm. A]{PS}.
\\In the paper \cite{PS22}, we furthermore study in detail the case of toric metrics on the hybrid space associated to a complex toric variety $Z$, and provide a combinatorial description thereof, in the spirit of \cite{BPS}. This allows us to describe the solution to the non-archimedean Monge-Ampère equation on the Fermat family of Calabi-Yau hypersurfaces, as studied in \cite{Li1} in relation with the SYZ conjecture.
\\ \textbf{Notation and conventions.}
All rings are assumed to be unitary and commutative.
\\We will use additive notation for line bundles: if $L$, $M$ are two line bundles on a variety $X$, we write $L+M:= L\otimes M$, and $kL:= L^{\otimes k}$ for $k \in \Z$.
\\If $X$ is a complex manifold and $\phi$ a smooth function on $X$, we set $dd^c \phi = \frac{i}{\pi} \ddbar \phi$. We extend the notation to Hermitian metrics on line bundles, so that if $L$ is a holomorphic line bundle on $X$ and $\phi$ a smooth metric on $L$, the curvature form $dd^c \phi \in c_1(L)$ - and similarly for singular psh metrics.
\\Throughout this text, whenever we say that $X \xrightarrow{\pi} \D^*$ is a degeneration of complex manifolds, we mean that $X$ is a smooth complex manifold and $\pi$ a holomorphic submersion (which will often be omitted from notation). We will furthermore always assume that the degeneration is meromorphic at $0$, i.e. that there exists a normal complex analytic space $\X \xrightarrow{\pi}  \D$ such that $\X_{| \D^*} = X$.
\\ \textbf{Organization of the paper.} In section \ref{sec Berko}, we recall some general facts about Berkovich analytic spaces over Banach rings, and in section \ref{sec NA ppt} we define and prove basic properties of the class $\PSH(X, L)$ of plurisubharmonic singular metrics on a polarized scheme $(X, L)/A$ over an integral Banach ring; we also give some explicit examples along the way. We then move on to the case of hybrid spaces: section \ref{sec psh hyb} is devoted to the statement and the proof of theorem A, and section \ref{sec MA} is a discussion on families of Monge-Ampère measures, where we explain the proof of theorem \ref{theo Favre intro}. Finally, in section \ref{sec KE neg}, we restrict our attention to degenerations of canonically polarized manifolds and prove theorem B.
\\ \textbf{Acknowledgments.} I am grateful to my advisor S. Boucksom for his constant encouragement and for many helpful conversations, and to A. Ducros for helpful comments on a preliminary version of this paper. I would also like to thank Y. Odaka, R. Reboulet and the anonymous referee for various comments on a preliminary version of this paper.
\section{Berkovich analytic spaces} \label{sec Berko}
\subsection{Definitions} \label{sec def Berko}
\begin{defn}
Let $A \neq 0$ be a ring. A (submultiplicative) semi-norm $\lVert \cdot \rVert$ on $A$ is a map $\lVert \cdot \rVert : A \fl \R_{\ge 0}$ such that :
\begin{itemize}
\item $\lVert 1 \rVert=1$ and $\lVert 0 \rVert=0$,
\item $\forall a, b \in A$, $\lVert a+b \rVert \le \lVert a\rVert+ \lVert b\rVert$,
\item $\forall a, b \in A$, $\lVert ab \rVert \le \lVert a \rVert \lVert b\rVert$.
\end{itemize}
Its kernel $\Ker \lVert \cdot \rVert =\{ a \in A / \lVert a\rVert=0 \}$ is an ideal of $A$, which is prime when $\lVert \cdot \rVert$ is multiplicative. A submultiplicative semi-norm on $A$ whose kernel is reduced to zero is called a norm on $A$.
\\Finally, a Banach ring $A$ is a non-zero ring equipped with a submultiplicative norm $\lVert \cdot \rVert$ such that $A$ is complete with respect to $\lVert \cdot \rVert$.
\end{defn}
For example, any non-zero ring $A$ endowed with the trivial norm $\lVert \cdot \rVert_0$ (such that $\lVert a \rVert_0 =1$ for any non-zero $a \in A$) is a Banach ring.
\begin{defn} Let $A$ be a Banach ring.
\\The Berkovich spectrum $\M(A)$ is the set whose points $x \in \M(A)$ are multiplicative semi-norms $\lvert \cdot \rvert_x : A \fl \R_{\ge 0}$ satisfying $\lvert \cdot \rvert_x \le \lVert \cdot \rVert$. 
\\It is equipped with the topology of pointwise convergence on $A$, which makes it into a non-empty Hausdorff compact topological space by \cite{Ber}, and with a map $q: \lvert \cdot \rvert_x \mapsto \mathfrak{p}_x = \Ker (\lvert \cdot \rvert_x)$ to $\Sp(A)$ which is continuous.
\end{defn}
For instance, if $A=k$ is a complete valued field, then $\M(k)$ is reduced to the point $\lVert \cdot \rVert$.
\begin{ex}
Let $A = \Z$ be the ring of integers, endowed with the usual archimedean absolute value $\lvert \cdot \rvert_{\infty}$. Let $x_0 = \lvert \cdot \rvert_0 \in \mathscr{M}(\Z)$ be the trivial absolute value on $\Z$, and let $x_p = \lvert \cdot \rvert_p$ be the p-adic absolute value, normalized setting $\lvert p \rvert_p = p^{-1}$. 
\\It follows from Ostrowski's theorem that any point $x \in \M(\Z)$ is of the following form: either there exists $\eps \in [0,+\infty]$ and $p$ a prime number such that $x = \lvert \cdot \rvert_p^{\eps}$, or there exists $\eps \in [0,1]$ such that $x = \lvert \cdot \rvert_{\infty}^{\eps}$; where we denote $\lvert \cdot \rvert_x^0 = \lvert \cdot \rvert_0$ the trivial absolute value for any $x \in \M(\Z)$, while $\lvert \cdot \rvert_p^{\infty}$ is the absolute value such that $\lvert n \rvert_p^{\infty} =0$ if $p$ divides $n$, and $\lvert n \rvert_p^{\infty} =1$ otherwise.
\\Topologically, the Berkovich spectrum $\M(\Z)$ is thus an infinite wedge of segments \\parametrized by the prime numbers and $\infty$, glued together at the trivial absolute value $x_0$. Note that for each neighbourhood $V$ of $x_0$ in $\M(\Z)$, the set of branches not contained in $V$ is finite. 
\\The map $\M(\Z) \xrightarrow{q} \Sp(\Z)$ maps the outer end  $\lvert \cdot \rvert_p^{\infty}$ of the p-adic branch to the prime ideal $(p)$, and any other point to the generic point of $\Sp(\Z)$.
\end{ex}
The analytification of an $A$-scheme of finite type is now defined as follows:
\begin{defn} \label{def an} Let $B$ be a finitely generated $A$-algebra. The analytification $Y^{\an}$ of $Y =\Sp(B)$ is  the set of multiplicative semi-norms $\lvert \cdot \rvert_y$ on $B$ whose restriction to $A$ belong to $\M(A)$. \\It is endowed with the coarsest topology making the maps $y \mapsto \lvert f \rvert_y$ continuous, for $f \in B$, and comes with a continuous structure map $Y^{\an} \fl \M(A)$, sending a semi-norm to its restriction to $A$.
\\If $X$ is a scheme of finite type over $A$, one can then glue the analytifications of affine charts of $X$ in order to define the analytic space $X^{\an}$, together with the structure map $\pi : X^{\an} \fl \M(A)$.
\end{defn}
The space $X^{\an}$ satisfies nice topological properties: if $X/A$ is separated, then $X^{\an}$ is Hausdorff; and $X^{\an}$ is compact whenever $X/A$ is projective.
\\Any $A$-analytic space comes with a sheaf of analytic functions, as defined in \cite[def. 1.5.3]{Ber}. One can then show that the above definition induces an \emph{analytification functor} $X \mapsto X^{\text{an}}$ from the category of $A$-schemes of finite type to the category of $A$-analytic spaces, that was defined in \cite{LP}. 
\begin{ex}
\label{ex GM}
Let $A = \C$ endowed with the Euclidean absolute value, and $B$ be a complex Banach algebra of finite type. Then the classical Gelfan'd-Mazur theorem implies that $Y^{\an} = (\Sp B)^{\an}$ is the set of maximal ideals of $B$. As a result, $Y^{\an}$ is the set $Y(\C)$ of closed points of the affine complex algebraic variety $Y$, and the induced topology on $Y$ is the Euclidean one. More generally, if $Y$ is a reduced scheme of finite type over $\C$, then the analytification $Y^{\an}$ of $Y$ with respect to the Euclidean absolute value on $\C$ is homeomorphic to the complex variety $Y$ endowed with the Euclidean topology.
\end{ex}

\begin{ex}
Let $A=K$ be a complete non-archimedean field, and let $X/K$ be an integral, separated scheme of finite type. Then the Berkovich space $X^{\an}$ can be described more explicitly as the set of couples $x=(\xi, v_x)$, where $\xi = \xi_x$ is a scheme-theoretic point of $X$ and $v_x$ is a real valuation on the function field of $\overline{\xi}_x$, extending the valuation on $K$. 
\\Moreover, the map $x \mapsto \xi_x$ from $X^{\an}$ to $X^{\sch}$ is continuous, surjective, and induces a bijection between the respective sets of connected components; the closed subscheme $Y := \overline{\xi}_x$ will be called the support of $x \in X^{\an}$.
\end{ex}
\begin{defn}
Let $A$ be a Banach ring, and $x \in \M(A)$. Write $\mathfrak{p}_x =q(x)= \Ker (\lvert \cdot \rvert_x)$, $\kappa(x) = \Frac(A/\mathfrak{p}_x)$ the schematic residue field of $A$ at $\mathfrak{p}_x$. The semi-norm $\lvert \cdot \rvert_x$ descends to an absolute value on $\kappa(x)$, we write $\mathscr{H}(x)$ for the associated completion of $\kappa(x)$ and call it the residue field of $\M(A)$ at $x$.
\end{defn}
Assume given two Banach rings $A$ and $B$, together with a bounded ring homomorphism $A\fl B$. This induces a continuous map:
$$\rho : \M(B) \fl \M(A),$$
which simply sends a semi-norm on $B$ to its pull-back to $A$.
\\If $X/A$ is a scheme of finite type and $X_B := X \times_{A} B$, we additionally have a continuous map $F: X^{\an}_B \fl X^{\an}$, defined explicitly by restriction of semi-norms on an open cover of $X$ by affine schemes (see def. \ref{def an});  and which fits into the commutative diagram:
\begin{center}
    \begin{tikzcd}
    X^{\an}_B \arrow[r, "F"] \arrow[d, "\pi_B"]  & X^{\an} \arrow[d, "\pi_A"] \\
    \M(B)  \arrow[r, "\rho"] & \M(A).
   \end{tikzcd}
\end{center}
The following proposition allows us to view - at least topologically - a Berkovich space over a Banach ring $A$ as a family of Berkovich spaces over complete valued fields, parametrized by $\M(A)$:
\begin{prop} \label{prop gbr}
Let $A$ be a Banach ring, $X$ a scheme of finite type over $A$, and $\pi : X^{\an} \fl \M(A)$ the associated analytic space. If $x \in \M(A)$, then $\pi^{-1}(x)$ is canonically homeomorphic to the analytification of the base change $X_{\mathscr{H}(x)} = X \times_A \mathscr{H}(x)$ with respect to the absolute value $\lvert \cdot \rvert_x$ on $\mathscr{H}(x)$. \\Moreover, the base change map $F_x : X^{\an}_{\mathscr{H}(x)} \fl X^{\an}$ is the inclusion $\pi^{-1}(x) \subset X^{\an}$ under this homeomorphism.
\end{prop}
\begin{proof} We treat the case where $X = \Sp B$, with $B$ an $A$-algebra of finite type, as the general case will follow from gluing. By definition of $X^{\an}$ and the structure map $\pi$, the fiber $\pi^{-1}(x)$ is the set of multiplicative semi-norms on $B$ that restrict to $\lvert \cdot \rvert_x$ on $A$ - or equivalently the set of (equivalence classes of) morphisms $B \fl K$ to a complete valued field extension $K/\mathscr{H}(x)$, restricting to the morphism $A \fl \mathscr{H}(x)$ on $A$.
\\By the universal property of the tensor product of algebras, this is the same as a morphism $B \otimes_A \mathscr{H}(x) \fl K$  inducing on $\mathscr{H}(x)$ the given embedding of $\mathscr{H}(x)$ into $K$. Such a morphism then produces a semi-norm on $B \otimes_A \mathscr{H}(x)$ restricting to $\lvert \cdot \rvert_x$ on $\mathscr{H}(x)$, hence a point in $X_{\mathscr{H}(x)}^{\an}$. It is then straightforward to see that this bijection is compatible with the weak topologies and thus induces a homeomorphism between $\pi^{-1}(x)$ and $X^{\an}_{\mathscr{H}(x)}$, compatible with the inclusion and base change map respectively.
\end{proof}
\begin{rem} When $A$, $B$ are geometric base rings in the sense of Lemanissier-Poineau \cite[def. 3.3.8]{LP} (which is the case of all Banach rings considered in this paper), the base change $X^{\an}_B$ is indeed the fiber product $X^{\an} \times_{\M(A)} \M(B)$ in the category of $A$-analytic spaces, see \cite[§4.2, 4.3]{LP}.
\end{rem}
\begin{defn} \label{def Mbir}
Let $A$ be an integral Banach ring, and write $\eta_A$ for the generic point of $\Sp(A)$. We say a point $x \in \M(A)$ is Zariski-dense if and only the kernel of $\lvert \cdot \rvert_x$ is reduced to zero.
\\We write $\M(A)^{\eta} = q^{-1}(\eta_A) \subset \M(A)$ for the subset of Zariski-dense points.
\end{defn}
As the name suggests, those are indeed the points of $\M(A)$ which are dense for the Zariski topology, where the Zariski topology on $\M(A)$ is the coarsest making the map $q : \M(A) \fl \Sp (A)$ continuous.
\subsection{Discretely-valued fields}\label{section DVF}
Let $K$ be a complete, discretely-valued field, with valuation $v$. We denote by $R= \{ v \ge 0 \}$ its valuation ring, $\mathfrak{m} = \{ v>0 \}$ its maximal ideal and $k= R/ \mathfrak{m}$ the residue field. We will furthermore assume that $K$ has equicharacteristic zero, i.e. $k$ and $K$ both have characteristic zero. In this case, it follows from Cohen's structure theorem that $K$ is isomorphic to the field $K =k((t))$ of Laurent series over its residue field, endowed with the valuation $v= \ord_0$.
\\We let $X$ be an integral, separated $K$-scheme of finite type, and write $n = \dim(X)$. The purpose of this section is to explain how one can understand the topological space $X^{\an}$ more concretely using piecewise-affine geometry. The basic idea is that for a large enough class of integral $R$-models $\X$ of $X$, there exists a finite-dimensional cell complex $\Sk(\X) \subset X^{\an}$, which we may view as a tropicalization of the model $\X$. As a matter of fact, by \cite[thm. 10]{KontsevichSoibelman}, the space $X^{\an}$ can be realized as the inverse limit of all such $\Sk(\X)$, so that $X^{\an}$ is homeomorphic to a tower of simplicial complexes.
\\We start with a definition:
\begin{defn} A model of $X$ is a flat, separated $R$-scheme $\X$, together with an isomorphism of $K$-schemes $\X \times_{R} K \simeq X$.
\\We will denote $\X_0 := \X \times_{R} k$ the special fiber of $\X$, and by $\Div_0(\X)$ the group of Weil divisors on $\X$ supported on the special fiber.
\end{defn}
If $\X$, $\X'$ are two models of $X$, a morphism of models $f: \X' \fl \X$ is an $R$-morphism whose base change to $K$ induces the identity on $X$. We will say that $\X'$ dominates $\X$ if there exists such a morphism, in which case it is unique.
\\Assume that $X/K$ is proper, and let $\X/R$ be a proper model of $X$. By the valuative criterion of properness, for any $x =(\xi_x,v_x) \in X^{\an}$, the $K$-morphism $\Sp \mathscr{H}(x) \fl X$ - whose image is the point $\xi_x$ - lifts in a unique way to an $R$-morphism from the valuation ring $\mathscr{H}(x)^{\circ}$ to $\X$:
\begin{center}
    \begin{tikzcd}
    \Sp \mathscr{H}(x) \arrow[r, "\xi_x"] \arrow[d]  & \X \arrow[d] \\
    \Sp \mathscr{H}(x)^{\circ} \arrow[ru, dashed] \arrow[r] & \Sp R.
   \end{tikzcd}
\end{center}
The image of the closed point of $\Sp \mathscr{H}(x)^{\circ}$ under the extended morphism is called the center of $x$ and denoted by $c_\X(x)$. The map $c_{\X} : X^{\an} \fl \X_0$ turns out to be surjective and anticontinuous, i.e. the preimage of a closed subset of $\X_0$ by $c_{\X}$ is open in $X^{\an}$.
\\In the case where $X/K$ is smooth, we say a model $\X/R$ has simple normal crossing singularities if $\X$ is regular, and the special fiber $\X_0$ is a divisor with simple normal crossing support inside $\X$. Such models always exists when $K$ has equicharacteristic zero and $X/K$ is projective, by Hironaka's theorem on resolution of singularities. More precisely, any model $\X/R$ can be dominated by an snc model.
\\To every snc model $\X$ of $X$, with special fiber $\X_0 = \sum_{i \in I} a_i D_i$, we can associate a cell complex encoding the combinatorics of the intersections of the irreducible components of $\X_0$. We say $Y \subset \X_0$ is a \emph{stratum} if there exists a non-empty $J \subset I$ such that $Y$ is a connected component of $D_J := \cap_{j \in J} D_j$. The dual complex of $\X$ is now defined as follows:
\begin{defn} \label{def dual cpx}
Let $\X$ be an snc model of $X$. 
To each stratum $Y$ of $\X_0$ which is a connected component of $D_J$, we associate a simplex:
$$ \tau_Y = \{ w \in \R_{\geqslant 0}^{|J|} \, | \sum_{j \in J} a_j w_j =1 \}.$$
We define the cell complex $\mathcal{D}(\X_0)$ by the following incidence relations: $\tau_Y$ is a face of $\tau_{Y'}$ if and only if $Y' \subset Y$.
\end{defn}
Given any snc model $\X$ of $X$ over $R$, there exists a natural embedding $i_{\X}$ of the dual complex $\mathcal{D}(\X_0)$ into $X^{\an}$, given as follows. 
The vertices $v_i$ of $\mathcal{D}(\X_0)$ are in one-to-one correspondence with irreducible components $D_i$ of the special fiber $\X_0 = \sum_{i \in I} a_i D_i$, so that we set $$i_{\X}(v_{i}) = v_{D_i} := a^{-1}_i \ord_{D_i},$$ where the valuation $\ord_{D_i}$ associates to a meromorphic function $f \in K(X) \simeq K(\X)$ its vanishing order along $D_i$ - the normalisation by $a^{-1}_i$ ensuring that $v_{D_i}(t) = 1$. 
\begin{defn} A valuation given in this way, for some snc model $\X$ of $X$, is called \emph{divisorial}. We write $X^{\xdiv} \subset X^{\an}$ for the set of divisorial valuations, it is a dense subset of $X^{\an}$.
\end{defn}
One can now interpolate between those divisorial valuations using \emph{quasi-monomial} valuations, in order to embed $\mathcal{D}(\X_0)$ into $X^{\text{an}}$:

\begin{prop}[{\cite[Proposition 2.4.4]{MN}}] \label{prop casi mono}
Let $\X$ be an snc model of $\X$, with special fiber $\X_0 = \sum_{i \in I} a_i D_i$.
Let $J \subset I$ such that $D_J = \cap_{j \in J} D_j$ is non-empty, and $Y$ a connected component of $D_J$, with generic point $\eta$. 
We furthermore fix a local equation $z_j \in \gO_{\X,\eta}$ for $D_j$, for any $j \in J$.
\\Then, for any $w \in \tau_Y = \{ w \in \R^{|J|}_{\geqslant 0} \,| \sum_{j \in J} a_j w_j =1 \}$, there exists a unique valuation
$$v_w : \gO_{\X,\eta} \fl \R_{\geqslant 0} \cup \{ + \infty\}$$
 such that for every $f \in \gO_{\X,\eta}$, with expansion $f = \sum_{\beta \in \N^{|J|}} c_{\beta} z^{\beta}$ (with $c_{\beta}$ either zero or unit), we have:
$$v_w(f) = \min \{ \langle w , \beta \rangle \,| \beta \in \N^{|J|},  c_{\beta} \neq 0 \},$$
where $\langle \; , \; \rangle$ is the usual scalar product on $\R^{|J|}$.
\end{prop}
The above valuation is called the quasi-monomial valuation associated with the data $(Y, w)$. Then:
\begin{align*}
    i_\X: & \quad \mathcal{D}(\X_0) \rightarrow X^{\text{an}} \\
    & \quad \tau_Y \ni w \mapsto v_w
\end{align*}
gives a well-defined continuous injective map from $\mathcal{D}(\X_0)$ to $X^{\text{an}}$. 
\begin{defn}
We call the image of $\mathcal{D}(\X_0)$ by $i_{\X}$ the \emph{skeleton} of $\X$, written as $\Sk(\X) \subset X^{\emph{an}}$. It is a cell complex of dimension at most $\dim X$.
\end{defn}
\noindent By compactness of $\mathcal{D}(\X_0)$, $i_{\X}$ induces a homeomorphism between $\mathcal{D}(\X_0)$ and $\Sk(\X)$, so that we will sometimes abusively identify $\mathcal{D}(\X_0)$ with $\Sk(\X)$.
\\We can now define a retraction for the inclusion $\Sk(\X) \subset X^{\an}$ as follows: for any  $v \in X^{\an}$, there exists a minimal stratum $Y \subseteq  \cap_{j \in J} D_j$ of $\X_0$ such that the center $c_{\X}(v)$ of $v$ is contained in $Y$. We now associate to $v$ the quasi-monomial valuation $\rho_{\X}(v)$ corresponding to the data $(Y, w)$ with $w_j = v(z_j)$, where $z_j$ is a local equation of $D_j$ at the generic point of $c_{\X}(v)$. This should be seen as a monomial approximation of the valuation $v$ at the generic point of $Y$, with respect to the model $\X$.
\begin{defn}
The above map $\rho_{\X} : X^{\an} \longrightarrow \Sk(\X)$ is the Berkovich retraction associated with the model $\X/R$.
\end{defn}
The Berkovich retraction is continuous, restricts to the identity on $\Sk(\X)$, and by \cite{Th} \cite{Ber2}, $\rho_{\X}$ is a strong deformation retraction, i.e. there exists a homotopy between $\rho_\X$ and the identity on $X^{\an}$ that fixes the points of $\Sk(\X)$. It follows that $X^{\an}$ and $\Sk(\X)$ are homotopy equivalent. 
\\Let $\X/R$ be a model of $X$, and let $\mathscr{I}$ be a coherent ideal sheaf on $\X$. It induces a function:
$$\phi_{\mathscr{I}} : X^{\an} \fl \R,$$
sometimes denoted by $\phi_{\mathscr{I}} = \log \lvert \mathscr{I} \rvert$, as follows. For $x \in X^{\an}$, we set:
$$\phi_{\mathscr{I}} = \sup_{f} \log \lvert f(x) \rvert,$$
where the supremum runs over the $f \in \mathscr{I}_{c_{\X}(x)}$ - or equivalently, over the finitely many generators of $\mathscr{I}$ at $c_{\X}(x)$.
\\If $\mathscr{I}_1$, $\mathscr{I}_2$ are vertical ideal sheaves on $\X$, $\X'$ respectively (i.e., cosupported on the special fiber), then the equality $\phi_{\mathscr{I}_1} = \phi_{\mathscr{I}_2}$ holds if and only there exists a model $\X''$ dominating both $\X$ and $\X'$ such that the pullbacks:
$$\gO_{\X''} \cdot \mathscr{I}_1 = \gO_{\X''}\cdot \mathscr{I}_2$$
agree on $\X''$. In particular, when working with a function of the form $\phi_{\mathscr{I}}$ for a vertical ideal sheaf, we may choose a log-resolution of the ideal $\mathscr{I}$, so that we may assume that there exists a model $\X$ such that $\mathscr{I} = \gO_{\X}(D)$ for a vertical divisor $D \in \Div_0(\X)$. This implies for instance the following:
\begin{lem} Let $\mathscr{I}$ be a vertical ideal sheaf on a model $\X$ of $X$, and $\phi_{\mathscr{I}} : X^{\an} \fl \R$ the associated function. Then for any snc model $\X'$ of $X$, the restriction of $\phi_{\mathscr{I}}$ to the skeleton $\Sk(\X')$ is integral piecewise-affine.
\end{lem}
Note that the converse also holds. This motivates the following terminology:
\begin{defn} \label{def model f}
A function $\phi : X^{\an} \fl \R$ is integral piecewise-affine ($\Z$-PA for short) if there exists an snc model $\X$ and a vertical ideal sheaf $\mathscr{I}$ on $\X$ such that $\phi = \phi_{\mathscr{I}}$.
\end{defn}
\subsection{Hybrid spaces} \label{sec hyb}
Let $(k, \lvert \cdot \rvert)$ be a non-trivially valued field, either archimedean or non-archimedean. The following Banach ring was introduced by Berkovich \cite{Ber3}, and further studied for instance by Boucksom-Jonsson \cite[app. A]{BJ}:
\begin{defn}
Let $k^{\hyb}$ be the Banach ring obtained by equipping the field $k$ with the norm $\lVert \cdot \rVert_{\hyb}$, defined for non-zero $z \in k$ by:
$$ \lVert z \rVert_{\hyb} = \max \{ 1 , \lvert z \rvert \}.$$
\end{defn}
One can show \cite[ex. 1.1.15]{LP} that the elements of $\M(k^{\text{hyb}})$ are of the form $\lvert \cdot \rvert^{\lambda}$, for $\lambda \in [0,1]$, where $\lvert \cdot \rvert^0 = \lvert \cdot \rvert_0$ denotes the trivial absolute value on $k$. This yields a homeomorphism $\lambda : \M(k^{\text{hyb}}) \xrightarrow{\sim} [0,1]$.
\\Thus, if $Z$ is a scheme of finite type over $k$, its analytification with respect to $\lvert \cdot \rvert_{\text{hyb}}$, which we denote by $Z^{\text{hyb}}$, comes with a structure morphism $\pi : Z^{\text{hyb}} \fl [0,1]$. If $Z = \Sp(A)$ is affine, the fiber over $\lambda \neq 0$ is by definition of $\pi$ the set of semi-norms extending the absolute value $\lvert \cdot \rvert^{\lambda}$ on $k$, so that by rescaling, this is easily seen to be homeomorphic to the analytification $Z^{\an}$ of $Z$ with respect to the absolute value $\lvert \cdot \rvert$. One can in fact show that for any $Z$ of finite type, we have a homeomorphism: 
$$p: \pi^{-1}((0,1]) \xrightarrow{\sim} (0,1] \times Z^{\an},$$
compatible with the projections to $(0, 1]$.
\\On the other hand, the fiber $\pi^{-1}(0)$ consists of the semi-norms extending the trivial absolute value on $k$, so that this is homeomorphic to the analytification $Z^{\text{an}}_0$ of $Z$ with respect to the trivial absolute value on $k$.
\\Hence, the space $Z^{\text{hyb}}$ allows us to see the analytic space $Z^{\an}$ degenerate to its trivially-valued counterpart.
\\In the case where $k = \C$ with the Euclidean absolute value, the analytification $Z^{\an}$ is homeomorphic to the usual complex analytification $Z^{\hol}$ of $Z$, by example \ref{ex GM}. Thus, the space $Z^{\text{hyb}} $ provides a natural way to degenerate the complex manifold $Z^{\text{hol}}$ to the non-archimedean analytic space $Z^{\text{an}}_0$.
\\We now want to perform a similar construction for degenerations of complex varieties. Let $X \xrightarrow{\pi} \D^*$ be a holomorphic family of $n$-dimensional complex manifolds, where $\D^* =\{ \lvert t \rvert <1 \}$ is the punctured unit disk in $\C$. We will furthermore assume that the family is quasi-projective and meromorphic at zero, i.e. that there exists a relatively algebraic embedding $\iota : X \hookrightarrow \CP^{N} \times \D^*$ such that $\pi =\pr_2 \circ \iota$, and the equations of $X$ have meromorphic singularity at $t=0$. This allows us to view (the base change of) $X$ as a quasi-projective scheme over the non-archimedean field $K = \C((t))$ of Laurent series, we will write $X^{\an}$ for the Berkovich analytification of $X$ with respect to the $t$-adic absolute value on $K$, normalized such that $\lvert t \rvert =r$.
\\We fix a radius $r \in (0,1)$, and consider the following Banach ring, which we call the \emph{hybrid ring}:
$$A_r = \{ f = \sum_{n \in \Z} a_n t^n \in K  \; / \; \lVert f \rVert_{\text{hyb}} := \sum_{n} \lVert a_n \rVert_{\text{hyb}} \; r^n < \infty \}.$$
The purpose of the above Banach ring is to provide a presentation of the closed complex disk as an affine non-archimedean analytic space: we denote by $C^{\hyb}(r) := \M(A_r)$ the Berkovich spectrum of $A_r$ and call it the \emph{hybrid circle}, the terminology stems from the fact that $C^{\hyb}(r)$ is homeomorphic to the circle $\{ \lvert T \rvert =r \}$ inside the Berkovich affine line over $\C^{\hyb}$ \cite{PoiZ}. We now have the following more explicit description of the hybrid circle:
\begin{lem}{(\cite[prop. A.4]{BJ})} \label{cercle hyb}
\\The map $\tau: t \mapsto \lvert \cdot \rvert_t$ defined by:
\[ \lvert f \rvert_t = \left\{ \begin{array}{ll} r^{\ord_0(f)} & \mbox{if $t =0$},\\
 r^{\log \lvert f(t) \rvert/ \log \rvert t \rvert} & \mbox{if $t \neq 0$}
\end{array} \right. \]
\\for $f \in A_r$ induces a homeomorphism from $\bar{\D}_r$ to $C^{\hyb}(r) =\M(A_r)$.
\end{lem}
The upshot of this construction is that is $f$ in $A_r$, then up to a constant $\log \lvert f(\tau(t)) \rvert = \frac{\log \lvert f(t) \rvert}{\log \lvert t \rvert}$ for $t \neq 0$ - we are viewing the point $t$ as a rescaling of the Euclidean absolute value composed with the evaluation map at $t$. Additionally, as $t \rightarrow 0$ these rescaled absolute values converge to the non-archimedean $t$-adic absolute value $r^{\ord_0}$ on $A_r \subset \C((t))$. This motivates the following definition:
\begin{defn}
Let $X \xrightarrow{\pi} \D^*$ be a quasi-projective degeneration of complex manifolds as above, and view it as a scheme of finite type over the ring of convergent power series. We write $X_{A_r}$ its base change to the ring $A_r$. We define the hybrid space $X^{\hyb}_r$ associated to $X$ as the analytification of $X_{A_r}$ over $A_r$, which comes with a structure map:
$$\pi_{\hyb} : X_r^{\hyb} \fl C^{\hyb}(r).$$
\end{defn}
The hybrid space allows us to see the complex space $X$ degenerate to its non-archimedean analytification, as a consequence of the following:
\begin{prop}{\cite[thm. 1.2]{Fav}} \label{topo hyb}
Let $X$ be a degeneration of complex manifolds, $X_{A_r}$ the associated $A_r$-scheme, and denote the associated hybrid space by $\pi_{\hyb} : X_r^{\hyb} \fl C^{\hyb}(r)$. Then:
\begin{itemize}
\item $\pi_{\hyb}^{-1}(0)$ can be canonically identified with $X^{\an}$,
\item there exists a homeomorphism $\beta : X_{ | \D^*_r} \xrightarrow{\sim} \pi_{\hyb}^{-1}(\tau(\D^*_r))$, satisfying $\pi_{\hyb} \circ \beta = \tau \circ \pi$,
\item if $\phi$ is a rational function on $X_{A_r}$, then $\lvert \phi (\beta(z)) \rvert = \lvert \phi (z) \rvert^{\frac{\log r}{\log \lvert t \rvert}}$ for $z$ not in the indeterminacy locus of $\phi$.
\end{itemize}
\end{prop}
Thus, heuristically, the hybrid space allows us to see the scalings of the usual modulus on $\C$ given by $\lvert \cdot \rvert^{\frac{\log r}{\log \lvert t \rvert}}$ degenerate to the non-archimedean absolute value on $K$, and hence to see the complex manifolds $\{X_t\}_{t \in \D^*}$ degenerate to the non-archimedean analytification $X^{\an}$.
\\Assume for instance that $(\mu_t)_{t \in \D^*}$ is a continuous family of probability measures on $X$, such that $\mu_t$ is supported on $X_t$ for each $t \in \D^*$. Since the hybrid space provides a canonical compactification of $X$ over the puncture, it is a natural question to ask whether or not the family of measures converges on $X^{\hyb}$, at least in a weak sense - more concrete examples of such situations will be given in sections \ref{sec MA family} and \ref{sec setup neg}; see also \cite{Shiv1}, \cite{Shiv2}.
\subsection{The isotrivial hybrid space} \label{sec iso}
We set $K = \C((t))$. Let $X$ be a projective complex variety, and write $\tilde{X} := X \times \D^*$ the associated trivial degeneration of complex varieties, as well as $X_K = X \times_{\C} K$. We thus have two hybrid spaces associated to $X$: the hybrid space $X^{\hyb}_0$ obtained by viewing $X$ as a scheme over $\C^{\hyb}$, and $X^{\hyb}_K$ the hybrid space associated to the degeneration $\tilde{X}$. The goal of this section is to compare both hybrid spaces, so that results established for hybrid spaces associated to degenerations will naturally yield similar statements for spaces over $\C^{\hyb}$, simply by specializing to a trivial degeneration.
\\We start by comparing the non-archimedean fibers. We write $X^{\an}_K$ for the analytification of $X_K$, and $X^{\an}_0$ for the analytification of $X$ with respect to the trivial absolute value on $\C$. The $t$-adic absolute value on $\C((t))$ restricts to the trivial absolute value on $\C$, so that there exists a base change morphism $f : X^{\an}_K \fl X^{\an}_0$, that can be described as follows. If $\mathcal{K}(X)$ is the function field of $X$, then the function field of $X_K$ is simply $\mathcal{K}(X_K) = \mathcal{K}(X)((t)) = \mathcal{K}(X) \otimes_{\C} K$. Hence, any valuation $v$ on $\mathcal{K}(X_K)$ induces by restriction a valuation $f(v)$ on $\mathcal{K}(X)$, and similarly for semi-valuations.
\\We now compare the base rings: for $r \in (0,1)$, the inclusion $\C^{\hyb} \hookrightarrow A_r$ is compatible with the hybrid norms, so that it induces a continuous map $\lambda: \M(A_r) \fl \M(\C^{\hyb})$, obtained by restricting semi-norms from $A_r$ to $\C^{\hyb}$. It is straightforward to check that under the homeomorphisms $\M(A_r) \simeq \bar{\D}_r$ and $\M(\C^{\hyb}) \simeq [0,1]$, we have $\lambda(t) = \frac{\log r}{\log \lvert t \rvert}$ for $t \in \bar{\D}_r$. We furthermore have the following description of the base change of $X$ from $\C^{\hyb}$ to $A_r$, which is a straightforward consequence of transitivity of base change:
\begin{prop} \label{prop bc hyb} Let $r \in (0,1)$, and let $\lambda : \bar{\D}_r \fl [0,1]$ be the map defined by $\lambda(t) = \frac{\log r}{\log \lvert t \rvert}$. We have a commutative diagram:
\begin{center}
    \begin{tikzcd}
    X^{\hyb}_K \arrow[r, "F"] \arrow[d, "\pi_K"]  & X^{\hyb}_0 \arrow[d, "\pi_0"] \\
    \bar{\D}_r  \arrow[r, "\lambda"] & \big[0,1\big].
   \end{tikzcd}
\end{center}
where $F$ is the base change of $X$ from  $\C^{\hyb}$ to $A_r$, and such that $F_{| \pi_K^{-1}(0)} =f$. Moreover, for any $t \in \bar{\D}_r^*$, $F_{| \pi_K^{-1}(t)}$ induces the identity on $X^{\hol}$ under the homeomorphisms from section \ref{sec hyb}.
\end{prop}
The map $f : X^{\an}_K \fl X^{\an}_0$ furthermore admits a continuous section, called the Gauss section, defined in the following way. 
In the terminology of \cite[def. 3.2]{Poi}, every point of $X^{\an}_0$ is universal (peaked point, in the terminology of Berkovich) as $\C$ is algebraically closed, so that any $x \in X^{\an}_0$ admits a canonical lift to $X^{\an}_K$, denoted by $\gamma(x)$; we call the map $\gamma : X^{\an}_0 \fl X^{\an}_K$ the \emph{Gauss section}. More concretely, if $x= v_x \in X^{\an}_0$ is a valuation on $\mathcal{K}(X)$, it is extended as a valuation:
$$\gamma(v) : \mathcal{K}(X)((t)) \fl \R \cup \{ - \infty \},$$
such that $\gamma(v)(t) =1$. For instance, if $S= \sum_{n \ge 0} s_n t^n$ an element of $\mathcal{K}(X)[t]$, the canonical extension is defined by the formula:
$$\gamma(v)(S) = \min_{n} (v(s_n) +n).$$
\section{Global pluripotential theory} \label{sec NA ppt}
Let $A$ be an integral Banach ring, and let $X$ be a projective $A$-scheme of finite type. Following common practice, we will call line bundle on $X$ any locally free $\gO_X$-module of rank $1$. We will use additive notation for the group law on the set of isomorphism classes of line bundles. 
\\Let $L$ be a semi-ample line bundle on $X$, and write $X^{\an} \rightarrow \M(A)$ the Berkovich analytification of $X$. The purpose of this section is to define a class of plurisubharmonic metrics $\PSH(X, L)$ on $L$, in a similar way to the case where $A=K$ is a complete valued field. For instance when $K = \C$, and $(X, L)$ is a smooth polarized variety, then the class $\PSH(X, L)$ we define is nothing but the usual class of plurisubharmonic metrics on $L$, which is one of the central objects of pluripotential theory, and has been extensively studied at this point. When $K$ is a non-archimedean field, a similar class of plurisubharmonic metrics has been defined in increasing order of generality in \cite{Zha}, \cite{Gub}, \cite{BFJ2}, \cite{BE}, and our definition is built so that it generalizes the latter.
\\In our setting, the basic idea is that a metric $\phi \in \PSH(X, L)$ can be viewed as a family of semi-positive metrics $(\phi_x)_{x \in \M(A)}$, where $\phi_x \in \PSH(X_{\mathscr{H}(x)}, L_{\mathscr{H}(x)})$, that also varies in a plurisubharmonic way with respect to $x \in \M(A)$.
\\The main case of interest for us will be when $A = A_r$ is a hybrid ring, so that $X^{\an} = X^{\hyb}$ is the hybrid space associated to a degeneration $X$ of complex varieties. In this setting, our main result, theorem \ref{theo psh hyb}, states that any psh metric $\phi \in \PSH(X,L)$ on a polarized degeneration $X$ of complex manifolds satisfying a certain growth condition, induces naturally a psh metric on $(X^{\hyb}, L^{\hyb})$, whose restriction to the non-archimedean fiber $X^{\an}$ is the non-archimedean metric $\phi^{\NA}$ constructed in \cite{BBJ}, \cite{Reb}, and encodes the logarithmic singularities of $\phi$ along the special fibers of models of $X$.
\\We also compare our definition with the setting of \cite{Fav}, and obtain the continuity on $X^{\hyb}$ of the family of Monge-Ampère measures associated to a continuous, semi-positive metric on $L^{\hyb}$.
\subsection{Metrics on Berkovich spaces}
We start with some very general definition of metrics on line bundles on Berkovich analytic spaces.
\begin{defn} Let $L$ be a line bundle on $X$. A continuous metric $\phi$ on $L^{\an}$ consists of the following data: for any Zariski open subset $U \subset X$ and $s \in H^0(U, L_{| U})$ a trivializing section, a continuous function:
$$\lVert s \rVert_{\phi} : U^{\an} \fl \R_{>0}$$
such that $\lVert f s \rVert_{\phi} = \lvert f \rvert \lVert s \rVert_{\phi}$ for any regular function $f \in H^0(U, \gO_U)$, and compatible with restriction of sections.
\end{defn}
This allows us to define, for any open subset $V \subset X^{\an}$ and any analytic section $s$ of $L^{\an}$ on $V$, a continuous function:
$$\lVert s \rVert_{\phi} : V \fl \R_{\ge 0},$$
as follows: cover $X$ by Zariski open subsets $U_i, i \in I$ such that $L_{| U_i} = s_i \cdot \gO_{U_i}$, and write, on $V \cap U_i$:
$$ s= f_i \cdot s_i,$$
with $f_i$ an analytic function on $V$. Then we set:
$$\lVert s \rVert_{\phi} := \lvert f_i \rvert \lVert s \rVert_{\phi}$$
on $V \cap U^{\an}_i$. It is straightforward to check that this is independent on the choice of trivializations, compatible with restrictions and that the equality:
$$\lVert f s \rVert_{\phi} = \lvert f \rvert \lVert s \rVert_{\phi}$$
holds for any section $s$ and any analytic function $f$ on $V$.
\\From now on, we will use additive notation for metrics, i.e. identify the metric $\lVert \cdot \rVert_{\phi}$ with $\phi = -\log \lVert \cdot \rVert_{\phi}$ (more precisely, the collection of local functions $\phi_i = - \log \lVert s_i \rVert_{\phi}$ associated to local trivializations of $L$). In particular, if $L_1$, $L_2$ are two line bundles on $X$ and $\phi_i$ is a continuous metric on $L_i$ for $i=1,2$, then $\phi_1+\phi_2$ is a continuous metric on $L_1+L_2$. Moreover, if $\psi: X^{\an} \fl \R$ is a continuous function and $\phi$ a continuous metric on $L$, then $\phi+ \psi$ is also a continuous metric on $L$.
\begin{ex} Let $A=\C$ with the Euclidian absolute value, and $X/\C$ a smooth variety endowed with a line bundle. Then our definition matches the standard definition of a continuous Hermitian metric on $L$.
\end{ex}
\begin{ex} Let $A=k$ be a trivially-valued field, $X/k$ a proper variety and $L$ a line bundle on $X$. Then the trivial metric $\phi^{\triv}$ on $L$ is the unique metric on $L^{\an}$ such that for any pair $(U, s)$, with $U \subset X$ a Zariski open and $s \in H^0(U,L)$ a nowhere-vanishing section of $L$, the equality:
$$\lvert s(x) \rvert_{\triv} = 1$$
holds whenever the center $c(v_x)$ is contained in $U$.
\end{ex}
\begin{ex} Let $X = \CP^N_A$, and $L = \gO(1)$. Then the Fubini-Study metric $\phi_{\FS}$ on $L$ is defined by the formula:
$$\lVert s(x) \rVert_{\phi_{\FS}} = \frac{\lvert s(x) \rvert}{\max (\lvert x_0 \rvert,..., \lvert x_N \rvert)},$$
where the $x_i$'s are standard coordinates on $\CP^N_A$. We will write:
$$\phi_{\FS} = \max_{i\le N} \log \lvert x_i \rvert.$$
Note that while this definition is well-suited for the case when $A$ is a non-archimedean field, it does not recover the usual Fubini-Study metric on $\C \CP^N$ when $A=\C$, so that we will sometimes call the metric above the tropical Fubini-Study metric.
\end{ex}
In order to define a large enough class of semi-positive metrics, we need to allow metrics with some singularities.
\begin{defn} Let $L$ be a line bundle on $X$. A singular metric $\phi$ on $L^{\an}$ (or simply on $L$ when the norm is clear from the context) consists of the following data: for any Zariski open subset $U \subset X$ and $s \in H^0(U, L_{| U})$ a trivializing section, an upper semi-continuous (usc) function:
$$\phi_{s} = - \log \lVert s \rVert_{\phi} : U^{\an} \fl \R \cup \{- \infty \}$$
not identically $-\infty$, such that $\lVert f s \rVert_{\phi} = \lvert f \rvert \times \lVert s \rVert_{\phi}$ for any regular function $f \in H^0(U, \gO_U)$, and compatible with restriction of sections.
\end{defn}
\begin{ex} The following example will be particularly relevant for our purposes. Let $m \ge 1$, and $s_0 \in H^0(X, L^m)$ be a global (algebraic) section of some positive power of $L$; we associate to it the metric on $L$ given by:
$$\phi =m^{-1} \log \lvert s_0 \rvert,$$
i.e. for $s$ a local section of $L$, we have:
$$\lVert s \rVert_{\phi}(x) =  (\bigl\lvert \frac{s^m(x)}{s_0(x)} \bigl\rvert)^{1/m}.$$
This metric is singular precisely along the zero locus of $s_0$. 
\end{ex}
\begin{defn}
Let $X$, $Y$ be two $A$-schemes of finite type, and $f: Y \fl X$ be an $A$-morphism. If $L$ is a line bundle on $X$ endowed with a (singular) metric $\phi$, we define the pull-back metric $f^*\phi$ on $f^*L$ as follows: cover $X = \cup_{i \in I} U_i$ by Zariski open subsets, and choose a trivialization $s_i$ of $L$ on each $U_i$.
\\This induces an open cover $Y = \cup_{i \in I} V_i$ with $V_i = f^{-1}(U_i)$, and local trivializations of $f^*L$ by the sections $f^*s_i$ on $V_i$. We now set:
$$\lVert f^*s_i \rVert_{f^* \phi} := \lVert s_i \rVert_{\phi} \circ f.$$
\end{defn}
It is straightforward to check that this is independent on the choice of open cover and trivializations, and thus defines a metric $f^*\phi$ on $f^*L$.
\\We conclude this section with a discussion on the behaviour on metrics under base change. We assume that $A$ and $B$ are two Banach rings, together with a bounded ring homomorphism $A\fl B$, so that we have a continuous map $F: X^{\an}_B \fl X^{\an}$ for any scheme $X/A$ of finite type, see section \ref{sec def Berko}.
\\Let $L$ be a line bundle on $X$, and $L_B = L \otimes_{\gO_X} \gO_{X_B}$ the induced line bundle on $X_B$. Given a continuous metric $\phi$ on $L^{\an}$, we want to define a continuous metric $\phi_B$ on $L_B$ by a base change operation. To that purpose, cover $X = \cup_{i \in I} U_i$ by Zariski open subsets trivializing $L$, and set $U_{B, i} = U_i \times_A B$, which yields an open cover of $X_B$. If $s_i$ is a generator of the free $\gO_X(U_i)$-module $H^0(U_i, L)$, then $s_i \otimes 1$ is a generator of $H^0(U_{B, i}, L_B)$ over $\gO_{X_B}(U_{B, i})$, so that we naturally set, for $x \in U_{B, i}^{\an}$:
$$\lVert (s \otimes 1)(x) \rVert_{\phi_B} := \lVert s(F(x)) \rVert_{\phi}.$$
It is now a straightforward verification that $\phi \mapsto \phi_B$ defines a base change map from the set of continuous metrics on $(X^{\an}, L^{\an})$ to the set of continuous metrics on $(X^{\an}_B, L_B^{\an})$, which commutes with the usual operations of addition and scaling of metrics, as well as finite maxima.
\begin{ex} Let $A$ be a Banach ring, and let $x \in \M(A)$. Then we have a canonical morphism of Banach rings $A \fl \mathscr{H}(x)$, so that any continuous metric $\phi$ on $L$ induces by base change a continuous metric $\phi_x$ on $(X^{\an}_{\mathscr{H}(x)}, L^{\an}_{\mathscr{H}(x)})$. The metric $\phi_x$ can also be seen as the restriction of $\phi$ to the fiber $\pi^{-1}(x)$ of the structure map $\pi : X^{\an} \fl \M(A)$, by prop. \ref{prop gbr}.
\end{ex}
\subsection{Pluripotential theory over a field}
Let $(X,L)$ be a smooth polarized variety over $\C$. The class $\PSH(X, L)$ of semi-positive metrics on $L$ lies at the heart of (global) complex pluripotential theory, it is the class of singular metrics $\phi$ on $L$ whose curvature form $dd^c \phi$ is semi-positive in the sense of currents. We refer the reader for instance to \cite{DemB}, \cite{GZB} for a more thorough introduction.
\\Let $m \ge 1$ such that $mL$ is globally generated. Given a family $(s_0,...,s_N)$ of global sections of $mL$ without commons zeroes, one can associate to them the continuous semi-positive metric:
$$\phi = \frac{1}{2m} \log( \lvert s_0 \rvert^2+...+\lvert s_N \rvert^2),$$
which is none other than the pull-back of the standard Fubini-Study metric on $(\C \CP^{N}, \gO(1))$ via the holomorphic map:
$$x \mapsto [s_0(x):...:s_N(x)].$$
We call such a metric on $L$ a \emph{Fubini-Study} metric.
The following theorem, due to Demailly \cite{Dem} when $X$ is smooth and $L$ ample, highlights the importance of such metrics as basic building blocks of complex pluripotential theory:
\begin{theo}{\cite[thm. 7.1]{BE}} \label{theo berg exp} 
\\Let $X$ be a complex projective variety, $L$ a semi-ample line bundle on $X$, and $\phi \in \PSH(X, L)$ a semi-positive singular metric on $L$.
\\Then there exists a decreasing sequence $(\phi_j)_{j \in \N}$ of Fubini-Study metrics on $L$, converging pointwise to $\phi$.
\end{theo}
In particular, the class $\PSH(X, L)$ is the smallest class of singular metrics containing all Fubini-Study metrics, and that is stable under addition of constants, finite maxima and decreasing limits.
\\We now move to the case of a non-archimedean field $(K, \lvert \cdot \rvert)$, and assume as above that $X$ is a variety over $K$ endowed with a semi-ample line bundle $L$, with $mL$ globally generated. Following the general heuristic of replacing sums of squares with maxima in the non-archimedean world, a (tropical) Fubini-Study metric on $L$ is a continuous metric of the form:$$\phi =m^{-1} \max_{0 \le i \le N} \big( \log \lvert s_i \rvert + \lambda_i \big),$$
where $(s_0,...,s_N)$ is a family of global sections of $mL$ without common zeroes and the $\lambda_i$'s are real constants - unlike in the Archimedean case, the valuation $v_K : K^{\times} \fl \R$ need not be surjective, so we are allowing these constants to ensure the class of Fubini-Study metrics is stable by addition of constants. The constants are in fact not necessary when $K$ is non-trivially valued, but will be convenient for us to treat the case of trivial and non-trivial valuation in an uniform way.
\\Continuous, plurisubharmonic metrics $\phi$ on $L$ can now be defined, as in the complex case, by the positivity of their curvature current $dd^c \phi$, this is the approach taken in \cite{CLD}, where Chambert-Loir and Ducros develop a theory of real differential forms and currents on Berkovich spaces, paralleling the complex case. We will not use this approach in this paper, and rather plurisubharmonic metrics on $L$ in a way such that Demailly's regularization theorem still holds:
\begin{defn} Let $X$ be a variety over a non-archimedean field $K$, and $L$ a semi-ample line bundle on $X$. A singular metric $\phi$ on $L$ is semi-positive if and only if it can be written as the pointwise decreasing limit of a net $(\phi_j)_{j}$ of tropical Fubini-Study metrics.
\end{defn}
This is consistent with the approach of Chambert-Loir -  Ducros by \cite[thm. 7.14]{BE}.
\\In the sequel, we will define a class of semi-positive metrics on analytifications of schemes over a Banach ring $A$, and the Berkovich spectrum $\M(A)$ will have both an (open) Archimedean part and a non-archimedean part - that is , for $x \in \M(A)$, the complete residue field $\mathscr{H}(x)$ may be Archimedean or not. As a result, it is desirable to have a more uniform definition of Fubini-Study metrics, independent of the nature of the residue field. To that extent, if $X$ is a projective variety over either $\R$ or $\C$ and $L$ a semi-ample line bundle on $X$, we say a continuous metric $\phi$ on $L$ is a tropical Fubini-Study metric if it can written as:
$$\phi =m^{-1} \max_{0 \le i \le N} \big( \log \lvert s_i \rvert + \lambda_i \big),$$
where $(s_0,...,s_N)$ is a family of global sections of $mL$ without common zeroes and $(\lambda_0,...,\lambda_N)$ are real constants - which can always be absorbed in the $s_i$, so that they are allowed only for convenience. Note that over the complex numbers, any tropical Fubini-Study metric is psh in the usual sense. In fact, Demailly's regularization theorem still holds after replacing Fubini-Study metrics by tropical ones: 
\begin{theo} \label{theo psh1} Let $X$ be a projective complex variety, and $L$ a semi-ample line bundle on $X$. Then any semi-positive metric $\phi \in \PSH(X, L)$ can be written as the decreasing limit of a net of tropical Fubini-Study metrics.
\end{theo}
The converse is a straightforward consequence of the usual properties of $\PSH(X, L)$: any decreasing limit of tropical Fubini-Study metrics is psh. As a result, given any complete valued field $K$, we have the following uniform characterization of the class psh metrics on $L$: it is the smallest class of metrics that contains tropical Fubini-Study metrics, and that is stable under addition of constants, finite maxima and decreasing limits.
\begin{proof}
We set $\PSH^{\tau}(X, L)$ for the class of singular metrics that can be written as the decreasing limit of a net of tropical Fubini-Study metrics. The class $\PSH^{\tau}(X, L)$ is closed under decreasing limits by the proof of \cite[prop. 5.6]{BJtriv}.
\\Let $(s_{\alpha})_{\alpha \in A}$ be a finite family of sections of $mL$ without common zeroes, and $\phi$ be the associated $L^2$-Fubini-Study metric:
$$\phi = \frac{1}{2m} \log( \sum_{\alpha \in A} \lvert s_{\alpha} \rvert^2),$$
and set $\phi_{\alpha} = m^{-1} \log \lvert s_{\alpha} \rvert \in \PSH(X, L)$. Then we have $\phi = \chi((\phi_{\alpha})_{\alpha \in A})$, with $\chi(x) = \frac{1}{2m} \log(\sum_{\alpha \in A} e^{2mx_{\alpha}})$. It now follows from the proof of lemma \ref{prop homo} that $\phi$ is a decreasing limit of a sequence of tropical Fubini-Study metrics, hence $\phi \in \PSH^{\tau}(X,L)$. By Demailly's regularization theorem when $X$ is smooth and $L$ ample, and \cite[thm. 7.1]{BE} in the general case, any metric $\phi$ in $\PSH(X,L)$ can be written as the decreasing limit of metrics in $\PSH^{\tau}(X, L)$. Since the latter is closed under decreasing limits, we infer that $\phi \in \PSH^{\tau}(X, L)$, which concludes.
\end{proof}
The following example provides an alternative description of tropical Fubini-Study metrics over a discretely-valued field $K$ of equicharacteristic zero:
\begin{ex} \label{ex model}
Assume that $A =K$ is a non-trivially valued non-archimedean field, with valuation ring $R$ and residue field $k$. If $(X,L)$ is a polarized variety over $K$, one can define the class of \emph{model metrics} on $L$ as follows: for any normal, projective $R$-model $\X/R$ of $X$ and $\Ld$ a model of $mL$ on $\X$ for $m \ge 1$, define:
$$ \phi_{\Ld}(x) = m^{-1} \log \lvert s_{\Ld}(x) \rvert,$$
where $s_{\Ld}$ is a trivialization of $\Ld$ at the center $c_{\X}(v_x)$ of $x$. One directly checks that this defines a continuous metric on $L$, such that the lattice $H^0(\X, m\Ld) \subset H^0(X, mL)$ is the unit ball for the induced supnorm:
$$\lVert s \rVert_{L^\infty(\phi_{\Ld})} := \sup_{x \in X^{\an}} \; \lVert s (x) \rVert_{\phi_{\Ld}}$$whenever $\X_0$ is reduced.
\\It then follows from \cite[thm. 5.14]{BE} that model metrics associated to semi-ample models are the same as pure Fubini-Study metrics on $L$, i.e. Fubini-Study metrics where the constants are taken to be zero in the definition. As an easy consequence, model metrics are the same as differences of pure Fubini-Study metrics.
\end{ex}
\subsection{Tropical Fubini-Study metrics}
Throughout this section, $X$ is a (not necessarily proper) $A$-scheme of finite type over a Banach ring $A$, and $L$ is a semi-ample line bundle on $X$.
\\The discussion from the previous section motivates the definition of the following class of metrics, that will be the building blocks for our class of semi-positive metrics:
\begin{defn}
Let $L$ be a line bundle on $X$, and let $m \ge 1$ be an integer. A tropical Fubini-Study metric on $L$ is a (continuous) metric of the form:
$$\phi = m^{-1} \max_{j \in J} (\log \lvert s_j \rvert + a_j),$$
where $(s_j)_{j \in J}$ is a finite family of sections of $mL$ without common zeroes and $a_j \in \R$.
\\We write $\FS^{\tau}(L)$ for the set of tropical Fubini-Study metrics on $L$.
\\If $L= \gO_X$, we will simply say that $\phi$ is a Fubini-Study function on $X$, and write $\FS^{\tau}(X) = \FS^{\tau}(\gO_X)$.
\\Finally, if the constants $a_j$ are all zero in the above definition, we will say that $\phi$ is a \emph{pure} Fubini-Study metric.
\end{defn}
It follows from the definition that $\FS^{\tau}(L)$ is non-empty if and only $L$ is semi-ample.
\\The following properties of $\FS^{\tau}(L)$ are straightforward consequences of the definition:
\begin{prop}
\label{prop.FS} 
Let $X$ be an $A$-scheme of finite type and $L$ a line bundle on $X$. Then:
\begin{enumerate} 
\item if $\phi \in \FS^{\tau}(L)$ and $c \in \R$, then $\phi +c \in \FS^{\tau}(L)$;
\item if $\phi_1, \phi_2 \in \FS^{\tau}(L)$, then $\max \{ \phi_1, \phi_2 \} \in \FS^{\tau}(L)$;
\item if $\phi_i \in \FS^{\tau}(L_i)$ for $i=1,2$ then $\phi_1 + \phi_2 \in \FS^{\tau}(L_1 +L_2)$;
\item if $\phi$ is a metric on $L$ such that $m\phi \in \FS^{\tau}(mL)$ for $m \ge 1$, then $\phi \in \FS^{\tau}(L)$;
\item if $\phi_1, \phi_2 \in \FS^{\tau}(L)$ and $c_1, c_2 \in \Q_{\ge 0}$ with $c_1+c_2=1$, then $c_1 \phi_1 + c_2 \phi_2 \in \FS^{\tau}(L)$;
\item if $f: Y \fl X$ is a morphism of $A$-schemes of finite type and $\phi \in \FS^{\tau}(L)$, then $f^*\phi \in \FS^{\tau}(f^*L)$.
\item if $B$ is a Banach ring together with a bounded homomorphism $A \fl B$ and $\phi \in \FS^{\tau}(X, L)$, then the base change metric $\phi_B \in \FS^{\tau}(X_B, L_B)$.
\end{enumerate}
\end{prop}
We now introduce the following class of metrics, which usually play the role of smooth metrics in the non-archimedean case:
\begin{defn} Let $L$ be a line bundle on $X$.
\\A $\DFS$ (difference of Fubini-Study) metric on $L$ is a metric of the form $\phi = \phi_1 - \phi_2$, where $\phi_i \in \FS(L_i)$ for $i=1,2$, with $L=L_1-L_2$.
\\We write $\DFS(L)$ for the set of DFS metrics on $L$, and $\DFS(X) \subset \mathcal{C}^0(X^{\an})$ for the set of DFS functions on $\gO_X$, i.e. DFS metrics on $\gO_X$. 
\end{defn}
\begin{theo}
Assume that $X/A$ is projective. Then the $\Q$-vector space $\DFS(X)$ is dense in $\mathcal{C}^0(X^{\an})$.
\end{theo}
\begin{proof}
This is essentially the same proof as in \cite[thm. 2.7]{BJtriv}.
\\It follows easily from prop. \ref{prop.FS} that $\DFS(X)$ is a $\Q$-subvector space of $\mathcal{C}^0(X)$, stable under $\max$ and containing constant functions. Hence, since $X^{\an}$ is compact, by the Stone-Weierstrass theorem, it is enough to prove that $\DFS(X)$ separates points. 
\\Since $\DFS$ is stable by pullback, we may assume that $X = \CP_A^n$. Let $x \neq y \in X^{\an}$, then by considering a hyperplane not containing either $x$ or $y$, we may assume $x, y \in \A^{n, \an}_A$, which is by definition the set of semi-norms on $A[t_1,...,t_n]$ whose restriction $A$ belongs to $\M(A)$. As a result, $\lvert \cdot \rvert_x \neq \lvert \cdot \rvert_y$ implies that there exists a polynomial $f \in A[t_1,...,t_n]$ such that $\lvert f(x) \rvert \neq \lvert f(y) \rvert$; we will assume $\lvert f(x) \rvert < \lvert f(y) \rvert$.
\\Take homogeneous coordinates $z_0,...,z_n \in H^0( \CP^n_A, \gO(1))$ on $\CP^n_A$, such that $t_i = z_i/z_0$ on $\A^n$. We may write $f(t_1,...,t_n) = z_0^{-d} s$, with $s \in H^0( \CP^n_A, \gO(d))$.
\\Let $N \in \Z$, and $\lambda_0,..., \lambda_n \in \Z$, and set:
$$\psi = d \max_{0 \le j \le n} (\log \lvert z_j \rvert - \lambda_j),$$
which is an FS metric on $\gO(d)$, so that:
$$u = \max \{ \log \lvert s \rvert, \psi -N \} - \psi=  \max \{ \log \lvert s \rvert - \psi, -N \}$$
is a DFS function on $\CP^{n}_A$.
\\Then for $\lambda_0 =0$ and $\lambda_j$ large enough, we have $\psi(x) = d \log \lvert z_0(x) \rvert$ and $\psi(y) =d \log \lvert z_0(y) \rvert$.
\\Thus, for $N> - \log \lvert f (y) \rvert$, we have:
$$u(x) = \max \{ \log \lvert s(x) \rvert - d \log \lvert t_0(x) \vert, -N \} = \max \{ \log \lvert f(x) \rvert, -N \}$$
$$< \max \{ \log \lvert f(y) \rvert, -N \} = u(y),$$
whence the result.
\end{proof}
\begin{ex}
Let $Y = \Sp A$, so that $Y^{\an} =\mathscr{M}(A)$. Then a Fubini-Study function on $Y^{\an}$ is a continuous function of the following form:
$$\eta =m^{-1} \max_{\alpha \in B} (\log \lvert f_{\alpha} \rvert + \lambda_{\alpha}),$$
where $B$ is a finite set, the $f_{\alpha} \in A$ have no common zeroes and $\lambda_{\alpha} \in \R$. 
\\Now, if $X/A$ is a scheme of finite type and $X^{\an} \xrightarrow{\pi} \mathscr{M}(A)$ is its analytification, the function $\eta \circ \pi$ (that we will still write as $\eta: X^{\an} \fl \R$) is also an FS function on $X$.
\end{ex}
\subsection{Semi-positive metrics}
From now on, we will assume that $A$ is an integral Banach ring. Recall that $\M(A)^{\eta}$ is the subset of $\M(A)$ whose elements $\lvert \cdot \rvert_x$ have trivial kernel. In particular, the residue field $\mathscr{H}(x)$ of $\M(A)$ at $x$ is the completion of the fraction ring $\kappa$ of $A$ with respect to $\lvert \cdot \rvert_x$, so that $X_{\mathscr{X}(x)}$ is the flat base change of $X$ to $\mathscr{H}(x)$.
\begin{defn}
Let $X$ be a scheme of finite type over $A$, and $L$ a semi-ample line bundle on $X$. A plurisubharmonic (or semi-positive) metric $\phi$ on $L$ is a singular metric on $L$ that is the pointwise limit of a decreasing net of tropical Fubini-Study metrics on $L$, and such that $ \phi_x \not\equiv - \infty$, for all $x \in \M(A)^{\eta}$, where $\phi_x$ is the restriction of $\phi$ to $X_{\mathscr{H}(x)}^{\an}$.
\\We write $\PSH(X, L)$ or $\PSH(L)$ for the set of semi-positive metrics on $L$, and $\PSH(X)$ for the set of PSH functions on $X^{\an}$.
\end{defn}
Note that since our base ring $A$ is arbitrary, the space $\PSH(X)$ could be very large; for instance even for $X = \Sp A$, every non-zero element $a \in A$ induces a PSH function $\phi_a = \log \lvert a \rvert \in \PSH(X)$ on $X^{\an} = \mathscr{M}(A)$. Indeed, we have $\phi = \lim_j (\phi_j)_j$, with:
$$\phi_j = \max( \log \lvert a \rvert , \log \lvert 1-a \rvert -j),$$
and $\phi(x) = \log \lvert a(x) \rvert > - \infty$ whenever $x \in \M(A)^{\eta}$.
\\Let us point out that while the condition $ \phi_x \not\equiv - \infty$ for all $x \in \M(A)^{\eta}$ is natural in the setting of hybrid spaces (we will see later that it translates into finiteness of Lelong numbers) and spaces over $\M(\Z)$, it might be too strong in general - the analytification of $\CP^2$ over a trivially valued field contains points that are pluripolar and Zariski dense. Such a point lies in an affinoid domain $\M(A)$, which will then admit a psh function in the sense of \cite{BJtriv} that is $- \infty$ at a Zariski-dense point. 
\begin{defn} We write $\CPSH(X, L)$ for the set of continuous, plurisubharmonic metrics on $L$. It is endowed with the topology of uniform convergence on $X$.
\end{defn}
Note that $\FS^{\tau}(X, L) \subset \CPSH(X,L)$ by definition.
\begin{prop} \label{prop psh} The following properties hold:
\begin{enumerate}
\item if $\phi \in \PSH(L)$ and $c \in \R$, $\phi+c \in \PSH(L)$;
\item if $\phi_i \in \PSH(L_i)$ for $i=1,2$, $\phi_1+\phi_2 \in \PSH(L_1+L_2)$;
\item if $\phi_1, \phi_2 \in \PSH(L)$, then $\max \{ \phi_1, \phi_2 \} \in \PSH(L)$;
\item if $\phi$ is a singular metric on $L$ such that $m\phi \in \PSH(mL)$ for $m \ge 1$, then $\phi \in \PSH(L)$;
\item if $\phi_1, \phi_2 \in \PSH(L)$ and $c_1, c_2 \in \R_{\ge 0}$ with $c_1+c_2=1$, then $c_1 \phi_1 + c_2 \phi_2 \in \PSH(L)$;
\item if $(\phi_j)_j$ is a decreasing net in $\PSH(L)$ and $\phi= \lim_j \phi_j$ is such that $\phi_x \not\equiv - \infty$ for all $x \in \M(A)^{\eta}$, then $\phi \in \PSH(L)$;
\item if $X/A$ is proper and $(\phi_j)_j$ is a net in $\PSH(L)$ converging uniformly to $\phi$, then $\phi \in \PSH(L)$.
\end{enumerate}
\end{prop}
Note that the difference of two singular psh metrics does not make sense as a function in general, so that the last item means the following: if $(\phi_j)_j$ is a net in $\PSH(L)$, such that there exists a net of continuous functions $(f_j)_j$ in $\mathcal{C}^0(X^{\an})$ converging uniformly to zero and such that $\phi_j = \phi +f_j$, then $\phi \in \PSH(L)$.
\begin{proof} The first 5 items are straightforward consequences of the corresponding properties for $\FS^{\tau}$, stated in prop. \ref{prop.FS}; while (6) and (7) follows from \cite[lem. 4.6]{BJtriv}, (i) and (ii) respectively.
\end{proof}
The subset $\CPSH(X,L) \subset \PSH(X,L)$ is naturally endowed with the topology of uniform convergence; the next proposition states that for this topology, $\FS^{\tau}(L)$ is dense in $\CPSH$:
\begin{prop}{\cite[prop. 5.20]{BJtriv}}
Assume that $X/A$ is proper, and let $\phi$ be a continuous metric on $L$. Then $\phi \in \CPSH(X,L)$ if and only if there exists a net $(\phi_j)_{j}$ in $\FS^{\tau}(L)$ converging uniformly to $\phi$.
\end{prop}
Note that the natural topology of uniform convergence on the subset of continuous PSH metrics does not extend to $\PSH(L)$, so that it is unclear in this generality how to put a reasonable topology on $\PSH(L)$. It is however know how to put a topology on $\PSH(L)$ when $A$ is a discretely or trivially-valued field, see \cite{BFJ2}, \cite{BJv4}.
\\While the Fubini-Study metrics considered above are always continuous - hence bounded - one can also consider FS metrics with singularities:
\begin{ex} \label{ex fs sing} Assume that $L$ is a semi-ample, and let $s \in H^0(X, mL)$ for some $m \ge 1$ be a non-zero global section. Then $ \phi =m^{-1} \log \lvert s \rvert \in \PSH(X, L)$, since we may write $\phi$ as the decreasing limit of the following:
$$\phi_j =(md)^{-1} \max \big(\log \lvert s^d \rvert , \max_{\alpha \in A} (\log \lvert s_{\alpha} \rvert -j) \big),$$
where $(s_{\alpha})_{\alpha \in A}$ is a family of sections of $mdL$ without common zeroes. Moreover, if $x \in \M(A)$ is a Zariski-dense point, the base change of $s$ to $X_{\mathscr{H}(x)}$ is non-zero by flatness, so that $\phi_x \not\equiv - \infty$.
\\More generally, by prop. \ref{prop psh}, for semi-ample $L$ and any finite family $(s_{\alpha})_{\alpha \in A}$ of sections of $mL$, the metric:
$$\phi = m^{-1} \max_{\alpha \in A} ( \log \lvert s_{\alpha} \rvert + c_{\alpha})$$
is semi-positive, i.e. $\phi \in \PSH(X,L)$.
\end{ex}
\begin{rem}
Let us point out that unlike the class $\FS^{\tau}$, the class of plurisubharmonic metrics is not - strictly speaking - stable under base change. Indeed, if the ring homomorphism $A \fl B$ is not flat and $s \in H^0(X, mL)$ is a non-zero global section, it could very well happen that the section $s_B \in H^0(X_B, mL_B)$ is the zero section. As a result, the base change to $X_B$ of the psh metric $\phi =m^{-1} \log \lvert s \vert$ satisfies $\phi_B \equiv - \infty$. For instance, if $\phi \in \PSH(X, L)$ and $x \in \M(A)$ is not a Zariski-dense point, then the restriction $\phi_x$ to $X^{\an}_{\mathscr{H}(x)}$ may be identically $-\infty$. 
\\This occurs in the complex world as well, as a psh metric on a the total space of a holomorphic fibration $X \fl B$ may be identically $- \infty$ on certain fibers $X_b$. One easy way to remedy this is for instance to allow psh metrics to be identically $-\infty$, we choose not to as this psh metric would have to be treated separately in several proofs, making the exposition more cumbersome.
\end{rem}
\subsection{Examples}
We start with the case where $A=\C$ with the usual absolute value. The following statement is a mere reformulation of theorem \ref{theo psh1}:
\begin{theo} \label{theo psh2} Let $X$ be a projective complex variety, and let $L$ be a semi-ample line bundle on $X$.
\\Then $\PSH(X, L)$ is the space of plurisubharmonic metrics on $L$ in the sense of usual pluripotential theory.
\end{theo}
As mentioned above, for any Banach ring $A$, the above definitions for $X= \Sp A$ and $L= \gO_X$ yield a space $\PSH(X)$ of plurisubharmonic functions on $\mathscr{M}(A)$. We will compute this space in simple examples.
\\We start with the case where $A = k^{\hyb}$ is a hybrid field. Since $k$ is non-trivially valued, we may assume (up to scaling) that $\log \lvert k^{\times} \rvert \supseteq \Z$. In that case, the homeomorphism:
$$\lambda : \mathscr{M}(A) \xrightarrow{\sim} [0,1]$$
is in fact such that $\lambda$ is a Fubini-Study function on $X$, since $\lambda(x) = \log \lvert a \rvert_x$ for any $a \in k^{\times}$ such that $\log \lvert a \rvert =1$.
\\Conversely, any Fubini-Study function on $X$ is of the form:
$$\phi(x) = m^{-1} \max_{j \in J} \big(\log \lvert a_j \rvert_x + c_j\big),$$
for $a_j \in k^{\times}$ and $c_j \in \Z$. Since $\log \lvert a_j \rvert_x = \lambda(x) \log \lvert a_j \rvert$, $\phi$ is a finite maximum of affine functions, hence convex, and $\FS(X)$ contains all finite maxima of affine functions with rational coefficients. Taking decreasing limits (which have finite values everywhere since all points in $\M(k^{\hyb})$ are Zariski-dense), we conclude that the homeomorphism $\lambda$ identifies the space $\PSH(\mathscr{M}(k^{\hyb}))$ with the space of real-valued convex functions on the segment $[0,1]$. In particular, plurisubharmonic functions on $X$ are continuous away from the boundary of the interval. Note however that the function $\phi : \M(k^{\hyb}) \fl \R$ defined by $\phi(0)=1$, $\phi(\lambda) =0$ for $\lambda >0$ is also psh.
\\We now move on to the case $A=A_r$, the ring of Laurent series that are convergent for the hybrid norm on $\C$; recall that we have a canonical homeomorphism from the hybrid circle $C^{\hyb}(r) :=\M(A_r) \simeq \bar{\D}_r$ to the closed Euclidean disk. The following proposition asserts that away from the boundary of the closed disk, we may, after rescaling, identify psh functions on the hybrid circle with subharmonic functions on the punctured disk that have logarithmic growth at the puncture:
\begin{prop}
Let:
$$A_r = \{ f = \sum_{n \in \Z} a_n t^n \in \C((t))  \; / \; \lVert f \rVert_{\hyb} = \sum_{n} \lVert a_n \rVert_{\hyb} \; r^n < \infty \},$$
and write $\tau : \bar{\D}_r  \xrightarrow{\sim} C^{\hyb}(r)$ the homeomorphism from prop. \ref{cercle hyb}.
\\There exists an order-preserving, injective map:
$$\rho_r : \PSH(C^{\hyb}(r)) \longrightarrow \SH(\D_r) + \R \log \lvert t \rvert$$
$$ \phi \longmapsto \big( t \mapsto \log_r \lvert t \rvert \times (\phi (\tau(t)) \big)$$
for $t \neq 0$. Moreover, if $\phi$ is continuous, then $\phi(0) = \log r \times \nu_0( \rho_r(\phi))$ is a negative multiple of the (generalized) Lelong number of $\rho_r(\phi)$ at $0$.
\\Conversely, for any $r'>r$, there exists an order-preserving, injective map: 
$$ \rho_{r', r} : \SH(\D_{r'}) + \R \log \lvert t \rvert \longrightarrow \PSH(C^{\hyb}(r)),$$
$$\phi \longmapsto \big(\tilde{\phi} : t \mapsto \frac{\phi(\tau^{-1}(t))}{\log_r \lvert t \rvert} \big)$$
with $\tilde{\phi}(0) = \frac{\nu_{0}(\phi)}{\log r}$.
Finally, the composition $\rho_r \circ \rho_{r', r}$ is (up to a scaling factor) the usual restriction map.
\end{prop}
\begin{proof}
Let $\phi \in \FS^{\tau} (C^{\hyb}(r))$, and write:
$$\phi = m^{-1} \max_{\alpha \in A} (\log \lvert f_{\alpha} \rvert + c_{\alpha} ),$$
with $f_{\alpha} \in A_r$, so that in particular the formal series $f_{\alpha}$ induces a holomorphic function $F_{\alpha}$ on $\D_r$. Moreover, we have $\log \lvert f_{\alpha} \rvert (\tau(t)) = \frac{\log r}{\log \lvert t \rvert} \log \lvert F_{\alpha}(t) \rvert$ for $t \neq 0$.Up to shifting $\phi$ by a constant, we may assume that $\phi \le 0$ on $C^{\hyb}(r)$, so that each term in the maximum is nonpositive.
\\ Thus, defining $\rho_r(\phi) (t) = \log_r \lvert t \rvert \times \phi(\tau(t))$ for $t \neq 0$, we have:
$$\rho_r(\phi) (t) = \max_{\alpha \in A} \big(\log \lvert F_{\alpha}(t) \rvert + c_{\alpha}\log_r \lvert t \rvert \big),$$
since all the terms in the maximum have the same sign. This implies that $\rho_r(\phi)$ extends at $t=0$ as the sum of a subharmonic function on $\D_r$ and a multiple of $\log \lvert t \rvert$, and with Lelong number at zero:
$$\nu_0(\rho_r(\phi)) = \max_{\alpha \in A} (\ord_0(f_{\alpha}) + \frac{c_{\alpha}}{\log r}) = \frac{ \phi(0)}{\log r} \ge 0.$$
Now if $\phi \in \PSH(C^{\hyb}(r))$, it is finitely-valued at the Zariski-dense point $0$, so that up to shifting by a constant, we may assume that $\phi(0) =-1$. We now write $\phi$ as the decreasing limit of a net $(\phi_j)_j$ in $\FS^{\tau}$, with $\phi_j(0) \le 0$ for all $j$ large enough. By the computations above, the latter condition means precisely that $\rho_r(\phi_j)$ extends over zero as a subharmonic function. We then define the function $\rho_r(\phi) \in \SH(\D_r)$ as the decreasing limit of the $\rho_r(\phi_j)$, which is independent of the choice of decreasing sequence, since $\rho_r(\phi)$ is determined uniquely by $\phi$ outside $0$. If $\phi$ is furthermore continuous, then:
$$\phi(0) = \lim_{t \rightarrow 0} \phi(\tau(t)) = \lim_{t \rightarrow 0} \frac{\rho_r(\phi)(t)}{\log_r \lvert t \rvert} = \log r \times \nu_0((\rho_r(\phi)).$$
\\Conversely, let $r'>r$ and let $\phi \in \SH(\D_{r'})$. The fact that $\frac{\log r}{\log \lvert t \rvert} \cdot \phi \in \PSH(C^{\hyb}(r))$ follows from the more general theorem \ref{theo psh hyb}.
\\Finally it is clear that from the constructions that if $\phi \in \SH(\D_{r'}) + \R \log \lvert t \rvert$, then $\rho_r(\rho_{r', r}(\phi))= \phi_{| \D_r}$.
\end{proof}
We furthermore expect that the image of the restriction of $\rho_r$ to $\CPSH(C^{\hyb}(r))$ is the space of continuous subharmonic functions on $\D^*_r$, extending continuously to the boundary of the disk and with finite Lelong number at zero.
\begin{rem}
Let $\eta : C^{\hyb}(r) \fl \R$ be such that $\eta(0) = 1$ and $\eta \equiv 0$ outside zero. Then unsatisfyingly, $\eta \in \PSH(C^{\hyb}(r))$; we interpret this as the non-archimedean realization of the following phenomenon. Let:
$$\psi_j = \max ( \log \lvert t \rvert, -j),$$
which decrease to $\psi = \log \lvert t \rvert$. We have $\nu_0(\psi_j) \equiv 0$ since $\psi_j$ is bounded near $0$, while $\nu_0(\psi) = 1 > \lim_j \nu_0(\psi_j)$. Writing $\psi_j^{\hyb}$ and $\psi^{\hyb}$ the associated psh functions on $C^{\hyb}(r)$ for some $r \in (0,1)$, the jump of Lelong numbers along this decreasing sequence means that the non-archimedean data $\nu_0(\psi)$ attached to $\psi$ differs from the restriction $\psi^{\hyb}(0) := \lim_j \psi_j^{\hyb}(0) = \lim_j \nu_0(\psi_j)$ to the origin of the hybrid data associated to $\psi$. The point $0 \in C^{\hyb}(r)$ is in fact non-pluripolar (as it is Zariski-dense), and thus is not negligible in the sense of hybrid pluripotential theory.
\end{rem}
We also describe subharmonic functions on the Berkovich spectrum $\M(\Z)$:
\begin{prop}
Let $X= \M(\Z)$, and write $X = \bigcup_{p \in \mathcal{P} \cup \infty} I_p$ as the union of the p-adic and archimedean branches.
\\Then a continuous function $\phi : X \fl \R \cup \{ - \infty \}$ is psh if and only:
\begin{itemize}
\item for every prime number $p$, its restriction $\phi_p$ to the branch $I_p$ is convex, with negative outgoing slopes $s_p, \tilde{s}_p$ at $0$ and $+ \infty$ respectively, and value at infinity $\phi_p(\lvert \cdot \rvert_p^{\infty}) \in \R \cup \{-\infty\}$,
\item its restriction to the branch $I_{\infty}$ is convex and increasing, with positive slope at $0$,
\item the sum of slopes at zero $\sum_{p \in \mathcal{P}} s_p + s_{\infty} \ge 0$; in particular the sum $\sum_{p \in \mathcal{P}} -s_p < + \infty$.
\end{itemize}
In other words, the function $\phi$ is psh on $\M(\Z)$ if and only it is subharmonic in the usual sense on the $\R$-tree $\M(\Z)$.
\end{prop}
As a consequence, a point $x \in \M(\Z)$ is polar (that is, contained in $\{ \phi = -\infty \}$ for some $\phi \in \CPSH(X)$) if and only if it is the outer end of a p-adic branch.
\begin{proof}
Let $\phi \in \FS(X)$ be a Fubini-Study function, then there exists a family of integers $(n_{\alpha})_{\alpha \in A}$, with $\min_{\alpha \in A} v_p(n_{\alpha}) =0$ for every prime $p$, such that:
$$\phi =m^{-1} \max_{\alpha \in A} (\log \lvert n_{\alpha} \rvert + c_{\alpha}),$$
with $c_j \in \R$. Denoting $0 \in X$ the trivial absolute value, we have $\phi(0) =m^{-1} \max_{\alpha} c_{\alpha}$, and we write $A' \subset A$ the set of indices realizing the maximum. Set $n_1 = \gcd_{\alpha \in A'} n_{\alpha}$, and $n_2 = \lcm_{\alpha \in A'} n_{\alpha}$.
\\Under the homeomorphism $I_p \simeq [0, + \infty]$, we have:
$$\phi_p(\eps) =m^{-1} \max_{\alpha \in A} (-v_p(n_{\alpha}) \log p \cdot \eps + c_{\alpha}),$$
which shows that $\phi_p$ is a piecewise-affine convex function, and with slope at zero: 
$$s_p = (-\log p) \min_{\alpha \in A'} v_p(n_{\alpha}) = \log \lvert n_1 \rvert_p \le 0,$$
and constant for $\eps \gg1$ since there exists a $n_{\alpha}$ with $v_p(n_{\alpha})=0$, so that the slope at infinity $\tilde{s}_{p} =0$. 
A similar computation on the archimedean branch shows that $\phi_{\infty}$ is also convex, with:
$$s_{\infty}= \sum_{p \in \mathcal{P}} \log p \cdot \big( \max_{\alpha \in A'} v_p(n_{\alpha}) \big) = -\sum_{p \in \mathcal{P}} \log \lvert n_2 \rvert_p = \log \lvert n_2 \rvert_{\infty} \ge 0,$$
so that the sum of slopes:
$$s_{\infty} + \sum_{p \in \mathcal{P}} s_p = \log \bigg\lvert \frac{n_2}{n_1} \bigg\rvert_{\infty}$$
is positive. Now if $\phi$ is a continuous psh function on $X$ (with possibly infinite values), then it is the uniform limit of Fubini-Study functions near zero, hence the only if part by taking decreasing - hence locally uniform by Dini's lemma - limits.
\\Conversely, let $\phi : X \fl \R \cup \{-\infty \}$ be a function satisfying the above properties. We assume that $\phi(0) = 0$, so that $\phi_p \le 0$ for every prime $p$, and $\phi_{\infty} \ge 0$. We divide the argument in several steps.
\\ \emph{Step 1:} Assume that $\phi_p \equiv 0$ for every prime $p$, and $\phi_{\infty} : [0,1] \fl \R$ is a continuous, increasing convex function, with $\phi_{\infty}(0)=0$. 
\\For any $a \ge 0$ and $b \le 0$ two real numbers, the function $\psi =\psi_{a, b}$ on $X$ defined by $\psi_{\infty}(x) = ax+b$ and $\psi_p \equiv 0$ for every prime $p$ is psh (although not necessarily continuous at $0$). Indeed, choosing a prime $q$ and writing $a = \lim_j r_j \log q$ as the decreasing limit of rational multiples of $\log q$, we see that $\psi = \lim_j \max \{ r_j \log \lvert q \rvert +b, 0 \}$ as a decreasing limit - the max is realized by $0$ on every $p$-adic branch, even when $p=q$ due to the assumption on $a, b$.
\\As a result, writing $\phi_{\infty}$ as a decreasing limit of piecewise affine convex functions of the form $\max_{\alpha \in A_j} (a_{\alpha} x +b_{\alpha})$ as above (recall that $\min_{\alpha} a_{\alpha} \ge 0$, and $\max_{\alpha} b_{\alpha} =0$ since $\phi_{\infty}(0)=0$), we get that:
$$\phi = \lim_j \max_{\alpha \in A_j} \psi_{a_{\alpha}, b_{\alpha}}$$
as a decreasing limit, and $\phi \in \PSH(X)$.
\\ \emph{Step 2:} we now regularize the $\phi_p$. By our assumptions on the slopes $s_p$, $\tilde{s}_p$, for every prime $p$, we may find a decreasing sequence $(\phi_{j,p})$ of convex functions on $I_p$ of the form:
$$\phi_{j,p}(\eps) = m^{-1} \max_{\alpha \in A_p} (-\ell_{\alpha} \log p \cdot \eps +c_{\alpha})$$
converging to $\phi_p$, where the $\ell_{\alpha}$'s are positive integers and such that $\phi(0) =\phi_p(0) = \max_{\alpha} c_{\alpha} =0$. Write $s_{j, p} = \min_{\alpha, c_{\alpha} =0} (-\ell_{\alpha} \log p)$ the slope at zero of $\phi_{j, p}$; by continuity near zero, the $s_{j, p}$ decrease to $s_p$. We use the same notation for the (singular) Fubini-Study function:
$$\phi_{j, p} =m^{-1} \max_{\alpha \in A} ( \log \lvert p^{\ell_{\alpha}} \rvert +c_{\alpha}),$$
by straightforward computation we see that the restriction of $\phi_{j, p}$ to the branch $I_{\infty}$ is linear:
$$\phi_{j, p}(\lvert \cdot \rvert^x_{\infty}) = -s_{j, p} x.$$
For $k \in \N$, let $\mathcal{P}_k =\{ 2,..,p_k \}$ be the $k$ smallest primes. We set:
$$\phi_{k} = \sum_{p \in \mathcal{P}_k} \phi_{p, k},$$
which is psh on $X$, and such that $(\phi_k)_k$ decreases to $\phi$ on each $p$-adic branch.
\\ \emph{Step 3:} this does not yield the desired outcome on the archimedean branch: the restrictions of $\phi_k$ to the archimedean branch are increasing to $x \mapsto s x$, where we have set $s := - \sum_{p} s_p \le s_{\infty}$. However, the convergence is uniform on $I_{\infty} \simeq [0,1]$, so that (after extraction of a subsequence) we may find a decreasing sequence $(\eps_k)_k$ of constants going to zero, such that on $I_{\infty}$, the $\phi'_{k, \infty} = \phi_{k, \infty} + \eps_k$ decrease to $x \mapsto sx$. As a result, the psh functions $\phi'_k = \phi_k + \eps_k$ decrease on $X$, and the limit $\phi'$ satisifies $\phi'_p = \phi_p$, and $\phi'_{\infty}(x) = sx$ for $x \in [0,1]$.
\\ \emph{Step 4:} by step $1$, the function $\psi : X \fl \R$ such that $\psi_p \equiv 0$ for every prime $p$, and $\psi_{\infty}(x) = \phi_{\infty}(x) - sx$ is psh, since $s\le s_{\infty}$. As a result, $\phi = \phi' + \psi$ is indeed subharmonic on $\M(\Z)$.
\end{proof}
\section{PSH metrics on hybrid spaces} \label{sec psh hyb}
Throughout this section, we let $X \xrightarrow{\pi} \D^*$ be a degeneration of projective complex manifolds, endowed with a semi-ample line bundle $L$. We fix $r \in (0,1)$, and write $X^{\hyb} \xrightarrow{\pi_{\hyb}} \bar{\D}_r$ the associated hybrid space, which is the analytification of $X$ viewed as an $A_r$-scheme, see section \ref{sec hyb}.
\\We will use the $t$-adic valuation on $K = \C((t))$ normalized so that $\lvert t \rvert = r$, and write:
$$\log_r \lvert t \rvert = \frac{\log \lvert t \rvert}{\log r},$$
which is non-negative on $\bar{\D}_r$.
\begin{defn}
Let $X \xrightarrow{\pi} \D^*$ be a degeneration of projective complex manifolds, and let $L$ be a line bundle on $X$. A hybrid (continuous) metric $\phi$ on $L$ is a singular (resp. continuous) metric on $L^{\hyb}$ in the sense of the previous section, viewing $X$ as an $A_r$-scheme.
\\We write $\PSH(L^{\hyb})$ for the set of hybrid semi-positive metrics on $L$.
\end{defn}
Using the explicit description of the hybrid space from prop. \ref{topo hyb}, we are able to describe more concretely continuous hybrid metrics on $L$:
\begin{prop} \label{prop metric hyb}
Let $X$ be a degeneration of complex manifolds, and $L$ be a line bundle on $X$. A continuous hybrid metric $\phi$ on $L$ is equivalent to the data of a continuous family of metrics $(\phi_t)_{t \in \bar{\D}_r^*}$ on the $L_t$'s, together with a continuous metric $\phi_0$ on $L^{\an}$, such that the following holds: for every Zariski open subset $U \subset X$ and any non-vanishing section $s \in H^0(U, L_{| U})$, the function:
$$ z \mapsto \frac{\log \lVert s(z) \rVert_{\phi_t}}{\log_r \lvert t \rvert}$$
on $U^{\hol}$ (with the Euclidean topology) extends as a continuous function to $U^{\hyb}$ via $x \in U^{\an} \mapsto \log \lVert s(x) \rVert_{\phi_0}$.
\end{prop}
\begin{proof}
If $\lvert t \rvert >0$, we have a canonical homeomorphism $\beta_t : X_t \xrightarrow{\sim} X^{\hyb}_t$, such that for any (local) regular function on $X$, we have:
$$\lvert f(\beta(z)) \rvert = \lvert f(z) \rvert^{\frac{\log r}{\log \lvert t \rvert}},$$
and a homeomorphism $\beta_0 : X^{\an} \xrightarrow{\sim} \pi_{\hyb}^{-1}(0).$
\\Hence if $\phi$ is a continuous hybrid metric on $L$, it induces a continuous family $\{\phi_{t}\}_{t\in \bar{\D}_r}$ of metrics $\phi_t$ on $L_t$, obtained as follows: if $U \subset X$ is a Zariski open and $s$ a trivialization of $L$ on $U$, set:
 $$ \lVert s(z) \rVert_{\phi_{t}} := \lVert s(\beta_t(z)) \rVert_{\phi}^{\frac{\log \lvert t \rvert}{\log r}},$$
 for $z \in U \cap X_t$. The fact that this defines a continuous metric on $L_t$ is an easy consequence of the above equality for functions. 
\\Similarly, the formula:
$$\log \lVert s(x) \rVert_{\phi_0} := \log \lVert s(\beta_0(x)) \rVert_{\phi}$$ 
defines a continuous metric $\phi_0$ on $L^{\an}$.
\\The fact that the data of $\phi_0$ and $\{ \phi_{t} \}_{t \in \bar{\D}^*_r}$ recovers $\phi$ uniquely is clear.
\end{proof}
We will sometimes write the above relation more loosely as:
 $$\phi_{t} = \phi_{| X_t} = \frac{\log \lvert t \rvert}{\log r} \phi_{|X^{\hyb}_{t}},$$
 where the left-hand side of the second equality lives on the complex fiber $X_t$ and the right-hand side lives on the Berkovich analytic space $X^{\hyb}$.
 \begin{ex}
 Let $s \in H^0(X,mL)$ be a global section of $mL$, and let:
 $$\phi = m^{-1} \log \lvert s \rvert \in \PSH(L^{\hyb})$$
 be the associated (singular) hybrid metric. Then one checks directly that for any $t \neq 0$ such that $s_t \not \equiv 0$, the metric $\phi_{t} \in \PSH(X_t, L_t)$ is equal to $m^{-1} \log \lvert s_t \rvert$, where $s_t = s_{| X_t} \in H^0(X_t, mL_t)$.
 \\Now let $c \in \R$ be a constant, and let:
 $$\phi = m^{-1} (\log \lvert s \rvert +c)  \in \PSH(L^{\hyb}).$$
 Then we have $\phi_{t} = m^{-1} (\log \lvert s \rvert +c \log_r \lvert t \rvert) \in \PSH(X_t, L_t).$
 \end{ex}
\begin{rem} \label{rem pure} Let $\phi = m^{-1} (\log \lvert s \rvert +\frac{c}{\log r})  \in \PSH(L^{\hyb})$ as in the example above. Then $\phi$ is the decreasing limit of the $(\phi_j)_{j \in \N}$, where:
$$\phi_j = m^{-1} (\log \lvert s \rvert +\frac{c_j}{\log r}),$$
where $(c_j)_{j \in \N}$ is a sequence of rational numbers decreasing to $c$. Up to replacing $m$ by a high enough multiple depending on the denominator of $c_j$, we may furthermore assume that $c_j \in \Z$, so that:
$$\phi_j = m_j^{-1} \log \lvert t^{c_j} s \rvert$$
is a pure Fubini-Study metric. Since finite maxima commute with decreasing limits, any tropical Fubini-Study metric can be written as a decreasing limit of pure Fubini-Study metrics. As a consequence, any psh metric on $L^{\hyb}$ can be written as a decreasing limit of \emph{pure} Fubini-Study metrics.
\\Let us emphasize that the key point here is that the constant function $\frac{c}{\log r}$ for $c \in \Q$ can be written as $\log \lvert f \rvert$ for some non-zero $f \in A_r$, which need not hold over a general Banach ring $A$ - it fails for instance for $A=\Z$.
\end{rem}
\subsection{Bergman metrics on the hybrid space}
Let $X \fl \D^*$ be a degeneration of complex manifolds, and $L$ a semi-ample line bundle on $X$.
In the sequel, it will sometimes be convenient for us to work with singular $L^2$-Bergman metrics in the complex world, i.e. metrics on $L$ of the form:
$$\phi = \frac{1}{2m} \log \big( \sum_{\alpha \in A} \lvert s_{\alpha} \rvert^2 \big),$$
where $(s_{\alpha})_{\alpha \in A}$ is a finite set of non-zero sections in $H^0(X, mL)$, possibly with common zeroes. It is clear that $\phi \in \PSH(X, L)$, and the following proposition asserts that $\phi$ extends naturally as a metric $\phi \in \PSH(L^{\hyb})$, replacing the square-norm with maxima at the non-archimedean limit:
\begin{prop} \label{prop homo sing} Let $(s_{\alpha})_{\alpha \in A}$ be finite family of global sections of $mL$ for $m \ge 1$, and set:
$$\phi_t = \frac{1}{2m} \log \big( \sum_{\alpha \in A} \lvert s_{\alpha, t} \rvert^2 \big),$$
$$\phi_0 = m^{-1} \max_{\alpha \in A} \log \lvert s_{\alpha} \rvert.$$
Then this data defines a semi-positive metric $\phi \in \PSH(X^{\hyb}, L^{\hyb})$, which we call the hybrid Bergman metric associated to the family $(s_{\alpha})_{\alpha \in A}$.
\end{prop}
\begin{proof}
 We may and will assume that $mL$ is basepoint-free, up to replacing $mL$ by $dmL$ and the $s_{\alpha}$'s by their $d$-th power, for $d$ large enough.
\\We thus choose a basepoint-free set $(s_{\alpha})_{\alpha \in B}$ of sections of $mL$, where $B = A \sqcup A'$. Set:
$$\phi_{j,t} := \frac{1}{2m} \log \big( \sum_{\alpha \in B} e^{b_{\alpha,j}} \lvert s_{\alpha,t} \rvert^2 \big),$$
and:
$$\phi_{j, 0} := m^{-1} \max_{\alpha \in B} (\log \lvert s_{\alpha} \rvert + b_{\alpha,j}),$$
where $b_{\alpha, j} = 0$ if $\alpha \in A$ and $b_{\alpha, j} = -j$ for $\alpha \in A'$. Then the fact that this defines a continuous psh metric $\phi_j$ on $L^{\hyb}$ follows from the following lemma, applied to the convex function $\chi(x) = \frac{1}{2} \log \big( \sum_{\alpha} e^{2x_{\alpha} + b_{\alpha}} \big)$. Finally, the $\phi_j$'s clearly decrease to $\phi$ by construction, so that $\phi \in \PSH(L^{\hyb})$ by prop. \ref{prop psh}.
\end{proof}
\begin{lem} \label{prop homo} Let $\sigma \subset \R^N$ be the standard simplex, i.e.:
$$\sigma = \{ (x_1,...,x_N) \in (\R_{\ge 0})^N / \sum_{i=1}^N x_i =1 \},$$
and write $\mathbf{1} = (1,...,1) \in \R^N$. We let $\chi : \R^N \fl \R_{\ge 0}$ be a convex function, such that:
\begin{itemize}
\item $\chi(x+c \mathbf{1}) = \chi(x)+c$ for any $c \in \R$, $x \in \R^N$,
\item the function $\big(\chi - \max (x_1,...,x_N) \big)$ is bounded on $\R^N$.
\end{itemize}
For any set $s_1,...,s_N$ of sections of $mL$ without common zeroes and $\phi_i =m^{-1} \log \lvert s_i \rvert$, the hybrid metric $\phi_{\chi}$ on $L$ defined, using prop. \ref{prop metric hyb}, by:
$$\phi_{\chi, t} = \chi (\phi_{1,t},..., \phi_{N, t}) \in \CPSH(X_t, L_t)$$
and:
$$\phi_{\chi,0} = m^{-1} \max_{i \le N} \log \lvert s_i \rvert \in \FS^{\tau}(X^{\an}, L^{\an})$$
is a semi-positive continuous hybrid metric on $L$.
\end{lem}
\begin{proof}
The assumption $\chi(x +c \mathbf{1}) = \chi(x) +c$ ensures that $\phi_{\chi}$ is compatible with multiplication of sections by functions, hence each $\phi_{\chi, t}$ defines a continuous metric on $L_t$, for $t \in \bar{\D}_r$. We now prove that $\phi_{\chi}$ is continuous on $X^{\hyb}$ using prop. \ref{prop metric hyb}, we assume that $m=1$ for convenience. Given a nowhere-vanishing section $s$ of $L$ on a Zariski open $U \subset X$, we have:
$$\frac{ \log \lvert s(z,t) \rvert_{\phi_{\chi, t}}}{\log_r \lvert t \rvert} = \frac{ \log \lvert s(z,t) \rvert - \chi(\log \lvert s_1(z,t) \rvert,..., \log \lvert s_N(z,t) \rvert)}{\log_r \lvert t \rvert},$$
so that writing $f_i = \log \lvert \frac{s_i}{s} \rvert$, the condition $\chi(x +c \mathbf{1}) = \chi(x) +c$ implies that:
$$\frac{ \log \lvert s(z,t) \rvert_{\phi_{\chi, t}}}{\log_r \lvert t \rvert} = \frac{ \chi\big(f_1(z,t),...,f_N(z,t)\big)}{\log_r \lvert t \rvert} = \frac{\max \{ f_1,..., f_N \}}{\log_r \lvert t \rvert} + \eps(t)$$
by our assumption on $\chi$, where $\lvert \eps(t) \rvert \le \frac{C}{\log_r \lvert t \rvert}$ for some constant $C>0$. As a result, since away from the zero locus of $s_i$ the function $\frac{f_i}{\log_r \lvert t \rvert}$ extends continuously to $U^{\hyb}$ via $x \mapsto \log \lvert \frac{s_i}{s} \rvert (x)$ on $U^{\hyb}_0$, we infer that $\phi_{\chi}$ is indeed a continuous metric on $(X^{\hyb}, L^{\hyb})$. 
\\We now prove that $\phi_{\chi}$ is semi-positive. Using \cite[prop. 2.6]{PS22}, there exists a sequence $(\chi_j)_{j \in \N}$ of piecewise-affine convex functions decreasing to $\chi$, written as:
$$\chi_j = \max_{\alpha \in A_j} \big(u_{\alpha} + c_{\alpha}\big),$$
with $u_{\alpha} \in \sigma \cap \Q^N$ and $c_{\alpha} \in \R$, and such that the convex hull $\Conv(u_{\alpha})_{\alpha \in A_j} = \sigma$. We write $\underline{\phi} = (\phi_1,..., \phi_N) \in \FS(L^{\hyb})^N$, and set:
$$\phi_j = \chi_j(\underline\phi)= \max_{\alpha \in A} \big( \langle u_{\alpha}, \underline\phi \rangle + c_{\alpha} \frac{ \log \lvert e \rvert}{\log r}\big),$$
where $e \in A_r$ is $\exp(1)$ viewed a constant power series. Since the $u_{\alpha}$ are rational and lie in the standard simplex, the $ \langle u_{\alpha}, \underline\phi \rangle$ are tropical Fubini-Study metrics on $L^{\hyb}$ as convex linear combinations thereof, so that $\phi_j \in \FS^{\tau}(L^{\hyb})$, and:
$$\phi_{j, t} = \max_{\alpha \in A_j} \big( \langle u_{\alpha}, \underline{\phi_t} \rangle + c_{\alpha} \big),$$
while:
$$\phi_{j, 0} = \max_{\alpha \in A_j} \big( \langle u_{\alpha}, \underline{\phi_0} \rangle \big),$$
since $\log \lvert e \rvert \equiv 0$ on $X^{\an} = \pi_{\hyb}^{-1}(0)$.
Moreover, the $(\phi_{j,t})_j$ decrease to $\phi_t$ by the choice of $(\chi_j)_j$, while for all $j$, $\phi_{j, 0} = \max_{i \le N} \phi_i$, as the formula $\max_{\alpha \in A_j} \langle u_{\alpha} , x \rangle = \max (x_1,...,x_N)$ holds for all $x \in \R^N$ since $\Conv(u_{\alpha})_{\alpha \in A_j} = \sigma$.
\\All in all, $(\phi_j)_j$ is a decreasing sequence in $\FS^{\tau}(L^{\hyb})$ converging pointwise to $\phi_{\chi}$, whence $\phi_{\chi} \in \PSH(L^{\hyb})$.
\end{proof}
\begin{ex} \label{ex homo}
Let $X = \C \CP^N \times \D^*$ with the standard polarization, and let $\phi_t = \phi_{FS}$ be the usual complex analytic Fubini-Study metric on $X_t$, i.e. :
$$\phi_{FS} = \log ( \lvert z_0 \rvert^2 +...+ \lvert z_N \rvert^2),$$
where the $z_i$ are standard homogeneous coordinates on $\CP^n$.
Let furthermore:
$$\phi_0= \max \{ \lvert z_0 \rvert,..., \lvert z_N \rvert \}$$ 
be the non-archimedean Fubini-Study metric on $\CP_K^{n, \an}$.
\\Then the metric $\phi_{FS}$ on $\gO(1)$ on $X^{\hyb}$ obtained by gluing the two above metrics is a continuous semi-positive metric. Note that with our definition, this metric is \emph{not} a Fubini-Study metric on $\gO(1)$, since we would have to work with the $\max$ instead of the square norm on the complex fibers.
\end{ex}
\begin{ex} Let $(X, L)$ be a complex polarized variety, and set $V = H^0(X, mL)$, for some $m>0$ such that $mL$ is globally generated. We let $\mathcal{N}(V)$ be the space of Hermitian norms of $V$, which is a symmetric space, as fixing a reference norm yields an identification:
$$\mathcal{N}(V) = \GL(V) / U(V).$$
Let $N = \dim V$, and $\mathbf{e} = (e_1,...,e_N)$ be a basis of $V$. Then for each tuple $(\lambda_1,..., \lambda_N)$ of real numbers, define the associated hermitian norm:
$$\lVert \sum_{i=1}^N {a_i} e_i \rVert^2 = \sum_{i=1}^N \lvert a_i \rvert^2 e^{-2\lambda_i}.$$
This yields an embedding:
$$\iota_{\mathbf{e}} : \R^N \hookrightarrow \mathcal{N}(V),$$
whose image consists precisely of the norms diagonalized by the basis $\mathbf{e}$. The image \\$\mathbb{A}_{\mathbf{e}}(\R^N): = \iota_{\mathbf{e}}(\R^N)$ is called the apartment associated to the basis $\mathbf{e}$.
We let $I \subset \R$ be an interval (not necessarily bounded), and $\gamma : I \fl \mathcal{N}(V)$ a geodesic. Then there exists a basis $\mathbf{e}$ of $V$ such that $\gamma(I) \subset \mathbb{A}_{\mathbf{e}}$, and an affine map $\alpha : I \fl \R^N$ such that:
$$\gamma = \iota_{\mathbf{e}} \circ \alpha.$$
More concretely, writing $\alpha(y) = ( \alpha_1 y + \beta_1,..., \alpha_n y+\beta_n)$, we have:
$$\lVert \sum_{i=1}^N a_i e_i \rVert_{y}^2 = \sum_{i=1}^N \lvert a_i \rvert^2 e^{-2 \beta_i -2 \alpha_i y}.$$
Assume that $I = (y_0, + \infty)$ for some $y_0 \in \R$, so that $\gamma$ is a geodesic ray. Then we see easily that for any non-zero $v \in V$, the limit $- \frac{\log \lVert v \rVert_y }{y}$ exists and is equal to $- \log \lVert v \rVert_{\gamma^{\NA}}$, where $\gamma^{\NA}$ is the non-archimedean norm defined by:
$$\lVert \sum_{i=1}^N {a_i} e_i \rVert_{\gamma^{\NA}} = \max_{i \le N} (\lvert a_i \rvert_0 \; e^{-\alpha_i}),$$
where $\lvert \cdot \rvert_0$ is the trivial norm on $\C$. Furthermore, two geodesic rays induce the same non-archimedean norm at infinity if and only they are parallel, i.e. $(\alpha - \alpha')$ is constant. Thus, the space $\mathcal{N}(V)^{\NA}$ of non-archimedean norms on $V$  can be interpreted as the space of asymptotic directions in $\mathcal{N}(V)$.
\\We now explain how each geodesic ray $\gamma : [0, + \infty) \fl \mathcal{N}(V)$ gives rise to a psh hybrid metric on $L$, which we call the associated Bergman metric.
More generally, let $W \subset V$ be a basepoint-free subspace, i.e. a subvector space such that the sections in $W$ have no common basepoints,  and let $\gamma : [0, + \infty) \fl \mathcal{N}(W)$ be a geodesic ray. By the above discussion, there exists a basis $\mathbf{e} = (s_1,..., s_N)$ of $W$, and tuples $\alpha, \beta  \in \R^N$ such that:
$$\lVert \sum_{i=1}^N a_i s_i \rVert^2 = \sum_{i=1}^N \lvert a_i \rvert^2 e^{-2\alpha_i y}.$$
The hybrid Bergman metric associated to $\gamma$, defined by:
$$\phi_{\lambda} =(2m)^{-1}(\log  \sum_{i=1}^N \lvert s_i \rvert^2 e^{\frac{2\alpha_i}{\lambda}}),$$
where $\phi_{\lambda}$ is the pull-back to $X^{\hol}$ of $\phi_{| X^{\hyb}_{\lambda}}$ via the rescaling of the absolute value (see the proof of prop. \ref{prop metric hyb}), and:
$$\phi_{0} = m^{-1} \max_{i \le N} ( \log \lvert s_i \rvert + \alpha_i),$$
is a continuous psh metric on $L^{\hyb}$, as follows from prop. \ref{prop homo} applied to the convex function $\chi(x_1,...,x_N) = (2m)^{-1} \log( \sum_{i=1} e^{2m x_i}).$
\\Note that $\phi_{\lambda}$ is the classical Bergman metric associated to the norm $\lVert \cdot \rVert_y$, while $\phi_0$ is the non-archimedean Bergman metric associated to the non-archimedean norm $\gamma^{\NA}$.
\end{ex}
\subsection{The non-archimedean limit of a psh family} \label{sec psh na}
We still let $X \fl \D^*$ be a degeneration of complex varieties, but now assume that our line bundle $L$ on $X$ is relatively ample. We let $\phi \in \PSH(X, L)$ be a semipositive metric on $L$. Following the general heuristic of viewing non-archimedean geometry as the asymptotic limit of Kähler geometry, we will explain how under a reasonable growth condition on $\phi$, the family of generic Lelong numbers of $\phi$ along prime vertical divisors on models of $X$ naturally induces a non-archimedean psh metric $\phi^{\NA} \in \PSH(X^{\an}, L^{\an})$. This construction is due to \cite{BBJ} in the isotrivial case and \cite{Reb} in the general case. Let us start with a definition.
\begin{defn}{\cite[lem. 2.3.2]{Reb}} \label{def log g} The metric $\phi \in \PSH(X,L)$ has logarithmic growth at $t=0$ if one of the following equivalent conditions are satisfied:
\begin{itemize}
\item there exists a normal model $(\X, \Ld)$ such that $\phi$ extends as a psh metric on $\Ld$,
\item for any normal model $(\X, \Ld)$, there exists $a \in \R$ such that $\phi + a \log \lvert t \rvert$ extends as a psh metric on $\Ld$,
\item there exists a normal model $(\X, \Ld)$ and a smooth metric $\phi_0$ on $\Ld$ such that $\sup_{X_t} (\phi - \phi_0) \le C \frac{\log \lvert t \rvert}{\log r}$ for some constant $C>0$.
\end{itemize}
\end{defn}
We will now explain how each psh ray $\phi$ on $(X, L)$ with logarithmic growth induces a psh metric $\phi^{\NA}$ on $(X^{\an}, L^{\an})$. We fix a model $(\X, \Ld)$ of $(X, L)$ with $\Ld$ nef, so that the associated model metric $\phi_{\Ld} \in \PSH(L^{\an})$ is semi-positive.
\\Let $v \in X^{\an}$ be a divisorial valuation, so that there exists a model $\X'$ of $X$ and a prime divisor $E \in \Div_0(\X')$ such that $v =v_E= (-\log r) \times b_E^{-1} \ord_E$, with $b_E = \ord_E(t)$. We may furthermore assume that $\X'$ dominates $\X$, via a morphism $\rho : \X' \fl \X$.
\\By logarithmic growth, there exists $a \in \R$ such that the metric $ \phi_a :=\phi + a \log \lvert t \rvert$ extends as a psh metric on $(\X', \rho^*(\Ld))$.
We choose a psh metric $\phi_E$ on $\X'$ with divisorial singularities along $E$, i.e. $\phi_E = \log \lvert z_E \rvert +O(1)$ locally, where $z_E$ is an equation of $E$. We define:
$$\nu_E(\phi) =  \sup \{ c \in \R / \phi_a \le c \phi_E + O(1) \}-a$$
the generic Lelong number of $\phi$ along $E$ \cite{BFJ0} - which is easily seen to be independent of $a$ - and set:
$$\psi^{\NA}(v) := \frac{\log r}{b_E} \nu_E(\phi).$$
\begin{theo}{\cite[thm. 6.2]{BBJ}, \cite[thm. 3.3.1]{Reb}} \label{theo ext phina}
\\The function $\psi^{\NA} : X^{\xdiv} \fl \R$ admits a unique lower semi-continuous extension to $X^{\an}$, still denoted by $\psi^{\NA}$, and the metric $\phi^{\NA}$ on $L^{\an}$ defined by:
$$\phi^{\NA} := (\phi_{\Ld} + \psi^{\NA}) \in \PSH(X^{\an}, L^{\an}),$$
is a semi-positive metric on $L^{\an}$.
\end{theo}
\begin{ex} \label{ex bdd model} Let $\phi \in \PSH(X,L)$ be a locally bounded psh metric, and assume that there exists a normal model $(\X, \Ld)$ of $(X, L)$ with $\Ld$ nef, such that $\phi$ extends to $\X$ as a locally bounded, psh metric on $\Ld$. Then we have $\psi_{\NA} \equiv 0$, so that $\phi^{\NA} = \phi_{\Ld}$. In other words, the non-archimedean limit of a metric extending without singularities as a semi-positive metric on a nef model $(\X, \Ld)$ is simply the associated model metric.
\end{ex}
\begin{ex} \label{ex Lelong TFS}
Let:
$$\phi = \max_{\alpha \in A} \big(\log \lvert s_{\alpha} \rvert + c_{\alpha} \big)$$
be a tropical Fubini-Study metric on $L^{\hyb}$, and let $v \in X^{\xdiv}$ be a divisorial valuation. We let $(\X, \Ld)$ be a model of $(X,L)$ such that $v = v_E$ for some $E \subset \X_0$, and the $s_{\alpha}$'s extend as holomorphic sections of $\Ld$ on $\X$. Then:
$$\nu_E(\phi) =(\phi_0- \phi_{\Ld})(v_E),$$
where $\phi_0$ is the induced FS metric on $X^{\an}$. Indeed, since the $s_{\alpha}$'s are holomorphic sections of $\Ld$, we have:
$$(\phi_0- \phi_{\Ld})(v_E) = \max_{\alpha \in A} \big( v_E(\frac{s_{\alpha}}{s_{\Ld}}) +c_{\alpha}),$$
where $s_{\Ld}$ is a trivialization of $\Ld$ at the generic point of $E$, while:
$$\phi_t = \max_{\alpha \in A} \big(\log \lvert s_{\alpha} \rvert + c_{\alpha} \log_r \lvert t \rvert \big)$$
is easily seen to have Lelong number:
$$\nu_E(\phi) = \min_{\alpha \in A} \big( \ord_E(\frac{s_{\alpha}}{s_{\Ld}}) +\frac{c_{\alpha}b_E}{\log r} \big)$$
along $E$. We infer from this that $\psi^{\NA} = (\phi_0 - \phi_{\Ld})$, so that when $\phi$ is a tropical Fubini-Study metric on $L^{\hyb}$, $\phi^{\NA} = \phi_0$ is simply its restriction to $X^{\an}$.
\end{ex}
We are now ready to prove Theorem A: the metric $\phi^{\NA}$ in fact extends $\phi$ as a semi-positive metric on the hybrid space, and up to shrinking the disk, every \emph{continuous} semi-positive metric on $L^{\hyb}$ arises in this way.
\begin{theo}\label{theo psh hyb}
Let $(X,L)$ be a polarized degeneration of complex polarized varieties over $\D^*$,  and $\phi = (\phi_t)_{t \in \D^*} \in \PSH(X,L)$ be a semipositive metric on $L$, with logarithmic growth at $t=0$. Then the metric $\phi^{\hyb}$ on $L^{\hyb}$ defined by setting:
$$\phi^{\hyb}_{|X_t} = \phi_t;$$
$$\phi^{\hyb}_0 = \phi^{\NA},$$
is semi-positive, i.e. $\phi^{\hyb} \in \PSH(X^{\hyb}, L^{\hyb})$.
\\Conversely, let $\phi^{\hyb} \in \CPSH(X^{\hyb}, L^{\hyb})$ and $\eps>0$, and set:
$$\phi_t = \phi_{| X_t}$$
for $\lvert t \rvert < r- \eps$.
Then $(\phi_t)_{t \in \D_{r-\eps}^*}$ is a psh metric on the restriction of $L$ to $X_{| \D^*_{r- \eps}}$, with logarithmic growth at $0$, and such that $\phi^{\NA} = \phi^{\hyb}_0$.
\end{theo}
The proof of the theorem will be provided in section \ref{sec proof psh}.
\\Note that the continuity assumption cannot be removed: in the case where $X$ is a point, the function $\eta$ such that $\eta(0)=1$ and $\eta \equiv 0$ on $C^{\hyb} \setminus 0$ is psh on the hybrid circle, but $\eta(0)$ is not the Lelong number at zero of the induced subharmonic function on the punctured disk. This says essentially that the point $0 \in C^{\hyb}$ is non-pluripolar, so that it is "large" in the sense of pluripotential theory: psh metrics are not uniquely determined by their restriction outside zero. It has however dense complement, so that it is negligible topologically, and thus \emph{continuous} psh metrics are determined by their restriction outside zero.
\\The next proposition characterizes, given $\phi \in \PSH(X, L)$, the set of semi-positive metrics on $L^{\an}$ that arise as restrictions to $X^{\an}$ of hybrid metrics extending $\phi$, under a finite-energy assumption. Let $\phi \in \PSH(L^{\an})$, and assume that $\phi$ lies in the class $\mathcal{E}^1(L^{\an})$ of finite-energy metrics on $L$ - we won't give the definition here, and refer to \cite[§6]{BFJ1}, note however that $\mathcal{E}^1(L^{\an})$ contains all continuous psh metrics on $L$. Then it follows from \cite{Reb} that after a suitable choice of boundary data, there exists a canonical extension $\phi^{\hyb} \in \PSH(L^{\hyb})$ of $\phi$ to the hybrid space, which also lies in the class $\mathcal{E}^1(L) \subset \PSH(X,L)$ of  fiberwise-finite energy metrics - for all $t \neq 0$, $\phi_t$ is a finite-energy metric on $L_t$. The extension is canonical in the sense that it is relatively maximal in the sense of pluripotential theory \cite{Klim}, and the maximal extension mapping: $\mathcal{E}^1(L^{\an}) \fl \mathcal{E}^1(L)$ defined this way is an isometric embedding for the Darvas metric \cite{Dar} on $\mathcal{E}^1(L)$ and its non-archimedean analog on $\mathcal{E}^1(L^{\an})$.
\begin{prop} \label{prop semic phina}
Let $\phi \in \PSH(L^{\hyb})$, write $\phi_0 \in \PSH(L^{\an})$ its restriction to $X^{\an}$ and $\phi^{\NA} \in \PSH(L^{\an})$ the non-archimedean metric associated to $\phi_{| X}$, as defined by theorem \ref{theo ext phina}. Then the inequality:
$$\phi_0 \ge \phi_{\NA}$$
holds on $X^{\an}$. 
\\Conversely, let $\phi \in \mathcal{E}^1(L)$ be a psh metric of fiberwise-finite energy. Then for any $\psi \in \mathcal{E}^1(L^{\an})$ such that $\psi \ge \phi^{\NA}$, there exists a psh extension $\tilde{\phi} \in \mathcal{E}^1(L^{\hyb})$ of $\phi$ to $L^{\hyb}$ satisfying $\tilde{\phi}_0 = \psi$.
\end{prop}
\begin{proof}
Write $\phi = \lim_j \phi_j$ as the decreasing limit of a net in $\FS^{\tau}(L^{\hyb})$. It follows from example \ref{ex Lelong TFS} that for all $j$, the equality $\phi^0_j = \phi^{\NA}_j$ holds, hence $\phi_0$ is the decreasing limit of the $\phi^{\NA}_j$. Since $\phi_0$ and $\phi^{\NA}$ are psh on $L^{\an}$, they are determined by their (finite) values on divisorial points, so that it is enough to prove that if $v_E$ is a divisorial valuation on $X$, the inequality:
$$\lim_j (\psi^{\NA}_j(v_E)) \le \psi^{\NA}(v_E)$$
holds, which follows from semi-continuity of Lelong numbers.
\\Conversely, assume that $\phi \in \mathcal{E}^1(L)$ and $\psi \in \mathcal{E}^1(L^{\an})$, with $\psi \ge \phi^{\NA}$. By \cite{Reb} and theorem \ref{theo psh hyb}, there exists $\psi^{\hyb} \in \mathcal{E}^1(L^{\hyb})$ such that $\psi^{\hyb}_0 = \psi$. We set:
$$\eta_j = \max( \psi^{\hyb}, \phi^{\hyb} -j \log \lvert e \rvert),$$
then the decreasing limit of the $(\eta_j)_j$ is a psh hybrid metric, restricting to $\psi$ on $X^{\an}$ and to $\phi$ on $X$.
\end{proof}
\begin{ex} Let $(A, L) \fl \D^*$ be a polarized, maximal degeneration of abelian varieties. We let $\omega_t \in c_1(L_t)$ be the flat Kähler metric on $L_t$, then there exists a family of smooth metrics $\phi_t \in \PSH(X_t, L_t)$ - called the cubic metrics - such that $\omega_t = dd^c \phi_t$. Then it follows from the proof of \cite[thm. 4.13]{GO} that $\phi \in \PSH(X, L)$ and has logarithmic growth at $t=0$. Moreover, the associated non-archimedean metric $\phi^{\NA}$ is computed explicitly in \cite[thm. 4.3]{Liu}, and \cite[thm. 4.13]{GO} states that the induced hybrid metric $\phi^{\hyb}$ is in fact continuous, i.e. $\phi^{\hyb} \in \CPSH(A^{\hyb}, L^{\hyb})$.
\end{ex}
We also prove that given a continuous hybrid metric, it is enough to test its plurisubharmonicity outside zero. 
\begin{prop} \label{prop psh cont} Let $\phi \in \PSH(X, L)$, and assume $\phi$ extends as a continuous metric on $L^{\hyb}$, that we still denote $\phi$. Then the extension is semi-positive, i.e. $\phi \in \CPSH(L^{\hyb})$.
\end{prop}
\begin{proof} By theorem \ref{theo psh hyb}, it is enough to prove that $\phi_0 = \phi^{\NA}$. By continuity of $\phi$ and the following lemma, we have:
$$\phi_0(v_E) = \phi^{\NA}(v_E)$$
for every divisorial valuation $v_E \in X^{\xdiv}$. Thus, after fixing a reference model metric $\phi_{\Ld}$, for every snc model $\X$ on which $\phi_{\Ld}$ is determined, we have $(\phi_0 - \phi_{\Ld}) \circ \rho_{\X} = (\phi_{\NA} - \phi_{\Ld}) \circ \rho_{\X}$ since those two continuous functions agree on the rational points of $\Sk(\X)$, and the result follows from \cite[prop. 7.6]{BFJ2}.
\end{proof}
\begin{lem} Let $\phi \in \PSH(X, L)$ with logarithmic growth and $v_E \in X^{\xdiv}$ a divisorial valuation. Let $(\X, \Ld)$ be an snc model of $(X, L)$ such that $v_E$ is determined on $\X$, and $\phi$ extends as a psh metric on $\Ld$. We fix a bounded reference metric $\phi_0$ on $\Ld$, and write $\phi^{\NA} = \phi_{\Ld} + \psi^{\NA}$ as in theorem \ref{theo ext phina}.
\\There exists a sequence $(z_j)_{j \in \N}$ in $X$ converging to $v_E$ in $X^{\hyb}$, such that the sequence $(\frac{(\phi-\phi_{0})(z_j)}{\log_r \lvert t \rvert})_j$ converges to $\psi^{\NA}(v_E)$.
\end{lem}
\begin{proof}Recall that up to a negative scaling factor, $\psi^{\NA}(v_E)$ is the generic Lelong number of the psh function $(\phi- \phi_{0})$ along $E$, hence is equal to the Lelong number of $(\phi- \phi_{0})$ at a very general point of $E$ \cite{BFJ0}. Thus, we may choose a point $z_{\infty} \in E$ such that $z_{\infty}$ is not contained in any other irreducible component of $\X_0$, and such that:
$$\psi^{\NA}(v_E) = \frac{ \log r}{b_E} \nu_{z_{\infty}}(\phi- \phi_{0}).$$
Choose a sequence $(z_j)_j$ in $X$ converging to $z_{\infty}$ inside $\X$. Then by construction $(\frac{(\phi-\phi_{0})(z_j)}{\log_r \lvert t \rvert})_j$ converges to $\psi^{\NA}(v_E)$, and $(z_j)_j$ converges to $v_E$ in $\X^{\hyb}$ (see \cite[def. 2.3]{BJ}), hence in $X^{\hyb}$, which concludes the proof.
\end{proof}
The following question is taken from Favre \cite[question 1]{Fav}:
\begin{qu} Let $\phi \in \CPSH(L)$ be a continuous, semi-positive metric on $L$, and assume that $\phi^{\NA}$ is a continuous metric on $L^{\an}$.
\\Then is it true that $\phi^{\hyb} \in \CPSH(L^{\hyb})$ ?
\end{qu}
In view of the proof of theorem \ref{theo psh hyb}, this amounts to proving that assuming $\phi^{\NA} \in \CPSH$, we have an estimate of the form:
$$ \phi_m - \phi \le \eps_m \big\lvert \log \lvert t \rvert \big\rvert^{-1},$$
on $X$ where $\phi_m$ are the Bergman kernels regularizing $\phi$, and $\eps_m \xrightarrow[m \rightarrow \infty]{} 0$ is independent of $t$. Such a bound seems difficult to attain without a uniform estimate on the oscillation of $\phi$.
\subsection{Proof of theorem \ref{theo psh hyb}} \label{sec proof psh}
This section is devoted to the proof of theorem \ref{theo psh hyb}. If $\rho>0$, we will write $\D_{\rho} := \{ \lvert t \rvert < \rho \}$ the open disk of radius $\rho$, and if $\X$ is a model of $X$, $\X_{\rho}:= \X_{| \D_{\rho}}$. 
\\We let $\phi \in \PSH(X, L)$ be a psh metric on $L$ with logarithmic growth, so that after choosing an ample model $(\X, \Ld)$ of $(X, L)$, there exists $c \in \R$ such that $\phi_c = \phi+ c \log \lvert t \rvert$ extends as a psh metric on $\Ld$ - to alleviate notation, we will still denote $\phi_c$ by $\phi$. 
The basic idea of the proof of theorem \ref{theo ext phina} is that if $\mathscr{I}_m :=\mathscr{I}(m \phi)$ is the multiplier ideal of the psh metric $m \phi$ on $m\Ld$, then the sequence of piecewise-affine functions $m^{-1} \phi_{\mathscr{I}(m\phi)}$ on $X^{\an}$ decrease pointwise to the relative potential $(\phi^{\NA} -\phi_{\Ld})$ on $X^{\an}$. However for $m \gg 1$, up to a controlled error term, the sheaf $\gO_{\X}(m \Ld \otimes \mathscr{I}_m)$ is relatively globally generated on $\X$, so that the $m^{-1} \phi_{\mathscr{I}(m\phi)}$ are a sequence of $\phi_{\Ld}$-psh functions on $X^{\an}$, hence $\phi \in \PSH(X^{\an}, L^{\an})$.
\\We will roughly apply the same idea here, except we will also have to regularize $\phi$ itself by a sequence of psh metrics with analytic singularities along $\mathscr{I}(m \phi)$ (see def. \ref{def ansing}). The procedure we apply to produce such a sequence is standard in complex pluripotential theory and goes back to the work of Demailly \cite{Dem}, so that we merely outline the proof here and rather refer the reader to the appendix \ref{appendix psh} for the technical details. 
\\We let $\psi \in \PSH(\X, \Ld)$ be a smooth metric, whose curvature form $\omega := dd^c \psi$ is a Kähler form on $\X$. We choose $\eps \in (0, 1-r)$ and $m_0 \in \N$ such that for all $m \ge 0$, the sheaf $\gO_{\X}((m+m_0)\Ld \otimes \mathscr{I}_m)$ is globally generated over $\X_{r+\eps}$ (see prop. \ref{prop glob gen}). We write:
$$\psi_{m, m_0} = \big( m \phi + m_0\psi \big) \in \PSH(\X, (m+m_0)\Ld),$$
so that $\mathscr{I}_m = \mathscr{I}(m\phi) = \mathscr{I}(\psi_{m, m_0}) $ is the multiplier ideal of the psh metric $\psi_{m, m_0}$ on $(\X, (m+m_0)\Ld)$. 
\\We will regularize $\phi$ by the Bergman metrics associated to the multiplier ideal $\mathscr{I}_m$. More explicitly, set $V_{m, m_0} := H^0(\X_{r+\eps}, (m+m_0)\Ld \otimes \mathscr{I}_{m})$ and define $ \mathscr{H}_{m, m_0} \subset V_{m, m_0}$ as the following Hilbert space:
$$\mathscr{H}_{m, m_0} = \{ s \in V_{m, m_0} / \lVert s \rVert^2 := \int_{\X_{r+\eps}} \lvert s \rvert^2_{\psi_{m, m_0}} \omega^{n+1} < \infty \}.$$
For every couple $(m, m_0)$, we may choose a Hilbert basis $\mathscr{B}_{m, m_0} = (s_{m, m_0, l})_{l \in \N}$ of $ \mathscr{H}_{m, m_0}$, and we now set:
$$\phi_{m, m_0} = \frac{1}{2(m+m_0)} \log ( \sum_{l \in \N} \lvert s_{m, m_0, l} \rvert^2),$$
and:
$$\phi^{\NA}_{m, m_0} =\phi_{\Ld}+ (m+m_0)^{-1} \phi_{\mathscr{I}_m}.$$
It is clear that $\phi_{m, m_0} \in \PSH(\X, \Ld)$, and $\phi_{m, m_0}^{\NA} \in \PSH(X^{\an}, L^{\an})$. We claim that this defines a semi-positive hybrid metric on $L$:
\begin{prop} \label{prop sum psh}
For every $m \in \N_{>0}$, the hybrid metric $\phi^{\hyb}_{m, m_0}$ on $L$ defined by:
$$\phi^{\hyb}_{m, m_0, t} = (\phi_{m, m_0})_{| X_{t}}$$
and:
$$\phi^{\hyb}_{m, m_0, 0} = \phi^{\NA}_{m, m_0}$$
is semi-positive, i.e. $\phi^{\hyb}_{m, m_0} \in \PSH(X^{\hyb}, L^{\hyb})$.
\end{prop}
\begin{proof} 
For $q \in \N$, set:
$$\phi_{m, m_0, q} = \frac{1}{2(m+m_0)} \log ( \sum_{l \le q} \lvert s_{m, m_0, l} \rvert^2 ),$$
and:
$$ \phi^{\NA}_{m, m_0, q} =(m+m_0)^{-1} \max_{l \le q} (\log \lvert s_{m, m_0, l} \rvert).$$
It follows from prop. \ref{prop homo sing} that this defines a semi-positive hybrid metric $\phi^{\hyb}_{m, m_0, q} \in \PSH(X^{\hyb}, L^{\hyb})$, we will prove that the $(\phi^{\hyb}_{m, m_0, q})_{q \in \N}$ converge uniformly to $\phi^{\hyb}_{m, m_0}$ on $X^{\hyb}$.
\\By prop. \ref{prop som part}, the $\phi_{m, m_0, q}$ converge uniformly to $\phi_{m, m_0}$ on $\X_{r+ \eps}$ as $q \rightarrow \infty$, so that it remains to prove uniform convergence over $X^{\an}$.
\\We have that 
$$\phi^{\NA}_{m, m_0, q} =  \phi_{\Ld} + (m+m_0)^{-1} \phi_{\mathscr{J}_q},$$
where $\mathscr{J}_q = \mathscr{I} \big( (s_{m, m_0, l})_{l \le q} \big) \equiv \mathscr{I}_m$ over $\X_{r+\eps}$ for $q \gg 1$ by the strong noetherian property for coherent sheaves and global generation (see the proof of prop. \ref{prop som part}), which concludes by Dini's lemma.
\end{proof}
We now conclude the proof of theorem \ref{theo psh hyb}. After extracting a subsequence, the $(m+m_0)^{-1} \phi_{m, m_0}$ decrease to $\phi$ on $\X_{r+\eps}$ by theorem \ref{theo reg psh} and its proof, while the $\phi^{\NA}_{m, m_0}$ decrease to $\phi^{\NA}$ on $X^{\an}$ by the proof of \cite[thm. 3.3.1]{Reb}. This proves that $\phi^{\hyb} \in \PSH(X^{\hyb}, L^{\hyb})$.
\\For the (partial) converse, if $\phi \in \CPSH(L^{\hyb})$, and $s \in H^0(U, L)$ is a local trivialization, then the function on $U$:
$$z \mapsto \frac{\log \lvert s(z) \rvert_{\phi_t}}{\log \lvert t \rvert}$$
extends continuously to $X^{\hyb}$ by prop. \ref{prop metric hyb}, hence is bounded. This proves that $\phi$ induces a psh metric with logarithmic growth on $L$, hence we can define $\phi^{\NA}$ as above. By prop. \ref{prop semic phina}, we have $\phi^{\NA} \le \phi_0$, while the semi-continuity of the hybrid metric induced by $(\phi, \phi^{\NA})$ implies $\phi_0 = \lim_{t \rightarrow 0} \phi_t \le \phi^{\NA}$, which concludes.
\\This can also be deduced using the following:
\begin{lem} The map:
$$( \cdot)^{\NA} : \CPSH(L^{\hyb}) \fl \CPSH(L^{\an})$$
is well-defined and continuous with respect to the topologies of uniform convergence on both spaces.
\end{lem}
\begin{proof}This follows from the fact that $( \cdot)^{\NA}$ is order-preserving \cite[thm. 3.3.1]{Reb}, and that $(\phi+c)^{\NA} = \phi^{\NA} +c$ for $c \in \R$.
\\Indeed, writing $\phi \in \CPSH(L^{\hyb})$ as the uniform limit of a net $(\phi_j)_j$ in $\FS^{\tau}$, the fact that $(\cdot)^{\NA}$ is order-preserving implies that the $(\phi_j^{\NA})_j$ converge uniformly to $\phi^{\NA}$, which is then continuous.
\end{proof}
It now follows that $(\cdot)^{\NA}$ and the restriction map to $X^{\an}$:
$$(\cdot)_{0} : \CPSH(L^{\hyb}) \fl \CPSH(L^{\an})$$
are both continuous, and coincide on the dense set $\FS^{\tau}$ by example \ref{ex Lelong TFS}, and thus coincide on $\PSH(L^{\hyb})$. This concludes the proof of theorem \ref{theo psh hyb}.
\subsection{The isotrivial case}
Let $(X,L)$ be a projective complex variety, and write $(X_0^{\an}, L_0^{\an})$ the associated Berkovich space obtained by endowing the field of complex numbers with the trivial absolute value. The latter analytic space has proven itself to be a powerful tool in Kähler geometry, and in particular has been central in the variational proof of the Yau-Tian-Donaldson conjecture by Berman-Boucksom-Jonsson \cite{BBJ}.
\\Without going into the details, the proof involves the study of various convex functionals on the space $\PSH(X, L)$, and relating their slopes at infinity on geodesic rays with the corresponding non-archimedean functionals on the space $\PSH(X^{\an}_0, L^{\an}_0)$.
\begin{defn} Let $(X, L)$ be a normal projective complex variety. A family $(\phi_y)_{y >0}$ of psh metrics on $L$ is called a psh ray if and only the $\bS^1$-invariant metric:
$$\Phi(x, t) := \phi_{- \log \lvert t \rvert} (x)$$
on $(X \times \D^*, p_1^*L)$ is psh.
\end{defn}
Let $\phi_0$ be a smooth, positively curved reference metric on $L$. If $(\phi_y)_{y >0}$ is a psh ray on $X$, then the function $y \mapsto \sup_{X} (\phi_y-\phi_0)$ is convex, so that its slope at infinity:
$$p_{\max} := \lim_{y \rightarrow \infty} \frac{ \sup_X (\phi_y - \phi_0)}{y}$$
exists in $\R \cup \{ +\infty \}$ and is independent of the choice of $\phi_0$. It is immediate that $p_{\max} < + \infty$ if and only there exists $C>0$ such that $\sup_X \phi_y \le C y$ (with slight abuse of notation), in which case we will say that the ray $(\phi_y)_{y \in (y_0, + \infty)}$ has \emph{linear growth}. This is easily seen to be equivalent to the fact that the psh metric $\Phi$ on $X \times \D^*$ has logarithmic growth at zero.
\\Following the general heuristic of viewing non-archimedean geometry as the asymptotic limit of Kähler geometry, each psh ray $(\phi_y)_y$ on $(X, L)$ with linear growth induces a psh metric $\phi^{\NA}$ on $(X_0^{\an}, L_0^{\an})$, defined as follows. We fix a smooth reference metric $\phi_0$ on $L$, whose curvature form $\omega_0 = dd^c \phi_0$ is a Kähler form on $X$, and write $\psi_y := (\phi_y - \phi_0) \in \PSH(X, \omega_0)$.
\\By linear growth, there exists $a \in \R$ such that the function:
$$\Psi_a(x, t) = \psi_{- \log \lvert t \rvert} (x) + a \log \lvert t \rvert$$
on $X \times \D^*$ is bounded from above near $X \times \{0 \}$, hence extends as a quasi-psh function on $X \times \D$, that we still denote by $\Psi_a$. Now if $v \in X_0^{\an}$ is a divisorial valuation on $X$, and $w= \gamma(v) \in (X_K)^{\an}$ denotes the Gauss extension of $v$ to the base change $X_K :=X \times_{\C} K$, one defines $w(\Psi_a)$ as the generic Lelong number of $\Psi_a$ along the center of $w$, as in section \ref{sec psh na}. In other words, if $\phi$ is a psh ray on $X$ and $v \in X^{\xdiv}$, then:
$$\phi^{\NA}(v) = \Phi^{\NA}(\gamma(v)),$$

\begin{theo}{\cite[thm. 6.2]{BBJ}} The function $\psi^{\NA} : X^{\xdiv} \fl \R$ extends continuously to $X^{\an}$, and the metric $\phi^{\NA}$ on $L^{\an}$ defined by:
$$\phi^{\NA} = (\psi^{\NA} + \phi^{\triv}) \in \PSH(X^{\an}, L^{\an})$$
is plurisubharmonic.
\end{theo}
This is essentially a special case of thm. \ref{theo ext phina}, since the trivial metric $\phi^{\triv}$ is the restriction to the Gauss section (see section \ref{sec iso}) of the model metric $\phi_{\Ld} \in \PSH(X_K^{\an}, L_K^{\an})$, where $\Ld = p_1^*L$ is the trivial model, living on $\X = X \times \D$.
\begin{prop} Let $r \in (0,1)$. Then the metric $\phi^{\hyb}$ on $L^{\hyb}$ defined by:
$$\phi^{\hyb}_0 = \phi^{\NA},$$
$$\phi^{\hyb}_{| X_{\lambda}} = \lambda \phi_{1/\lambda}$$
is semi-positive on $X^{\hyb}_0$. Moreover, we have:
$$\Phi^{\hyb} = F^* \phi^{\hyb},$$
where $F : X^{\hyb}_K \fl X^{\hyb}_0$ is the base change map from prop. \ref{prop bc hyb}.
\end{prop}
\begin{proof} The equality $\Phi^{\hyb} = F^* \phi^{\hyb}$ is straightforward from the construction of $\phi^{\hyb}$.
\\To prove that $\phi^{\hyb}$ is psh, the argument is the same as in the proof of theorem \ref{theo psh hyb}, except we need the sections $s_{m, m_0}$ from the proof of prop. \ref{prop sum psh} to be equivariant with respect to the $\bS^1$-action on $X \times \D$. To achieve this, with the notation of the previous section for $\ell \in \Z$, we let $\mathscr{H}_{m, m_0, \ell}\subset \mathscr{H}_{m, m_0}$ be the space of $\ell$-equivariant sections, i.e. sections $s$ such that $(e^{i\theta})^*s = e^{i \ell \theta} \cdot s$. Then $\mathscr{H}_{m, m_0}$ is the completion of $\bigoplus_{\ell \in \Z} \mathscr{H}_{m, m_0, \ell}$, so that we may choose a Hilbert basis of $\mathscr{H}_{m, m_0}$ adapted to the weight decomposition. The rest of the proof of theorem \ref{theo psh hyb} carries out without changes after replacing $\log \lvert s_{m, m_0, \ell} \rvert$ on $X^{\an}_0$ by $(\log \lvert s_{m, m_0, \ell} \rvert -\ell)$ for $s_{m, m_0, \ell} \in \mathscr{H}_{m, m_0, \ell}$, so that we omit the details.
\end{proof}
\section{The Monge-Ampère operator} 
\label{sec MA}
In this section, we discuss families of Monge-Ampère measures associated to a continuous psh metric on an analytic space over a Banach ring, and explain how to deduce theorem \ref{theo Favre intro} from \cite{Fav}.
\subsection{The case of a valued field}
Let $K$ be a complete valued field, $X$ an $n$-dimensional projective scheme over $K$, and let $L_1,..., L_n$ be semi-ample line bundles on $X$. By Ostrowski's theorem, either $K = \R$ or $\C$ with (a power of) the usual absolute value, or $K$ is non-archimedean.
\\Let us start by assuming that $K = \C$ and that $X$ is smooth. It then follows from the seminal work of Bedford-Taylor \cite{BT} that the mixed Monge-Ampère pairing:
$$(\phi_1,...,\phi_n) \mapsto dd^c \phi_1 \wedge ... \wedge dd^c \phi_n,$$
\emph{a priori} defined when each $\phi_i$ is a smooth Hermitian metric on $L_i$, actually extends in a unique way to semi-positive, locally bounded metrics. The pairing was then further extended to semi-positive singular metrics by Boucksom-Eyssidieux-Guedj-Zeriahi \cite{BEGZ}. More precisely, there exists a class $\mathcal{E}^1(X, L) \subset \PSH(X,L)$ of finite-energy metrics on $(X,L)$, such that the above mixed Monge-Ampère pairing extends uniquely to a multilinear, measure-valued pairing on $\mathcal{E}^1(X,L_1) \times ... \times \mathcal{E}^1(X,L_n)$. Note that the space $\mathcal{E}^1(X, L)$ contains in particular $\CPSH(X,L)$.
\\We now assume $K$ is non-archimedean, and that $\phi_i \in \DFS(X, L_i)$ for every $i=1,...,n$ are differences of Fubini-Study metrics. One can then associate to the $\phi_i$'s a signed Radon measure, the mixed Monge-Ampère measure $\MA(\phi_1,...,\phi_n)$, with similar properties as in the complex analytic case.
\begin{ex} \label{ex MA model} Assume that $K=k((t))$ is a discretely-valued field of characteristic zero. Then pure Fubini-Study metrics on $L_i$ are the same as model metrics on $L_i$, so that by multilinearity we may assume that $\phi_i = \phi_{\Ld_i}$, where $(\X_i, \Ld_i)$ is a nef model of $(X, L_i)$. Up to passing to a higher model, we may assume $\X_1=...=\X_n$. Then the Monge-Ampère measure has the following explicit description \cite{CL}:
$$\MA(\phi_1,...,\phi_n) = \sum_{E} b_E (\Ld_1 \cdot .... \cdot \Ld_n \cdot E) \delta_{v_E},$$
where the sum ranges over the irreducible components $E$ of $\X_0$, $b_E = \ord_E(t)$ and $\delta_{v_E}$ is the Dirac mass at the associated divisorial point $v_E = b_E^{-1} \ord_E$.
\end{ex}
In general, the mixed Monge-Ampère measure satisfies the following basic properties:
\begin{prop}{\cite[prop. 8.3]{BE}}
Let $\phi_i \in \DFS(L_i)$ for $i=1,...,n$.
\begin{itemize}
\item The pairing $(\phi_1,...,\phi_n) \mapsto \MA(\phi_1,...,\phi_n)$ is symmetric and multilinear;
\item if $\phi_i \in \FS(L_i)$ for all $i$, then $\MA(\phi_1,...,\phi_n)$ is a positive Radon measure;
\item the total mass $\int_{X^{\an}} \MA(\phi_1,...,\phi_n) = (L_1 \cdot ... \cdot L_n)$;
\item if $L_0 = L_1 = \gO_X$ and $\phi_0, \phi_1 \in \DFS(X)$, then: 
$$\int_X \phi_0 \MA(\phi_1,...,\phi_n) = \int_X \phi_1 \MA(\phi_0, \phi_2,...,\phi_n).$$
\end{itemize}
\end{prop}
The Monge-Ampère measure can then be extended by density to more general metrics:
\begin{theo}{\cite[thm. 8.4]{BE}} 
\\Let $K$ be a complete valued field, $X/K$ an $n$-dimensional projective scheme and $L_1$,...,$L_n$ line bundles on $X$. Then the Monge-Ampère operator:
$$(\phi_1,..., \phi_n) \mapsto dd^c \phi_1 \wedge ... \wedge dd^c \phi_n$$
admits a unique extension to continuous psh metrics on the $L_i$. The extension is furthermore continuous with respect to the topology of uniform convergence and the weak topology of Radon measures.
\end{theo}
\begin{rem}
While the class $\mathcal{E}^1(X, L)$ of finite-energy psh metrics on $L$ is defined over any non-archimedean field $K$, it is unclear in general how to extend the Monge-Ampère operator on the latter. Note however that in the case where $K$ is discretely-valued of characteristic zero, this extension was constructed in \cite{BFJ2}, and the trivially-valued case was handled in \cite{BJv4}.
\end{rem}
\begin{rem} The Chambert-Loir - Ducros \cite{CLD} approach to pluripotential theory on Berkovich spaces makes sense of the curvature current $dd^c \phi$ of a continuous psh metric, and its wedge products, in a spirit close to the work of Bedford-Taylor in the complex case. Notably, \cite[thm. 6.9.3]{CLD} states that the wedge product $dd^c\phi_1 \wedge ...\wedge dd^c \phi_n$ coincides with the mixed Monge-Ampère measure as described in example \ref{ex MA model} when $\phi_1$,..., $\phi_n$ are psh model metrics, whence the notation.
\end{rem}
\subsection{The Monge-Ampère equation}
Let $X$ be a smooth $n$-dimensional complex manifold, and $L$ an ample line bundle on $X$. Then Yau's celebrated solution to the Calabi conjecture asserts the following:
\begin{theo}{\cite{Y}} Let $\mu$ be a smooth volume form on $X$, normalized to have total mass $1$. Then there exists a unique (up to an additive constant) smooth, positive definite metric $\phi$ on $L$ such that:
$$(dd^c \phi)^n = (L^n) \mu.$$
\end{theo}
The main motivation for this theorem was the case where $X$ is a Calabi-Yau manifold, and $\mu = i^{n^2} \Omega \wedge \bar{\Omega}$ is the square-norm of a nowhere-vanishing holomorphic $n$-form on $X$: in that case, the curvature form $\omega = dd^c \phi$ is a smooth Kähler Ricci-flat metric on $X$.
\\Throughout the years, various generalizations of the above theorem in a more singular setting have appeared in the literature: let us simply mention Kołodziej's result \cite{Kolo}, that states that under the same assumptions on $(X, L)$, then the statement of the above theorem holds for a much wider range of probability measures on $X$ (for instance, measures $\mu$ with $L^p$-density for some $p>1$), when we don't require for the solution to be smooth - here the solution is a continuous psh metric $\phi \in \CPSH(X, L)$, and the equality $(dd^c \phi)^n = \mu$ is understood in the sense of Bedford-Taylor.
\\We now let $K= k((t))$ be a discretely-valued field of equicharacteristic zero. The following result can be understood as an analog of Kołodziej's result over $K$:
\begin{theo}{(\cite[thm. A]{BFJ1}, \cite{BGJKM})} \label{theo sol MA}
\\Let $(X, L)$ be a smooth polarized variety over $K$. Let $\mu$ be a probability measure on $X^{\an}$, supported on the skeleton of an snc $R$-model of $X$. Then there exists a unique (up to an additive constant) continuous, semi-positive metric $\phi$ on $L$ satisying the non-archimedean Monge-Ampère equation:
$$\MA(\phi) = \mu.$$
\end{theo}
The above theorem was proved when $X$ is defined over the function field of a curve over $k$ in \cite{BFJ1}, and then extended to varieties over non-archimedean fields of residual characteristic zero in \cite{BGJKM}.
\subsection{Family of Monge-Ampère measures} \label{sec MA family}
Let $A$ be an integral Banach ring, and $X/A$ a projective scheme. Since we may view $X^{\an}$ as the family of analytic spaces $\{ X^{\an}_{\mathscr{H}_x}
\}_{x \in \M(A)}$, which are analytic spaces over fields, one may define a Monge-Ampère operator on each fiber.
\begin{defn} Let $A$ be an integral Banach ring, $X/A$ an $n$-dimensional projective scheme, and $L_1,...,L_n$ semi-ample line bundles on $X$. If $(\phi_1,..., \phi_n)$ is a tuple of continuous psh metrics on the $L_i$, we define the associated family of Monge-Ampère measures as follows. For $x \in \M(A)$, write $X_x = \pi^{-1}(x) \simeq X^{\an}_{\mathscr{H}(x)}$, and $\iota_x : X_x \hookrightarrow X^{\an}$ the inclusion. Then we set:
$$(\MA(\phi_1,..., \phi_n))_x := (\iota_x)_* \big(\MA((\phi_1)_{|X_x},...,(\phi_n)_{|X_x})\big).$$
This is a family of measures on $X^{\an}$ parametrized by $\M(A)$.
\end{defn}
A natural question is to ask whether or not this family of measures is continuous, at least in the weak sense:
\begin{qu} Let $X$ be a flat, projective $A$-scheme, $L_1,...,L_n$ as above, and let $\phi_i \in \CPSH(L_i)$ for each $i$. Is it true that for $\psi \in \mathcal{C}^0(X^{\an})$, the function:
$$x \mapsto \int_{X^{\an}} \psi (\MA(\phi_1,..., \phi_n))_x$$
is continuous on $\M(A)$ ?
\end{qu}
The flatness assumption on $X$ is necessary to ensure that the total mass of the measure:
$$(\MA(\phi_1,..., \phi_n))_x (X^{\an}) = ((L_1)_{| X_{\mathscr{H}(x)}} \cdot ... \cdot (L_n)_{| X_{\mathscr{H}(x)}})$$
is indeed independent of $x \in \M(A)$.
\\By density of $\DFS(X) \subset \mathcal{C}^0(X^{\an})$ and $\FS(L_i) \subset \CPSH(L_i)$, and by using the following very general Chern-Levine-Nirenberg estimate, it is enough to prove the above statement for $\psi \in \DFS(X)$ and $\phi_i \in \FS(L_i)$. 
\begin{lem}{\cite[lem. 8.6]{BE}} \label{lem CLN}Let $K$ be a non-archimedean field, and $X/K$. 
Let $L_0,...,L_n$ be line bundles on $X$, and $\phi_i, \phi'_i \in \FS(L_i)$ for each $i$. Then:
$$\bigg\lvert \int_{X^{\an}} (\phi_0-\phi'_0) \MA(\phi_1,...,\phi_n) - \int_{X^{\an}} (\phi_0-\phi'_0) \MA(\phi'_1,...,\phi'_n) \bigg\rvert \le C \sum_{i=1}^n \sup_{X^{\an}} \lvert \phi_i - \phi'_i \rvert.$$
\end{lem}
While the answer to the above question seems unclear without further assumptions over the Banach ring $A$, we are able to provide an affirmative answer in the case of hybrid spaces in the next section. We also expect the statement to hold when $A$ is the ring of integers of a number field, we mention for instance the work \cite{PoiDyn} which proves a special case of the above statement for $A= \Z$ and provides further applications.
\subsection{Hybrid metrics and admissible data}
Throughout this section, we let $X \xrightarrow{\pi} \D^*$ be a smooth degeneration of complex manifolds, relatively polarized by an ample line bundle $L$; and write $X^{\hyb}$ the associated hybrid space. In this set-up, after fixing a reference Fubini-Study metric on $(X,L)$, Favre \cite{Fav} defines a certain class of \emph{model functions} $\phi_{\mathcal{F}} : X^{\hyb} \fl \R$ associated to admissible data on $X$; and uniform functions which are uniform limits of model functions. We will explain how those are nothing but hybrid Fubini-Study and hybrid cpsh metrics on $L$. In particular, the following result is a mere reformulation of \cite[thm. 4.2]{Fav}:
\begin{theo} \label{theo cont MA}
Let $(X,L)$ be as above, and $\phi \in \CPSH(X^{\hyb}, L^{\hyb})$. Then the associated family of Monge-Ampère measures is continuous on $X^{\hyb}$ in the weak sense: for any $f \in \mathcal{C}^0(X^{\hyb})$, we have:
$$\int_{X_t} f (dd^c \phi_t)^n \xrightarrow[t \fl 0]{} \int_{X^{\an}} f \MA(\phi_0).$$
\end{theo}
Recall that $\DFS(X^{\hyb})$ is dense in $\mathcal{C}^0(X^{\hyb})$, while $\FS^{\tau}(L^{\hyb})$ is dense in $\CPSH(L^{\hyb})$ for the topology of uniform convergence. As a result, by lemma \ref{lem CLN}, it is enough to prove the above convergence with $f \in \DFS(X^{\hyb})$, $\phi \in \FS^{\tau}(L^{\hyb})$.
\\As in \cite{Fav}, we fix an snc model $\X$ of $X$ such that $L$ has an ample model $\Ld$ on $\X$, so that we get a relative embedding $\iota : \X \hookrightarrow \C \CP^N \times \D$ by sections of $m\Ld$ for $m \ge 1$, and write $\phi_{\rf} = m^{-1} \iota^* \phi_{FS}$. By example \ref{ex homo} and pullback, we have $\phi_{\rf} \in \CPSH(L^{\hyb})$.
\\Then a regular admissible datum $\mathcal{F} = \{ \X', d, D,  s_1,..., s_l \}$ consists of the following: $p: \X' \fl \X$ is an snc model dominating $\X$, $D \in \Div_0(\X')$ is a vertical divisor on $\X'$, and $(s_1,...,s_l)$ is a tuple of sections of $p^*(\Ld^d)( D)$ without common zeroes.
An regular admissible datum defines a model function $\phi_{\mathcal{F}}: X^{\hyb} \fl \R$ as follows:
$$\phi_{\mathcal{F}} = \max_{i=1,...,l} \log \lVert s_i \rVert_{\phi_{\rf}} = \max_{i=1,...,l} (\phi_i - \phi_{\rf})$$
with $\phi_i = d^{-1} \log \lvert s_i \rvert$.
\\Thus, we naturally define:
$$\psi_{\mathcal{F}}=\phi_{\rf} + \phi_{\mathcal{F}} = \max_{i \le l} \phi_i,$$
which is a Fubini-Study metric on $L^{\hyb}$, since the $s_i$'s have no common zeroes.
\\Conversely, any pure Fubini-Study metric on $L^{\hyb}$, i.e. a metric of the form:
$$\phi = d^{-1} \max_{i \le l} \log \lvert s_i \rvert$$
for $d \ge 1$ and $s_1,...,s_l \in H^0(X, dL)$ without common zeroes, defines a regular admissible datum. Indeed, we may extend $s_1,...,s_l$ to meromorphic sections of $d\Ld$ on $\X$, and set:
$$\mathscr{I} = < s_1,..., s_l >,$$
which is a vertical fractional ideal sheaf on $\X$. It follows from \cite[prop. 2.2]{Fav} that if $p: \X' \fl \X$ is a log-resolution of $\X$, this yields an admissible datum:
$$\mathcal{F} = \{ \X', d , D, p^*s_1,...,p^*s_l \},$$
where $D \in \Div_0(\X')$ is such that $p^*(\Ld^d \otimes \mathscr{I}) = (p^*\Ld)^d \otimes \gO_{\X'}(D)$, and it follows from the previous computation that the associated model function $\phi_{\mathcal{F}}$ satisfies:
$$\phi_{\mathcal{F}} = \max_{i=1,...,l} \log \lVert s_i \rVert_{\phi_{\rf}} = \phi -\phi_{\rf},$$
with $\phi = d^{-1} \max_{i \le l} \log \lvert s_i \rvert$. As a result, model functions on $X^{\hyb}$ in the sense of \cite{Fav} are precisely the continuous functions $\psi : X^{\hyb} \fl \R$ such that $(\phi_{\rf} + \psi) \in \FS^{\tau}(L^{\hyb})$ is a pure Fubini-Study metric, so that by remark \ref{rem pure}, uniform functions are those such that $(\phi_{\rf} + \psi) \in \CPSH(L^{\hyb})$. This proves that the statement of theorem \ref{theo cont MA} is equivalent to the statement \cite[thm. 4.2]{Fav}.
\section{Degenerations of canonically polarized manifolds} \label{sec KE neg}
\subsection{The set-up} \label{sec setup neg}
In this section, $X \xrightarrow{\pi} \D^*$ will denote a meromorphic degeneration of canonically polarized manifolds, and we will write $L = K_{X/\D^*}$ the polarization. It follows from the seminal work of Aubin and Yau (\cite{Au}, \cite{Y}) that every fiber $X_t$ admits a unique negatively curved Kähler-Einstein metric $\omega_t \in -c_1(X_t)$, satisfying the equation $\Ric(\omega_t) = - \omega_t$. The Kähler form $\omega_t$ can be written as the curvature form of a smooth Hermitian metric $\phi_t$ on $L_t = K_{X_t}$, i.e.:
$$\omega_t = dd^c \phi_t.$$
The metric $\phi_t$ is unique up to addition of a constant. In this situation, the family of Hermitian metrics $(\phi_t)_{t \in \D^*}$ turns out to also have plurisubharmonic variation in the horizontal direction by the work of Schumacher \cite{Schu}, and has logarithmic growth at $t=0$ by \cite[thm. 3]{Schu}. Thus, the family of metrics $(\phi_t)_{t \in \D^*}$ induces an element $\phi \in \CPSH(X, L)$, and it is a natural question to try and determine the non-archimedean limit $\phi^{\NA}$ of this family, as it provides a non-archimedean analog of the Kähler-Einstein metric.
\\As we will explain more thoroughly in section \ref{sec CV alg neg}, after a finite base change $t \mapsto t^d$ on the punctured disk, that we omit from notation, the family $X/ \D^*$ admits a canonical model $\X_c / \D$ such that the canonical bundle $K_{\X_c / \D}$ is relatively ample; the model $\X_c / \D$ is furthermore unique for this property (although it is more singular than an snc model). This in turn yields a canonical model metric $\phi_{K_{\X_c /R}} \in \CPSH(X^{\an}, K_X^{\an})$, which we will prove to be the non-archimedean Kähler-Einstein metric - the following statement is Theorem B from the introduction:
\begin{theo} \label{THEO KE NEG}
Let $X \xrightarrow{\pi} \D^*$ be a degeneration of canonically polarized manifolds, $L = K_X$, and let $\phi_{\KE} \in \CPSH(X, L)$ be the family of Kähler-Einstein metrics. We assume that the family $X$ has semi-stable reduction over $\D$. Then the metric on $L^{\hyb}$ defined by:
$$\phi_{| X} = \phi_{KE},$$
$$\phi_0 = \phi_{K_{\X_c/R}}$$
is continuous and semipositive, i.e. $\phi \in \CPSH(X^{\hyb}, L^{\hyb})$.
\end{theo}
If the family of Kähler-Einstein metrics $\phi$ were to extend as a bounded metric $\phi \in \PSH(\X_c, K_{\X_c / \D})$, then it would follow from example \ref{ex bdd model} that $\phi^{\NA} = \phi_{K_{\X_c/R}}$. We will see that this is however not the case, but the singularities of $\phi_{\KE}$ along the special fiber of the canonical model are mild enough for the result to still hold - they are milder than any log poles.
\\Using theorem \ref{theo cont MA}, we have the following immediate consequence:
\begin{cor}
Let $X \fl \D^*$ be a meromorphic degeneration of $n$-dimensional canonically polarized manifolds, and let $\omega_t \in -c_1(X_t)$ be the unique Kähler-Einstein metric with negative curvature on $X_t$. 
\\The Kähler-Einstein measures $\mu_t = \omega_t^n$ converge weakly on $X^{\hyb}$ to $\mu_0 = \MA( \phi_{K_{\X_c}})$, the non-archimedean Monge-Ampère measure of the canonical model metric. 
\\More explicitly, writing $\X_{c, 0} = \sum_{i \in I} b_i D_i$ as the sum of its irreducible components, the measure $\mu_0$  is a sum of Dirac masses supported at the divisorial valuations $v_{D_i}$, and we have:
$$ \mu_0 = \sum_{i \in I}  b_i\big(( K_{\X_c})^n \cdot D_i \big) \delta_{v_i}.$$
\end{cor}
This was proved in \cite[thm. A]{PS} in a direct way, without using thm. \ref{theo cont MA}.
\subsection{The canonical model} \label{sec CV alg neg}
Let $X \xrightarrow{\pi} \D^*$ be a degeneration of smooth, canonically polarized varieties. In this case, the Minimal Model Program provides us with a \emph{unique} canonical model $\X_c$ of $X$ over the disk, at the cost of going out of the class of simple normal crossing models, and allowing some slightly worse singularities. The appropriate class of varieties for the central fiber is a higher-dimensional analog of the stable curves, the correct notion being that of \emph{semi-log canonical models}.
\\If $\X$ is a normal model of $X$, saying that $\X_0$ is semi-log canonical (see for instance \cite{Kollar2013}) is a condition on the singularities of the normalisation of $\X_0$, which can be seen as a mild generalization of the simple normal crossing condition; in particular we require $\X_0$ to be reduced and simple normal crossing in codimension 1. More precisely, the normalization morphism $\nu : \X^{\nu}_0 \fl \X_0$ is required to yield a disjoint union $\X_0^{\nu} = \sqcup_{i \in I} (\tilde{D}_i , C_i)$ of log canonical pairs, $C_i$ being the restriction of the \emph{conductor} $C$ of $\nu$ to $\tilde{D}_i$. This is a Weil divisor on $\X_0^{\nu}$, whose support is precisely the locus where the normalization $\nu$ fails to be an isomorphism, and which is simply given here by the inverse image by $\nu$ of the codimension one nodes of $\X_0$. It furthermore satisfies the formula : $\nu^*K_{\X_0} = K_{\X^{\nu}_0} + C$ (note that the canonical divisor of a semi-log canonical variety is assumed to be $\Q$-Cartier).
\\A semi-log canonical model (or \emph{stable variety}) is now by definition a proper semi-log canonical variety, with ample canonical divisor. For instance, one-dimensional semi-log canonical models are nothing but Deligne-Mumford's stable curves. 
\\The compactness theorem for moduli of stable varieties of higher dimension is now as follows:
\begin{theo}{(\cite{BCHM}, \cite{KNX}).}
\\Let $X \fl \D^*$ be an algebraic degeneration of canonically polarized manifolds. 
\\There exists (possibly after a finite base change) a unique \emph{canonical} model $\X_c$ of $X$ over the disk, satisfying the following properties:\\
i) the total space $\X_c$ has at worst canonical singularities, while the central fiber $\X_{c,0}$ is reduced and has at worst semi-log canonical singularities;\\
ii) the relative canonical divisor $K_{\X_c/ \D}$ is relatively ample.
\end{theo}
The canonical model is constructed as follows: by the semi-stable reduction theorem \cite{KK}, there exists a finite base change $t = (t')^d$ on the punctured disk such that the family $X' = X \times_{\D^*_t} \D^*_{t'}$ admits a semi-stable model over $\D^*$, i.e. an snc model with reduced special fiber. Omitting the base change from notation and starting from a semi-stable model $\X / \D$ of $X$, we set $\X_c = \Proj_{\D} \bigoplus_{m \ge 0} H^0 (\X, mK_{\X/ \D})$, the main difficulty being to prove finite-generatedness of the relative canonical algebra. This is established in \cite{BCHM} when $X / \D^*$ is defined over an algebraic curve, and extended to families over the disk in \cite{KNX}. The uniqueness of the canonical model is now a straightforward consequence of the birational invariance of the relative canonical ring $R=\bigoplus_{m \ge 0} H^0 (\X, mK_{\X/ \D})$ - as the notation suggests, $R$ does not depend on the choice of the model $\X$.
\begin{rem}
If $\X$ is any semi-stable model of $X$, then the natural rational map $h : \X \dashrightarrow \X_c$ is in fact a \emph{rational contraction} - this means that its inverse does not contract any divisors.
\end{rem}
\subsection{Metric convergence and proof of theorem \ref{THEO KE NEG}} \label{CV KE neg}
The complete understanding of the Gromov-Hausdorff convergence of the fibers $(X_t, g_t)$, is due to J. Song \cite{Song} (whose results were further improved recently in \cite{SSW}). The crucial first step, is to show that there exists on the central fiber $\X_{c,0} = \sum_{i \in I} D_i$ of the canonical model of $X$ a unique Kähler-Einstein current $\omega_{KE}$, and to derive some geometric estimates on the singularities of this current. The current $\omega_{KE}$ on the stable variety $\X_{c,0}$ was first constructed by Berman-Guenancia \cite{BG} using a variational method, while it is reconstructed in \cite{Song} using the techniques of \cite{EGZ}, \cite{Kolo}, in order to obtain some stronger control on its singularities:
\begin{theo}{(\cite[thm. 1.1]{Song}).}
\\Let $\X_c \fl \D$ be the canonical model of $X$, with semi-log canonical central fiber $\X_{c,0}$.
\\There exists a unique Kähler current $\omega_{KE} \in - c_1(\X_{c,0})$ on $\X_{c,0}$, satisfying the following properties:
\\i) $\omega_{KE}$ is smooth and satisfies the Kähler-Einstein equation on the regular locus of $\X_{c,0}$;
\\ii) $\omega_{KE}$ has locally bounded potentials on the locus where $\X_{c,0}$ is log terminal;
\\iii) $\omega_{KE}^n$ does not charge mass on the singularities of $\X_{c,0}$, and $\int_{\X_{c,0}} \omega_{KE}^n = [K_{\X_{c,0}}]^n$.
\end{theo}
\begin{rem}
The fact that the above Kähler-Einstein current on $\X_{c,0}$ matches the one constructed in \cite{BG}, follows from the uniqueness statement in \cite[thm. A]{BG}. Moreover, the construction of \cite{BG} implies that $\int_{D_i} \omega_{KE}^n = (K_{\X_c}^n \cdot D_i)$. 
\\Indeed, if $\nu : \X^{\nu}_{c,0} \fl \X_{c,0}$ denotes the normalization morphism, where $\X_{c,0}^{\nu} = \sqcup_{i \in I} \tilde{D}_i$, then the Kähler-Einstein metric $\omega_{KE}$ is obtained by descending the (singular) Kähler-Einstein metrics $\omega_i \in c_1(K_{\tilde{D}_i} + C_i)$ on the log canonical pairs $(\tilde{D_i}, C_i)$, $C_i$ being the restriction of the conductor $C$ of $\nu$ to $\tilde{D}_i$.
\\Thus by construction, the mass $\int_{D_i} \omega_{KE}^n$ equals the intersection number $(K_{\tilde{D}_i} + C_i)^n = (\nu^* K_{\X_{c,0}})^n$, the last intersection number being computed on $\tilde{D}_i$. Applying the projection formula, this is equal to the intersection number $ K_{\X_{c,0}}^n \cdot D_i =K_{\X_c}^n \cdot D_i$, by adjunction and principality of $\X_{c,0}$.
\end{rem}
Let us now fix a $m>0$ such that $mK_{\X_c/\D}$ is relatively very ample, and a relative embedding $\iota : \X_c \hookrightarrow \CP^N \times \D$ of the canonical model inside projective space by sections of $m K_{\X_c/ \D}$. We let $\phi_{\FS}$ be the hybrid Bergman metric on $(\CP_K^{N})^{\hyb}$ from example \ref{ex homo}, and still write $\phi_{\FS}\in \CPSH(X^{\hyb}, L^{\hyb})$ its pullback to $X^{\hyb}$ via the embedding $\iota$. More explicitly, we have that
$$\phi_{\FS, t} =m^{-1} \iota_t^* \phi_{\FS},$$
where $\phi_{\FS}$ is the usual (Euclidean) Fubini-Study metric on $\C \CP^N$; while
$$\phi_{\FS, 0} = \phi_{K_{\X_c}/R},$$
since the model $(\X_c, K_{\X_c/R})$ is ample.
\\This allows us to write the Kähler-Einstein metric $\phi_{\KE, t} =\phi_{\FS, t} + \psi_t$, with $\psi_t \in \mathcal{C}^{\infty}(X_t)$. The potential $\psi_t$ is the unique solution of the Monge-Ampère equation:
$$(\omega_{\FS, t} + dd^c \psi_t)^n = e^{\psi_t} \omega_{\FS, t}^n,$$
with the normalization $\int_{X_t} e^{\psi_t} \omega_{\FS, t}^n = (K_{X_t})^n$. In order to derive uniform estimates for the family of potentials $(\psi_t)_{t \in \D^*}$, it is more convenient to work on a semi-stable model, as a result we perform an additional base change and consider a diagram of the form:
$$\begin{tikzpicture}
\matrix(m)[matrix of math nodes, row sep=3em,
    column sep=2.5em, text height=1.5ex, text depth=0.25ex]
{ \X & \X'_c& \X_c \\
      & \D & \D \\ };
	\path 
	(m-1-1) edge[->] node[auto] {$ p $} (m-1-2)
			edge[->] node[auto] {$ \pi' $} (m-2-2)
         (m-1-2) edge[->] node[auto] {} (m-1-3)
         (m-1-2) edge[->] node[auto] {} (m-2-2)
         (m-1-3) edge[->] node[auto] {} (m-2-3)
         (m-2-2) edge[->] node[auto] {$t \mapsto t^d$} (m-2-3);
\end{tikzpicture}$$
where $\X'_c$ is the base change of the canonical model $\X_c$ via $t \mapsto t^d$, and $\X$ is a semi-stable resolution of $\X'_c$. We write the special fiber $\X_0 = \sum_{i \in I} \tilde{D}_i + \sum_{j \in J} E_j$, where $\tilde{D}_i$ is the strict transform of $D_i \subset \X'_{c, 0} = \X_{c, 0}$ and the $E_j$'s are the exceptional divisors of $p$. For each $i \in I$, let $\phi_i$ be a psh metric on $\gO_{\X}(\tilde{D}_i)$ with divisorial singularities along $\tilde{D}_i$, i.e. $\phi_i = \log \lvert z_i \rvert +O(1)$ locally, where $z_i$ is a local equation for $\tilde{D}_i$. Similarly, we choose $\psi_j$ with divisorial singularities along $E_j$, so that
$$\sum_{i \in I} \phi_i + \sum_{j \in J} \psi_j = \log \lvert t \rvert + O(1).$$
In order to apply Cheeger-Colding theory to the Kähler-Einstein metrics on $X_t$, Song derives uniform estimates on volumes of small balls, which are obtained via comparison lemmas for volume forms. The estimate focuses on a strict transform $\tilde{D}_{i_0} \subset \X_0$, and shows that the potentials $\psi_t$ do not blow-up as we approach the interior of $\tilde{D}_{i_0}$:
\begin{prop}{(\cite[cor. 4.1, lem. 4.2]{Song})} \label{prop Song}
\\We have a uniform bound: $\sup_{X_t} \psi_t \le C$.
\\Moreover, letting $J' \subset J$  be the subset of exceptional divisors in $\X_0$ that meet $\tilde{D}_{i_0}$, the following holds: for any $\eps >0$, there exists a constant $C_{\eps} >0$ such that for all $t \neq 0$:
$$\psi_t \ge \eps \big( \sum_{\substack{i \in I \\ i \neq i_0}} \phi_i + \sum_{j \in J'} \psi_j \big) -C_{\eps}.$$
This implies smooth convergence of the Kähler-Einstein metrics $\omega_{t}$ to the current $\omega_{KE}$ on $\X_{c,0} \setminus \Sing(\X_{c,0})$ in the following sense:
\\for any point $p \in \X_{c,0} \setminus \Sing(\X_{c,0})$, and any choice of neighbourhood $\U$ of $p$ such that the $U_t = \U \cap X_t$ are all biholomorphic to $U_0$, and such that the $(\omega_{\FS,t})^n$'s are uniformly equivalent, the pulled-back $\psi_t$ converge in the $\mathcal{C}^{\infty}$-sense to $\psi_0$.
\end{prop}
Note that we have made a small abuse of notation, since the object in the right-hand side of the inequality is a metric and not a function.
\begin{rem}
Even if we will not need it here, one can show that the previous theorem combined with a uniform non-collapsing condition implies pointed Gromov-Hausdorff convergence of $X_t$ to a complete metric space, whose regular part (in the Cheeger-Colding sense) is precisely $(\X_{c,0} \backslash \Sing(\X_{c,0}), \omega_{KE})$.
\\We also point out that this holds for degeneration of canonically polarized manifolds over a higher-dimensional base by the results of \cite{SSW}, building on the semi-stable reduction theorem from \cite{AK}, \cite{ALT}.
\\The behaviour of the metrics in the region where the metric collapses is also well-understood, under the technical assumption that the canonical model is semi-stable, see \cite{Zh}.
\end{rem}
We are now ready to prove theorem \ref{THEO KE NEG}. We let $\phi_0 = \phi_{K_{\X_c/R}} \in \CPSH(X^{\an}, L^{\an})$, and $\phi_{KE} \in \CPSH(X, L)$ the family of Kähler-Einstein metrics. In order to prove that the hybrid metric $\phi$ defined by the statement of theorem \ref{THEO KE NEG} is continuous and semi-positive, it is enough to prove that it defines a continuous metric on $L^{\hyb}$, by prop. \ref{prop psh cont}.
\\Substracting the reference metric hybrid $\phi_{\FS}$, whose restriction to $X^{\an}$ is the model metric $\phi_{K_{\X_c/R}}$, it is enough to prove that the potential $\psi_t$'s converge to zero as $t \rightarrow 0$ in the hybrid topology. In other words, we need to prove that:
$$\big\lvert \frac{\psi_t}{\log \lvert t \rvert} \big\rvert \xrightarrow[t \rightarrow 0]{} 0,$$
which is an easy consequence of the estimates from prop. \ref{prop Song}. This concludes the proof of theorem \ref{THEO KE NEG}.
\section*{Appendix} \label{appendix psh}
In this appendix, we state and prove a regularization result for psh metrics in the complex analytic setting, that will be used in the proof of theorem \ref{theo psh hyb}. We expect that the statement below (thm. \ref{theo reg psh}) is well-known to experts, as well as the techniques we use in the proof - which are largely due to Demailly \cite{Dem}. Nevertheless, since we could not find the precise statement required in the literature, we include a proof.
\\We let $(Y, \omega)$ be a Kähler manifold, and assume given a proper holomorphic map $\pi : Y \fl \Omega$, where $\Omega$ is a bounded open subset of $\C$. We furthermore assume that there exists a $\pi$-relatively ample line bundle $L$  on $Y$ such that $\omega \in c_1(L)$.
\\Given a semi-positive metric $\phi \in \PSH(Y,L)$, we want to write it as a decreasing limit of a sequence $(\phi_j)_{j \in \N}$ of psh metrics on $L$ with analytic singularities along the multiplier ideal sheaf $\mathscr{I}_j := \mathscr{I}(j \phi)$. We recall the basic definitions:
\begin{defn} Let $\phi \in \PSH(Y, L)$ be a semi-positive metric on $L$. The multiplier ideal sheaf $\mathscr{I}(\phi)$ is the ideal generated by the germs of holomorphic functions $f$ such that $\lvert f \rvert^2 e^{-2\phi}$ is locally integrable on $Y$.
\end{defn}
Here locally integrable means locally integrable in any coordinate chart, we also abusively view $\phi$ as a psh function this way. 
\begin{defn} \label{def ansing} Let $\mathscr{J} \subset \gO_Y$ be a coherent ideal sheaf on $Y$, and $\phi \in \PSH(Y, L)$. We say that $\phi$ has analytic singularities along $\mathscr{J}$ if $\phi$ can be written locally as:
$$\phi = \log \big( \lvert f_1 \rvert^2 +...+ \lvert f_r \rvert^2 \big) + \chi,$$
where $(f_1,...,f_r)$ is a family of local generators of $\mathscr{J}$ and $\chi$ is a smooth function.
\end{defn}
Given a coherent ideal sheaf $\mathscr{J}$, one can always produce quasi-psh functions with analytic singularities along $\mathscr{J}$, using a partition of unity argument. 
\\The rest of this appendix will be devoted to the proof of the following:
\begin{theo} \label{theo reg psh} Let $\Omega \subset \C$ be a bounded open subset, $Y$ a smooth Kähler manifold together with a proper holomorphic map $\pi : Y \fl \Omega$, and let $L$ a relatively ample line bundle on $Y$. We let $\psi$ be a smooth Hermitian metric on $L$ whose curvature form $\omega= dd^c \psi$ is a Kähler metric on $Y$.
\\Let $\phi \in \PSH(Y, L)$, and write $\mathscr{I}_m := \mathscr{I}(m\phi)$ the multiplier ideal of $m \phi$, for $m \in \N$. 
Then for any  relatively compact, open subset $Y' \Subset Y$, there exists a sequence $(\phi_j)_{j \in \N} \in \PSH(Y', L)$ such that:
\begin{itemize}
\item the $\phi_j$ decrease pointwise to $\phi$ on $Y'$,
\item for all $j \in \N$, the psh metric $2^j \phi_j$ on $2^j L$ has analytic singularities of the form $\mathscr{I}_{2^j}$.
\end{itemize}
\end{theo}
Let us fix a psh exhaustion function $\eta : \Omega \fl \R$, i.e. such that the sublevel sets $\Omega_c := \{ \eta < c \}$ are relatively compact subsets of $\Omega$; note that the $Y_c := \{ \eta \circ \pi < c \}$ are also relatively compact in $Y$ and weakly pseudoconvex. Since the subset $Y' \subset Y$ is relatively compact, we have $Y' \subset Y_c$ for $c \gg 1$.
\begin{prop} \label{prop glob gen} For any $c \in \R$ such that $Y' \Subset Y_c \Subset Y$, there exists a $m_0 \gg 1$ such that for all $m \ge 1$, the sheaf $\gO_Y((m+m_0)L\otimes \mathscr{I}_m)$ is generated by its global sections on $Y_c$.
\end{prop}
\begin{proof} We argue as in \cite[lem. 5.6]{BBJ}, and write $n = \dim Y$.
\\We let $H$ be a relatively very ample line bundle on $Y$, and we choose $m_0 \gg 1$ such that the line bundle $A = m_0 L -K_Y -nH$ is relatively ample on $Y_c$.
\\By the relative Castelnuevo-Mumford regularity criterion (see \cite[lem. 1.4]{DEL}), the sheaf $\gO_Y((m+m_0)L\otimes \mathscr{I}_m)$ is $\pi$-globally generated on $Y_c$ as soon as:
$$R^q \pi_* (\big((m+m_0)L -qH\big) \otimes \mathscr{I}_m) =0$$
for $1 \le q \le n-1$, which holds by Nadel vanishing \cite{Nad}, \cite[thm. B.8]{BFJ2}.
\end{proof}
We will now regularize $\phi$ by a sequence of psh metrics with analytic singularities of the form $\mathscr{I}(m \phi)$, up to some controlled error term. We mostly follow the argument from the proof of \cite[thm. 8.1]{GZ}. For $m_0$ large enough so that prop. \ref{prop glob  gen} holds, write:
$$\psi_{m, m_0} = \big( m \phi + m_0\psi \big) \in \PSH(Y, (m+m_0)L),$$
we have that $\mathscr{I}_m = \mathscr{I}(\psi_{m, m_0}) = \mathscr{I}(m\phi)$ is the multiplier ideal of the psh metric $\psi_{m, m_0}$ on $(Y, L)$.
\\We are naturally led to introduce the Bergman metrics associated to the multiplier ideal $\mathscr{I}_m$; for $Y_c$ as in prop. \ref{prop glob gen}, we set $V_{m, m_0} := H^0(Y_c, (m+m_0)L \otimes \mathscr{I}_{m})$ and define $ \mathscr{H}_{m, m_0} \subset V_{m, m_0}$ as the following Hilbert space:
$$\mathscr{H}_{m, m_0} = \{ s \in V_{m, m_0} / \lVert s \rVert^2 := \int_{Y_c} \lvert s \rvert^2_{\psi_{m, m_0}} \omega^{n} < \infty \}.$$
For every $m$, we may choose a Hilbert basis $\mathscr{B}_{m, m_0} = (s_{m, m_0, l})_{l \in \N}$ of $ \mathscr{H}_{m, m_0}$, and we now set:
$$\phi_{m, m_0} = \frac{1}{2(m+m_0)} \log ( \sum_{l \in \N} \lvert s_{m, m_0, l} \rvert^2),$$
we have $\phi_{m, m_0} \in \PSH(Y_c, L)$.
\begin{prop} \label{prop som part}
For $q \in \N$, set:
$$\phi_{m, m_0, q} = \frac{1}{2(m+m_0)} \log ( \sum_{l \le q} \lvert s_{m, m_0, l} \rvert^2 ).$$
Then the $\phi_{m, m_0, q}$ converge uniformly to $\phi_{m, m_0}$ over $Y'$. Moreover, $\phi_{m, m_0}$ has analytic singularities of the form $\mathscr{I}_m$ over $Y'$.
\end{prop}
\begin{proof} 
We drop the $(m, m_0)$ subscript to alleviate notation. We choose $c' < c$ such that let $Y' \Subset Y_{c'} \Subset Y_c$, with the freedom to slightly decrease $c'$ throughout the steps of the proof. 
\\We set $\mathscr{J}_q = \mathscr{I}\big((s_l)_{l \le q}\big)$, and $\mathscr{J} = \cup_{q \ge 0} \mathscr{J}_q = \mathscr{I}\big((s_l)_{l \in \N}\big)$. Then we have $\mathscr{J}=(m+m_0)L\otimes \mathscr{I}(m\phi)$ over $Y'''$ by global generation, which is a coherent ideal sheaf on $Y_{c'}$ by Nadel's theorem \cite{Nad}. By the strong Noetherian property for coherent sheaves, the ascending chain $(\mathscr{J}_q)_{q \ge 0}$ of ideals is locally stationary, so that we have $\mathscr{J}_q \equiv \mathscr{J}$ for $q \gg1$ on $Y_{c'}$ after shrinking.
\\We now set:
$$F = \sum_{l \in \N} \lvert s_l \rvert^2,$$
and:
$$F_q =\sum_{l \le q} \lvert s_l \rvert^2.$$
We will prove that on $Y'$, there exists $C_q \ge 1$ such that:
$$F_q \le F \le C_q F_q,$$
and that $C_q \xrightarrow[q \fl \infty]{} 1$. We mostly mimic the argument from step 2 of the proof of \cite[thm. 2.2.1]{DPS}. Up to slightly decreasing $c'$, we may work locally, so that we assume that the $s_l$ are holomorphic functions. By the strong Noetherian property of coherent ideal sheaves, the sequence of ideal sheaves $\mathscr{K}_q$ on $Y_c \times Y_c$ generated by the $(s_l(z)\overline{s_l(\bar{w})})_{l \le q}$ is locally stationary, so that it is stationary at $\mathscr{K} = \cup_{q \ge 0} \mathscr{K}_q$ on $Y_{c'} \times Y_{c'}$ for $q \gg1$. From the bound:
$$\sum_{l \le q} \lvert s_l(z)\overline{s_l(\bar{w})} \rvert \le \bigg( \big( \sum_{l \in \N} \lvert s_l(z) \rvert^2\big) \cdot \big(\sum_{l \in \N} \lvert s_l(\bar{w}) \rvert^2\big) \bigg)^{1/2},$$
we infer that the series $\sum_{l \le q} s_l(z)\overline{s_l(\bar{w})}$ converges locally uniformly on $Y_{c'} \times Y_{c'}$, and thus by closedness of the space of sections of a coherent ideal sheaf, we get that the holomorphic function on $Y_{c'} \times Y_{c'}$:
$$\sum_{l \in \N} s_l(z)\overline{s_l(\bar{w})} \in \mathscr{K}.$$
Since $\mathscr{K} = \mathscr{K}_q$ for $q$ large enough, we get that:
$$\sum_{l \in \N} s_l(z)\overline{s_l(\bar{w})} \le C_q \sum_{l \le q} s_l(z)\overline{s_l(\bar{w})}$$
on $Y_{c'} \times Y_{c'}$ for some $C_q >0$, and thus:
$$F \le C_q F_q$$
over $Y_{c'}$ for $q$ large enough by restricting to the complex diagonal $z= \bar{w}$. 
\\Finally, let us set $\chi_q = \phi - \phi_q$: then we proved that for any $q \gg 1$, there exists $b_q >0$ such that $0 \le \chi_q \le b_q$ on $Y_{c'}$. As a result $\phi_q$ and $\phi$ have the same singularities, which are analytic singularities along $\mathscr{J}_q \equiv \mathscr{J}$ for $q \gg 1$. 
\\Moreover, since the sum converges locally uniformly, the $\chi_q$ are continuous and decrease to zero pointwise, so that from Dini's lemma the $\chi_q$'s decrease uniformly to zero on $Y'$, and $\phi_q$ converges uniformly to $\phi$ over $Y'$.
\end{proof}
We now want to prove that over $Y'$, the $\phi_{m, m_0}$ decrease with respect to $m$ (up to an error term) to our initial metric $\phi$. 
\\We start by proving that the $\phi_{m, m_0}$ converge pointwise to $\phi$ over $Y_c$. Writing $B_{m, m_0} \subset \mathscr{H}_{m, m_0}$ the unit ball with respect to the $L^2$-norm, we have:
$$\phi_{m, m_0} = \frac{1}{2(m+m_0)} \log \big(\sum_{l \ge 0} \lvert s_{m, m_0, l} \rvert^2 \big)= \frac{1}{2(m+m_0)} \sup_{s \in B_{m, m_0}} \log \lvert s \rvert^2,$$
since for $z \in Y_c$, the quantity $\sum_{l \ge 0} \lvert s_{m, m_0, l}(z) \rvert^2$ is the operator norm of the evaluation map $\ev_z : \mathscr{H}_{m, m_0} \fl L^{\otimes m+m_0}_z$.
\\Covering $Y_c$ by coordinates charts $(U_i)_{i \in I}$, we let $z \in U_i$ and $\rho>0$ such that $B(z, \rho) \subset U_i$. For $s \in \mathscr{H}_{m, m_0}$, since $\lvert s \rvert^2$ is subharmonic in $U_i$, we have:
$$\lvert s(z) \rvert^2 \le \frac{C}{\rho^{2(n+1)}} \int_{B(z,\rho)} \lvert s \rvert^2 \omega^{n+1} \le C' e^{2\sup_{B(z, \rho)} \psi_{m, m_0}} \int_{Y_c} \lvert s\rvert_{\psi_{m, m_0}}^2 \omega^{n+1}, $$
so that if $s \in B_{m, m_0}$, the bound:
$$\log \lvert s(z) \rvert^2 \le \sup_{B(z, \rho)} \psi_{m, m_0} +C$$
holds, hence $\phi_{m, m_0}(z) \le \sup_{B(z, \rho)} \phi + (m+m_0)^{-1}C$.
\\The converse inequality follows from the Ohsawa-Takegoshi theorem \cite{DemOT}: there exists $m_0 \gg1$ and a universal constant $C>0$ such that for all $m \in \N$ large enough and $z \in Y_c$, there exists $s \in \mathscr{H}_{m, m_0}$ such that:
$$\int_{Y_c} \lvert s \rvert^2 e^{-2\psi_{m, m_0}} \omega^{n+1} \le C \lvert s(z) \rvert^2 e^{-2\psi_{m, m_0}(z)},$$
so that if we choose the right-hand side to be equal to one, we get $s \in B_{m, m_0}$ such that:
$$\log \lvert s(z) \rvert \ge \psi_{m, m_0}(z)-C,$$
hence $\phi_{m, m_0}(z) \ge \phi(z) - (m+m_0)^{-1}C$, which proves pointwise convergence on $Y''$.
\\We now prove that the $\phi_{m, m_0}$ are almost subadditive. We let $s \in B_{m_1+m_2, m_0} \subset \mathscr{H}_{m_1+m_2, m_0}$, and set:
$$\mathscr{H}_{m_1, m_2, m_0} = \{ S \in H^0 \big(Y_c \times Y_c, p_1^*((m_1+\frac{m_0}{2})\Ld)(\mathscr{I}_{m_1}) \otimes p_2^*((m_2+\frac{m_0}{2})\Ld)(\mathscr{I}_{m_2}) \big) ; $$
$$\int_{Y_c \times Y_c} \lvert S(z_1, z_2) \rvert^2 e^{-2\psi_{m_1 + \frac{m_0}{2}}(z_1)-2\psi_{m_2 + \frac{m_0}{2}}(z_2)}(\omega_1 \otimes \omega_2)^{n} < \infty \},$$
where we have written $\omega_i = p_i^*\omega$.
\\By the Ohsawa-Takegoshi theorem, there exists $S \in \mathscr{H}_{m_1, m_2, m_0}$ with $L^2$-norm $\lVert S \rVert \le C$ for a universal constant $C$, such that $S_{| \de_{Y_c}} =s$, where $\de_{Y_c} \subset Y_c \times Y_c$ is the diagonal. To be more precise, we let: 
$$\xi : Y \times Y \fl \R \cup \{ - \infty \}$$
be a quasi-psh function with analytic singularities along $\mathscr{I}(\de_Y)$, and we may assume that $\xi \le 0$ on $Y_c$. Then in the notation of \cite[thm. 1.4]{DemOT} the measure $dV_{\de_Y, \omega}[\xi]$ is uniformly equivalent to $\omega^n$ on $\de_{Y_c}$, as $\de_Y$ is a local complete intersection, so that our estimate follows from the aforementioned theorem with $\delta =2$.
\\Since $\mathscr{H}_{m_1, m_2, m_0} = \mathscr{H}_{m_1, m_0/2} \hat{\otimes} \mathscr{H}_{m_2, m_0/2}$, the family $(s_{m_1, m_0/2,  l_1} \otimes s_{m_2, m_0/2, l_2})_{(l_1,l_2) \in \N^2}$ form a Hilbert basis of $\mathscr{H}_{m_1, m_2, m_0}$. We may write:
$$S(z_1, z_2) = \sum_{(l_1, l_2) \in \N^2} c_{l_1,l_2} s_{m_1, m_0/2,  l_1}(z_1) \otimes s_{m_2, m_0/2, l_2}(z_2),$$
with $\sum_{l_1, l_2} \lvert c_{l_1,l_2} \rvert^2 \le C$. Thus:
$$\lvert s(z) \rvert^2=\lvert S(z, z)\rvert^2 \le C \big(\sum_{l_1} \lvert s_{m_1, m_0/2,  l_1}(z) \rvert^2 \big) \times \big( \sum_{l_2} \lvert s_{m_2, m_0/2, l_2}(z) \rvert^2 \big),$$
so that:
$$\phi_{m_1+m_2, m_0} \le \frac{C}{m_1+m_2+m_0} + \frac{(m_1+m_0/2) \phi_{m_1, m_0/2}}{ m_1+m_2+ m_0} + \frac{(m_2+m_0/2) \phi_{m_2, m_0/2}}{ m_1+m_2+ m_0}.$$
Since $\phi - \psi$ is bounded from above over $Y_c$, we may assume without loss of generality that $\phi - \psi \le 0$, so that $(m+ \frac{m_0}{2})^{-1} \psi_{m, m_0/2} \le (m+m_0)^{-1} \psi_{m, m_0}$ and thus $\phi_{m_1, m_0/2} \le \phi_{m, m_0}$. This now implies that the sequence:
$$\phi_j = \phi_{2^j -m_0, m_0} + 2^{-j-2} C$$
is decreasing to $\phi$ over $Y_c$, and has the required singularities over $Y'$ by prop. \ref{prop som part}. This concludes the proof of theorem \ref{theo reg psh}.
\bibliographystyle{alpha}
\bibliography{biblio.bib}

\newcommand{\etalchar}[1]{$^{#1}$}
\begin{thebibliography}{KKMSD73}

\bibitem[AK00]{AK}
D.~Abramovich and K.~Karu.
\newblock Weak semistable reduction in characteristic 0.
\newblock {\em Invent. Math.}, 139 no. 2:241--273, 2000.

\bibitem[ALT18]{ALT}
K.~Adiprasito, G.~Liu, and M.~Temkin.
\newblock Semistable reduction in characteristic 0.
\newblock 2018.
\newblock arXiv:1810.03131.

\bibitem[Aub78]{Au}
T.~Aubin.
\newblock Equations du type {M}onge-{A}mpère sur les variétés kählériennes
  compactes.
\newblock {\em Bull. Sci. Math.}, (2) 102:63--95, 1978.

\bibitem[BBJ21]{BBJ}
R.~Berman, S.~Boucksom, and M.~Jonsson.
\newblock A variational approach to the {Y}au-{T}ian-{D}onaldson conjecture.
\newblock {\em J. Amer. Math. Soc.}, 34 (3):605--652, 2021.

\bibitem[BCHM10]{BCHM}
C.~Birkar, P.~Cascini, C.~Hacon, and J.~McKernan.
\newblock Existence of minimal models for varieties of log general type.
\newblock {\em J. Amer. Math. Soc.}, 23:405--468, 2010.

\bibitem[BE21]{BE}
S.~Boucksom and D.~Eriksson.
\newblock Spaces of norms, determinant of cohomology and {F}ekete points in
  non-archimedean geometry.
\newblock {\em Advances in Mathematics}, 378:107501, 2021.

\bibitem[BEGZ10]{BEGZ}
S.~Boucksom, P.~Eyssidieux, V.~Guedj, and A.~Zeriahi.
\newblock Monge-{A}mpère equation in big cohomology classes.
\newblock {\em Acta Math.}, 205:199--262, 2010.

\bibitem[Ber90]{Ber}
V.~Berkovich.
\newblock {\em Spectral theory and analytic geometry over non-Archimedean
  fields}.
\newblock Mathematical Surveys and Monographs, {A}{M}{S}, 1990.

\bibitem[Ber99]{Ber2}
V.~Berkovich.
\newblock Smooth p-adic analytic spaces are locally contractible.
\newblock {\em Invent. Math.}, 137:1--84, 1999.

\bibitem[Ber09]{Ber3}
Vladimir~G. Berkovich.
\newblock {\em A Non-Archimedean Interpretation of the Weight Zero Subspaces of
  Limit Mixed Hodge Structures}, pages 49--67.
\newblock Birkh{\"a}user Boston, Boston, 2009.

\bibitem[BFJ08]{BFJ0}
S.~Boucksom, C.~Favre, and M.~Jonsson.
\newblock Valuations and plurisubharmonic singularities.
\newblock {\em Publ. RIMS}, 44:449--494, 2008.

\bibitem[BFJ15]{BFJ1}
S.~Boucksom, C.~Favre, and M.~Jonsson.
\newblock Solution to a non-{A}rchimedean {M}onge-{A}mpère equation.
\newblock {\em J. Amer. Math. Soc.}, 28:617--667, 2015.

\bibitem[BFJ16]{BFJ2}
S.~Boucksom, C.~Favre, and M.~Jonsson.
\newblock Singular semipositive metrics in non-{A}rchimedean geometry.
\newblock {\em J. Algebraic Geom.}, 25:77--139, 2016.

\bibitem[BG14]{BG}
R.~Berman and H.~Guenancia.
\newblock Kähler-{E}instein metrics on stable varieties and log canonical
  pairs.
\newblock {\em Geometric and Functional Analysis}, 24 (6):1683--1730, 2014.

\bibitem[BGGJ{\etalchar{+}}20]{BGJKM}
J.~Burgos~Gil, W.~Gubler, P.~Jell, K.~Künnemann, and F.~Martin.
\newblock Differentiability of non-archimedean volumes and non-archimedean
  {M}onge-{A}mpère equations (with an appendix by {R}obert {L}azarsfeld).
\newblock {\em Algebraic {G}eometry}, 7 (2):113--152, 2020.

\bibitem[BGPS14]{BPS}
J.~Burgos~Gil, P.~Philippon, and M.~Sombra.
\newblock Arithmetic geometry of toric varieties. {M}etrics, measures and
  heights.
\newblock {\em Ast\'{e}risque}, (360):vi+222, 2014.

\bibitem[BJ17]{BJ}
S.~Boucksom and M.~Jonsson.
\newblock Tropical and non-{A}rchimedean limits of degenerating families of
  volume forms.
\newblock {\em Journal de l'Ecole polytechnique}, 4:87--139, 2017.

\bibitem[BJ18]{BJtriv}
S.~Boucksom and M.~Jonsson.
\newblock Singular semipositive metrics on line bundles on varieties over
  trivially valued fields.
\newblock 2018.
\newblock arXiv:1801.08229.

\bibitem[BJ22]{BJv4}
S.~Boucksom and M.~Jonsson.
\newblock Global pluripotential theory over a trivially valued field.
\newblock {\em Annales de la Facult\'e des sciences de Toulouse :
  Math\'ematiques}, Ser. 6, 31(3):647--836, 2022.

\bibitem[BR10]{BR}
M.~Baker and R.~Rummely.
\newblock {\em Potential theory and dynamics on the {B}erkovich projective
  line}, volume 159 of {\em Mathematical Surveys and Monographs}.
\newblock American Mathematical Society, 2010.

\bibitem[BT76]{BT}
E.~Bedford and A.~Taylor.
\newblock The {D}irichlet problem for a complex {M}onge-{A}mpere equation.
\newblock {\em Invent. Math.}, 37.1:1--44, 1976.

\bibitem[CL06]{CL}
A.~Chambert-Loir.
\newblock Mesures et équidistribution sur les espaces de {B}erkovich.
\newblock {\em J. für die {R}eine und {A}ngewandte {M}athematik},
  595:215--235, 2006.

\bibitem[CLD12]{CLD}
A.~Chambert-Loir and A.~Ducros.
\newblock Formes différentielles réelles et courants sur les espaces de
  {B}erkovich.
\newblock 2012.
\newblock arXiv:1204.6277.

\bibitem[Dar15]{Dar}
T.~Darvas.
\newblock The {M}abuchi geometry of finite energy classes.
\newblock {\em Advan. Math.}, 285:182–219, 2015.

\bibitem[DEL00]{DEL}
J.-P. Demailly, L.~Ein, and R.~Lazarsfeld.
\newblock A subadditivity property of multiplier ideals.
\newblock {\em Michigan Math. J.}, 48:137–156, 2000.

\bibitem[Dem92]{Dem}
J.-P. Demailly.
\newblock Regularization of closed positive currents and intersection theory.
\newblock {\em J. Algebraic Geom.}, 1 no. 3:361--409, 1992.

\bibitem[Dem12]{DemB}
J.-P. Demailly.
\newblock {\em Analytic methods in algebraic geometry}.
\newblock International Press Somerville, MA, 2012.

\bibitem[Dem15]{DemOT}
J.-P. Demailly.
\newblock {Extension of holomorphic functions defined on non reduced analytic
  subvarieties}.
\newblock In {\em {The legacy of Bernhard Riemann after one hundred and fifty
  years}}. December 2015.

\bibitem[DPS01]{DPS}
J.~Demailly, T.~Peternell, and M.~Schneider.
\newblock Pseudo-effective line bundles on compact {K}ähler manifolds.
\newblock {\em International Journal of Math.}, 6:689-- 741, 2001.

\bibitem[EGZ09]{EGZ}
P.~Eyssidieux, V.~Guedj, and A.~Zeriahi.
\newblock Singular {K}ähler-{E}instein metrics.
\newblock {\em J. Amer. Math. Soc.}, 22:607--639, 2009.

\bibitem[Fav20]{Fav}
C.~Favre.
\newblock Degeneration of endomorphisms of the complex projective space in the
  hybrid space.
\newblock {\em J. Inst. Math. Jussieu}, 19 (4):1141--1183, 2020.

\bibitem[GO22]{GO}
K.~Goto and Y.~Odaka.
\newblock Special {L}agrangian fibrations, {B}erkovich retraction, and
  crystallographic groups.
\newblock 2022.
\newblock arXiv:2206.14474.

\bibitem[Gub98]{Gub}
W.~Gubler.
\newblock Local heights of subvarieties over non-archimedean fields.
\newblock {\em J. Reine Angew. Math.}, 498:61--113, 1998.

\bibitem[GZ05]{GZ}
V.~Guedj and A.~Zeriahi.
\newblock Intrinsic capacities on compact {K}ähler manifolds.
\newblock {\em J. Geom. Anal.}, 15-4:607--639, 2005.

\bibitem[GZ17]{GZB}
V.~Guedj and A.~Zeriahi.
\newblock {\em Degenerate complex {M}onge-{A}mpère equations}.
\newblock EMS Tracts in Mathematics, 2017.

\bibitem[KKMSD73]{KK}
G.~Kempf, F.~Knudsen, D.~Mumford, and B.~Saint-Donat.
\newblock {\em Toroidal {E}mbeddings {I}}.
\newblock Lect. Notes in Math., Springer-Verlag, 1973.

\bibitem[Kli91]{Klim}
M.~Klimek.
\newblock {\em Pluripotential theory}.
\newblock London Math. Soc. Monogr., 1991.

\bibitem[KM98]{KM}
J.~Koll\'{a}r and S.~Mori.
\newblock {\em Birational geometry of algebraic varieties}, volume 134 of {\em
  Cambridge Tracts in Mathematics}.
\newblock Cambridge University Press, 1998.

\bibitem[KNX18]{KNX}
J.~Kollar, J.~Nicaise, and C.~Xu.
\newblock Semi-stable extensions over 1-dimensional bases.
\newblock {\em Acta Mathematica Sinica}, 34:1:103--113, 2018.

\bibitem[Kol98]{Kolo}
S.~Kolodziej.
\newblock The complex {M}onge-{A}mpère equation.
\newblock {\em Acta Math.}, 180:69--170, 1998.

\bibitem[Kol13]{Kollar2013}
J.~Koll{\'a}r.
\newblock {\em Singularities of the Minimal Model Program}, volume 200 of {\em
  Cambridge Tracts in Mathematics}.
\newblock Cambridge University Press, 2013.

\bibitem[KS06]{KontsevichSoibelman}
M.~Kontsevich and Y.~Soibelman.
\newblock {\em Affine Structures and Non-Archimedean Analytic Spaces}, pages
  321--385.
\newblock Birkh{\"a}user Boston, 2006.

\bibitem[Li22]{Li1}
Y.~Li.
\newblock Strominger-{Y}au-{Z}aslow conjecture for {C}alabi-{Y}au hypersurfaces
  in the {F}ermat family.
\newblock {\em Acta Math.}, 229:1--53, 2022.

\bibitem[Liu11]{Liu}
Y.~Liu.
\newblock A non-archimedean analogue of {C}alabi-{Y}au theorem for totally
  degenerate abelian varieties.
\newblock {\em J. Differential Geom.}, 89:87--110, 2011.

\bibitem[LP20]{LP}
T.~Lemanissier and J.~Poineau.
\newblock Espaces de {B}erkovich sur {Z} : catégorie, topologie, cohomologie.
\newblock 2020.

\bibitem[MN15]{MN}
M.~Musțăta and J.~Nicaise.
\newblock Weight functions on non-archimedean analytic spaces and the
  {K}ontsevich-{S}oibelman skeleton.
\newblock {\em Algebraic Geom. 2}, no. 3:365--404, 2015.

\bibitem[MS84]{MS}
J.~Morgan and P.~Shalen.
\newblock Valuations, trees, and degenerations of hyperbolic structures {I}.
\newblock {\em Ann. of Math.}, (2) 120:401--476, 1984.

\bibitem[Nad89]{Nad}
A.~Nadel.
\newblock Multiplier ideal sheaves and {K}ähler-{E}instein metrics of positive
  scalar curvature.
\newblock {\em Annals of Math.}, 132:549--596, 1989.

\bibitem[NX16]{NX}
J.~Nicaise and C.~Xu.
\newblock The essential skeleton of a degeneration of algebraic varieties.
\newblock {\em Amer. Math. J.}, 138(6):1645--1667, 2016.

\bibitem[NXY19]{NXY}
J.~Nicaise, C.~Xu, and T.~Y. Yu.
\newblock The non-archimedean {SYZ} fibration.
\newblock {\em Compositio Mathematica}, 155(5):953--972, 2019.

\bibitem[Poi10]{PoiZ}
J.~Poineau.
\newblock La droite de {B}erkovich sur {Z}.
\newblock {\em Astérisque}, 334, 2010.

\bibitem[Poi13]{Poi}
J.~Poineau.
\newblock Les espaces de {B}erkovich sont angéliques.
\newblock {\em Bull. de la Soc. Math. de France}, 141 no. 2:267--297, 2013.

\bibitem[Poi22]{PoiDyn}
J.~Poineau.
\newblock Dynamique analytique sur {Z}. {I} : {M}esures d'équilibre sur une
  droite projective relative.
\newblock 2022.
\newblock arXiv:2201.08480.

\bibitem[PS22a]{PS}
L.~Pille-Schneider.
\newblock Hybrid convergence of {K}ähler–{E}instein measures.
\newblock {\em Ann. de l'Institut Fourier}, 72 (2):587--615, 2022.

\bibitem[PS22b]{PS22}
L.~Pille-Schneider.
\newblock Hybrid toric varieties and the non-archimedean {S}{Y}{Z} fibration on
  {C}alabi-{Y}au hypersurfaces.
\newblock 2022.
\newblock arXiv:2210.05578.

\bibitem[Reb21]{Reb}
R.~Reboulet.
\newblock The space of finite-energy metrics over a degeneration.
\newblock 2021.
\newblock arxiv:2107.04841.

\bibitem[Sch12]{Schu}
G.~Schumacher.
\newblock Positivity of relative canonical bundles and applications.
\newblock {\em Invent. Math.}, 190 (1):1--56, 2012.

\bibitem[Shi20a]{Shiv1}
S.~Shivaprasad.
\newblock Convergence of {B}ergman measures towards the {Z}hang measure.
\newblock 2020.
\newblock arXiv:2005.05753.

\bibitem[Shi20b]{Shiv2}
S.~Shivaprasad.
\newblock Convergence of {N}arasimhan-{S}imha measures on degenerating families
  of {R}iemann surfaces.
\newblock 2020.
\newblock arXiv:2011.14471.

\bibitem[Son17]{Song}
J.~Song.
\newblock Degeneration of {K}ähler-{E}instein manifolds of negative scalar
  curvature.
\newblock 2017.
\newblock arxiv.:1706.01518.

\bibitem[SSW20]{SSW}
J.~Song, J.~Sturm, and X.~Wang.
\newblock Riemannian geometry of {K}ähler-{E}instein currents {I}{I}{I}:
  compactness of {K}ähler-{Einstein} manifolds of negative scalar curvature.
\newblock 2020.
\newblock arXiv:2003.04709.

\bibitem[Thu05]{Thu2}
A.~Thuillier.
\newblock {\em {Th{\'e}orie du potentiel sur les courbes en g{\'e}om{\'e}trie
  analytique non archim{\'e}dienne. Applications {\`a} la th{\'e}orie
  d'Arakelov}}.
\newblock Phd {T}hesis, {Universit{\'e} Rennes 1}, 2005.

\bibitem[Thu07]{Th}
A.~Thuillier.
\newblock Géométrie toroïdale et géométrie analytique non archimédienne.
  {A}pplication au type d’homotopie de certains schémas formels.
\newblock {\em Manuscr. Math. 123}, 4:381--451, 2007.

\bibitem[Tia90]{Tian}
G.~Tian.
\newblock On a set of polarized {K}ähler metrics on algebraic manifolds.
\newblock {\em J. Differential Geom.}, 32(1):99--130, 1990.

\bibitem[Vil21]{Vil}
C.~Vilsmeier.
\newblock A comparison of the real and non-archimedean {M}onge–{A}mpère
  operator.
\newblock {\em Math. Zeitschrift}, 297:633–668, 2021.

\bibitem[Yau78]{Y}
S.~Yau.
\newblock On the {R}icci curvature of a compact {K}ähler manifold and complex
  {M}onge-{A}mpère equation {I}.
\newblock {\em Comm. Pure Appl. Math.}, 31:339--411, 1978.

\bibitem[Zha95]{Zha}
S.~Zhang.
\newblock Small points and adelic metrics.
\newblock {\em J. Algebr. Geom.}, 4:281--300, 1995.

\bibitem[Zha15]{Zh}
Y.~Zhang.
\newblock Collapsing of negative {K}ähler-{E}instein metrics.
\newblock {\em Math. Res. Let.}, 22(6), 2015.

\end{thebibliography}
\end{document}